\documentclass[10pt]{article} 

\usepackage{amsmath}
\usepackage{latexsym}
\usepackage{amssymb}

\usepackage{amsthm}

\usepackage{hyperref}
\numberwithin{equation}{section}

\font\tengoth=eufm10 at 10pt
\font\sevengoth=eufm7 at 6pt
\newfam\gothfam
\textfont\gothfam=\tengoth
\scriptfont\gothfam=\sevengoth

\newcommand{\mlabel}[1]{\marginpar{#1}\label{#1}}

\newcommand{\g}{{\mathfrak g}}

\newcommand{\fa}{{\mathfrak a}}

\newcommand{\ff}{{\mathfrak f}}
\newcommand{\fg}{{\mathfrak g}}
\newcommand{\fh}{{\mathfrak h}}

\newcommand{\fk}{{\mathfrak k}}

\newcommand{\fm}{{\mathfrak m}}
\newcommand{\fn}{{\mathfrak n}}

\newcommand{\fq}{{\mathfrak q}}
\newcommand{\fp}{{\mathfrak p}}

\newcommand{\fu}{{\mathfrak u}}

\newcommand{\fz}{{\mathfrak z}}

\newcommand{\1}{\mathbf{1}}
\newcommand{\0}{{\bf 0}}

\newcommand{\cA}{\mathcal{A}}

\newcommand{\cD}{\mathcal{D}}
\newcommand{\cE}{\mathcal{E}}
\newcommand{\cF}{\mathcal{F}}

\newcommand{\cH}{\mathcal{H}}
\newcommand{\cK}{\mathcal{K}}
\newcommand{\cL}{\mathcal{L}}
\newcommand{\cM}{\mathcal{M}}

\newcommand{\cO}{\mathcal{O}}

\newcommand{\cR}{\mathcal{R}}
\newcommand{\cS}{\mathcal{S}}

\newcommand{\cU}{\mathcal{U}}

\newcommand{\cW}{\mathcal{W}}

\newcommand\bx{{\bf{x}}}

\newcommand{\eset}{\emptyset}

\newcommand{\derat}[1]{\frac{d}{dt}\big\vert_{t = #1}}

\newcommand{\dd}{{\tt d}}

\newcommand{\trile}{\trianglelefteq}
\newcommand{\subeq}{\subseteq}
\newcommand{\supeq}{\supseteq}

\newcommand{\into}{\hookrightarrow}
\newcommand{\eps}{\varepsilon}

\newcommand{\N}{{\mathbb N}}
\newcommand{\Z}{{\mathbb Z}}
\newcommand{\R}{{\mathbb R}}
\newcommand{\C}{{\mathbb C}}

\renewcommand{\H}{{\mathbb H}}
\newcommand{\T}{{\mathbb T}}

\newcommand{\bE}{{\mathbb E}}
\newcommand{\bH}{{\mathbb H}}
\newcommand{\bS}{{\mathbb S}}

\renewcommand{\hat}{\widehat}

\renewcommand{\tilde}{\widetilde}

\newcommand{\Aff}{\mathop{{\rm Aff}}\nolimits}

\newcommand{\GL}{\mathop{{\rm GL}}\nolimits}
\newcommand{\SL}{\mathop{{\rm SL}}\nolimits}

\newcommand{\PSL}{\mathop{{\rm PSL}}\nolimits}
\newcommand{\SO}{\mathop{{\rm SO}}\nolimits}

\newcommand{\U}{\mathop{\rm U{}}\nolimits}

\newcommand{\fsl} {\mathop{{\mathfrak{sl} }}\nolimits}

\newcommand{\su}  {\mathop{{\mathfrak{su} }}\nolimits}
\newcommand{\so}  {\mathop{{\mathfrak{so} }}\nolimits}

\newcommand{\Exp}{\mathop{{\rm Exp}}\nolimits}
\newcommand{\Fix}{\mathop{{\rm Fix}}\nolimits}

\newcommand{\ad}{\mathop{{\rm ad}}\nolimits}
\newcommand{\Ad}{\mathop{{\rm Ad}}\nolimits}

\renewcommand{\Re}{\mathop{{\rm Re}}\nolimits}
\renewcommand{\Im}{\mathop{{\rm Im}}\nolimits}

\newcommand{\Hom}{\mathop{{\rm Hom}}\nolimits}

\newcommand{\bL}{{\mathbb L}}

\newcommand{\Aut}{\mathop{{\rm Aut}}\nolimits}

\newcommand{\End}{\mathop{{\rm End}}\nolimits}

\renewcommand{\dim}{\mathop{{\rm dim}}\nolimits}

\newcommand{\supp}{\mathop{{\rm supp}}\nolimits}

\newcommand{\Inn}{\mathop{{\rm Inn}}\nolimits}

\newcommand{\ev}{\mathop{{\rm ev}}\nolimits}

\newcommand{\dS}{\mathop{{\rm dS}}\nolimits}
\newcommand{\PSO}{\mathop{{\rm PSO}}\nolimits}

\newcommand{\indlim}{{\displaystyle \lim_{\longrightarrow}}\ }

\newcommand{\Rarrow}{\Rightarrow}
\newcommand{\nin}{\noindent} 
\newcommand{\oline}{\overline}

\newcommand{\la}{\langle}
\newcommand{\ra}{\rangle}

\newcommand{\res}{\vert}

\newcommand{\spann}{{\rm span}}

\newcommand{\Spec}{{\rm Spec}}
\newcommand{\Spin}{{\rm Spin}}

\newcommand{\ssssarr}{\hbox to 15pt{\rightarrowfill}}
\newcommand{\sssarr}{\hbox to 20pt{\rightarrowfill}}
\newcommand{\ssarr}{\hbox to 30pt{\rightarrowfill}}
\newcommand{\sarr}{\hbox to 40pt{\rightarrowfill}}
\newcommand{\arr}{\hbox to 60pt{\rightarrowfill}}
\newcommand{\larr}{\hbox to 60pt{\leftarrowfill}}
\newcommand{\Arr}{\hbox to 80pt{\rightarrowfill}}

\def\theoremname{Theorem}
\def\propositionname{Proposition}
\def\corollaryname{Corollary}
\def\lemmaname{Lemma}
\def\remarkname{Remark}
\def\conjecturename{Conjecture} 

\def\definitionname{Definition}
\def\exercisename{Exercise}
\def\examplename{Example}
\def\examplesname{Examples}
\def\problemname{Problem}
\def\problemsname{Problems}

\def\satzname{Satz} 
\def\koroname{Korollar}
\def\folgname{Folgerung}
\def\bemerkname{Bemerkung}
\def\aufgname{Aufgabe}

\def\beisname{Beispiel}
\def\beissname{Beispiele}
\def\bewname{Beweis}

\def\@thmcounter#1{\noexpand\arabic{#1}}
\def\@thmcountersep{}
\def\@begintheorem#1#2{\it \trivlist \item[\hskip 
\labelsep{\bf #1\ #2.\quad}]}
\def\@opargbegintheorem#1#2#3{\it \trivlist
      \item[\hskip \labelsep{\bf #1\ #2.\quad{\rm #3}}]}
\makeatother
\newtheorem{theor}{\theoremname}[section]
\newtheorem{propo}[theor]{\propositionname}
\newtheorem{coro}[theor]{\corollaryname}
\newtheorem{lemm}[theor]{\lemmaname}

\newenvironment{thm}{\begin{theor}\it}{\end{theor}}
\newenvironment{theorem}{\begin{theor}\it}{\end{theor}}

\newenvironment{prop}{\begin{propo}\it}{\end{propo}}
\newenvironment{proposition}{\begin{propo}\it}{\end{propo}}

\newenvironment{cor}{\begin{coro}\it}{\end{coro}}

\newenvironment{lem}{\begin{lemm}\it}{\end{lemm}}
\newenvironment{lemma}{\begin{lemm}\it}{\end{lemm}}

\newtheorem{rema}[theor]{\remarkname}

\newenvironment{remark}{\begin{rema}\rm}{\end{rema}}
\newenvironment{rem}{\begin{rema}\rm}{\end{rema}}

\newtheorem{stepnow}[theor]{}

\newtheorem{defin}[theor]{\definitionname} 

\newenvironment{defn}{\begin{defin}\rm}{\end{defin}}

\newtheorem{exerc}{\exercisename}[section]

\newtheorem{exa}[theor]{\examplename}

\newenvironment{ex}{\begin{exa}\rm}{\end{exa}}

\newtheorem{exas}[theor]{\examplesname}

\newtheorem{conj}[theor]{\conjecturename}

\newtheorem{pro}[theor]{\problemname}

\newenvironment{prob}{\begin{pro}\rm}{\end{pro}}

\newtheorem{prs}[theor]{\problemsname}

\newtheorem{aufg}{\aufgname}[section]

\newenvironment{prf}{\begin{proof}}{\end{proof}}
 
\newcommand{\pmat}[1]{\begin{pmatrix} #1 \end{pmatrix}}

{\hfill\qed\end{trivlist}}

\newenvironment{beweis*}{\begin{trivlist}\item[\hskip%
\labelsep{\bf\bewname.\quad}]}%
{\end{trivlist}}

\newtheorem{satzn}[theor]{\satzname}

\newtheorem{koro}[theor]{\koroname}

\newtheorem{folg}[theor]{\folgname}

\newtheorem{bem}[theor]{\bemerkname}

\newtheorem{aufgn}[theor]{\aufgname}

\newtheorem{beis}[theor]{\beisname}

\newtheorem{beiss}[theor]{\beissname}

\newcommand\be{{\bf{e}}}

\newcommand{\sH}{{\sf H}}

\newcommand{\sV}{{\tt V}}
\newcommand{\sE}{{\tt E}}
\newcommand{\sF}{{\tt F}}

\newcommand{\cX}{\mathcal X}

\renewcommand{\phi}{\varphi} 

\newcommand{\AdS}{\mathop{{\rm AdS}}\nolimits}

\newcommand{\hgf}{{}_2F_1}
\newcommand{\hgfabc}{{}_2F_1(\alpha,\beta;\gamma;z)}

\newcommand{\wz}{\widetilde{z}}

\renewcommand\mlabel{\label} 

\begin{document}

\title{Nets of standard subspaces on \\
  non-compactly causal   symmetric spaces} 
\author{Jan Frahm, Karl-Hermann Neeb, Gestur \'Olafsson}

\maketitle

\begin{abstract}
Let $G$ be a connected simple linear Lie group and $H\subset G$ a symmetric subgroup such that
the corresponding symmetric space 
  $G/H$ is non-compactly causal. 
We show that any irreducible unitary representation of $G$ leads naturally
to a net of standard subspaces on $G/H$
that is isotone, covariant and has the Reeh--Schlieder
and the Bisognano--Wichmann property. 
We also  show that this result extends to the universal covering
group of $\SL_2(\R)$, which has some interesting application
to intersections of standard subspaces associated to representations
of such groups. For this a detailed study of hyperfunction and distribution vectors is needed. In
  particular we show that every
$H$-finite hyperfunction vector is in fact a distribution vector.
\end{abstract}

\tableofcontents 

\section{Introduction} 
\mlabel{sec:1}

This article is part of an ongoing project
exploring the connections between causal structures
on homogeneous spaces, Algebraic Quantum Field Theory (AQFT), 
modular theory of operator algebras 
and unitary representations of Lie groups
(cf.~\cite{NO21, NO23b, NO23, MN21, MNO23a}).

The main achievement of the present paper is the construction
of a net of real subspaces for any irreducible representation 
of a connected simple linear Lie group
$G$ on any corresponding non-compactly causal symmetric space
$G/H$, such that the Reeh--Schlieder and the Bisognano--Wichman
condition are satisfied. To explain these concepts, 
let us call a closed real subspace $\sV$ of a complex Hilbert space
$\cH$ {\it cyclic} if $\sV + i \sV$ is dense in $\cH$,
{\it separating} if $\sV \cap i \sV = \{0\}$, and
{\it standard} if it is both. If $\sV$ is standard, then
the complex conjugation $T_\sV$ on $\sV + i \sV$ has a polar decomposition
$T_\sV = J_\sV \Delta_\sV^{1/2}$, where
$J_\sV$ is a conjugation (an antilinear involutive isometry)
and $\Delta_\sV$ is a positive selfadjoint operator
satisfying $J_\sV \Delta_\sV J_\sV = \Delta_\sV^{-1}$.
The unitary one-parameter group $(\Delta_\sV^{it})_{t \in \R}$,
the {\it modular group of $\sV$}, preserves the subspace $\sV$.
We refer to \cite{Lo08, NO17} for more on standard subspaces.

For a unitary representation $(U,\cH)$ of a Lie group~$G$ 
and a homogeneous space $M = G/H$, we are interested in
families $(\sH(\cO))_{\cO \subeq M}$ of closed real subspaces of $\cH$,
indexed by open subsets $\cO \subeq M$;
so-called {\it nets of real subspaces}. For such nets, we
consider the following properties:
\begin{itemize}
\item[(Iso)] {\bf Isotony:} $\cO_1 \subeq \cO_2$ 
implies $\sH(\cO_1) \subeq \sH(\cO_2)$ 
\item[(Cov)] {\bf Covariance:} $U_g \sH(\cO) = \sH(g\cO)$ for $g \in G$. 
\item[(RS)] {\bf Reeh--Schlieder property:} 
$\sH(\cO)$ is cyclic  if $\cO \not=\eset$. 
\item[(BW)] {\bf Bisognano--Wichmann property:} 
There exists an open subset $W \subeq M$ (called a {\it wedge region}),  
and an element $h \in \g$ 
such that $\sH(W)$ is standard and its modular group satisfies 
\[ \Delta_{\sH(W)}^{-it/2\pi} = U(\exp t h) \quad \mbox{ for } \quad 
t \in \R.\] 
\end{itemize}
Presently we leave locality conditions aside.
We refer to Section~\ref{subsec:locality}, where the related
difficulties are explained. We plan to address them in subsequent work. 

Nets satisfying (Iso) and (Cov) can easily be constructed as follows. 
The subspace  
$\cH^\infty \subeq \cH$ of vectors $v \in \cH$ for which the orbit map 
\[ U^v \colon G \to \cH, \quad g \mapsto U(g)v,\]  is smooth 
({\it smooth vectors}) is dense
and carries a natural Fr\'echet topology for which the action of 
$G$ on this space is smooth (\cite{Go69, Ne10}). 
The space $\cH^{-\infty}$ of continuous antilinear functionals $\eta \colon \cH^\infty \to \C$ 
({\it distribution vectors}) 
contains in particular Dirac's kets 
$\la \cdot, v \ra$, $v \in \cH$, so that 
we obtain complex linear embeddings 
\[ \cH^\infty \into \cH \into \cH^{-\infty},\] 
where $G$ acts on all three spaces 
by representations denoted $U^\infty, U$ and $U^{-\infty}$, respectively.
Here we follow the convention common in physics to
require inner products to be conjugate linear in the first and 
complex linear in the second argument. 

All of the three above  representations can be integrated to the
convolution algebra $C^\infty_c(G) := C^\infty_c(G,\C)$ of
test functions, for instance $U^{-\infty}(\phi) := \int_G \phi(g)U^{-\infty}(g)\, dg$,
where $dg$ stands for a left Haar measure on $G$. 
To any real subspace $\sE \subeq \cH^{-\infty}$ and every open subset 
$\cO \subeq G$, we associate the closed real subspace 
\begin{equation}
  \label{eq:HE}
  \sH_\sE^G(\cO) := \oline{\spann_\R U^{-\infty}(C^\infty_c(\cO,\R))\sE}.
\end{equation}
Note that the operators $U^{-\infty}(\phi)$ map 
$\cH^{-\infty}$ into $\cH$ because they are adjoints of
  continuous operators $U(\phi) \colon \cH \to \cH^\infty$
  (\cite{Ga47}, \cite[App.~A]{NO21}).
Accordingly, the closure is taken with respect
  to the topology of $\cH$ (cf.\ in particular Proposition~\ref{prop:4.8}).
  Clearly, this definition does also make sense for real subspaces
  $\sE \subeq \cH$, but the key advantage of working with the larger
  space $\cH^{-\infty}$ of distribution vectors is that it contains
  finite-dimensional subspaces invariant under
 $\ad$-diagonalizable elements and
  non-compact subgroups (cf.\ Lemma~\ref{lem:2.16}).
  For  finite-dimensional subspaces of $\cH$, this is excluded by 
  Moore's Theorem (\cite{Mo80}).
On a homogeneous space $M = G/H$ with the projection map
$q \colon G \to M$, we now obtain a ``push-forward net''
\begin{equation}
  \label{eq:pushforward}
  \sH^M_\sE(\cO) := \sH_\sE^G(q^{-1}(\cO)).
\end{equation}
The so-obtained net on $M$ thus corresponds to the restriction
  of the net $H^G_\sE$ indexed by open subsets of $G$ to those open subsets
  $\cO\subeq G$  which are {$H$-right} invariant in the sense that $\cO = \cO H$;
  these are the inverse images of open subsets of $M$ under~$q$.
  If, in addition, $\sE$ is invariant under $U^{-\infty}(H)$, then \cite[Lemma~2.11]{NO21}
  implies that $\sH^G_\sE(\cO)= \sH^G_\sE(\cO H)$ for any open subset
  $\cO \subeq G$, so that $\sH^G_\sE$ can be recovered from the
  net $\sH^M_\sE$ on $M$ by
  $\sH^G_\sE(\cO) = \sH^G_\sE(\cO H) = \sH^M_\sE(q(\cO))$.   
  The assignment \eqref{eq:pushforward} trivially satisfies (Iso), and (Cov)
follows from the simple relation
\[ U^{-\infty}(g) U^{-\infty}(\phi)  = U^{-\infty}(\delta_g * \phi),\]
where $(\delta_g * \phi)(x) = \phi(g^{-1}x)$ is the left translate
of $\phi$. Therefore, a key problem
is to specify subspaces $\sE$ of distribution vectors for which
(RS) and (BW) hold as well.
According to \cite{MN23},
the potential generators $h \in \g$ of the modular groups
in (BW) are {\it Euler elements}, i.e.,
$\ad h$ defines a $3$-grading
\[ \g = \g_1(h) \oplus \g_0(h) \oplus \g_{-1}(h), 
\quad \mbox{ where } \quad \g_\lambda(h)
= \ker(\ad h - \lambda \1).\]
Moreover, if $\ker(U)$ is discrete, then the conjugation
$J := J_{\sH(W)}$ satisfies
\begin{equation}
  \label{eq:j-rel}
  J U(g) J = U(\tau_h^G(g)) \quad \mbox{ for } \quad g \in G,
\end{equation}
where $\tau_h^G \in \Aut(G)$ is an involution
inducing $\tau_h := e^{\pi i \ad h}\res_{\g}$ on the Lie algebra.
So we may fix an Euler element $h \in \g$ from the start.
Throughout, $\g$ denotes a {\bf simple real Lie algebra}
(if not otherwise stated). We choose a Cartan involution
$\theta$ with $\theta(h)= -h$ and observe that
$\tau := \theta e^{\pi i \ad h}$ defines an involution on~$\g$.
Writing $\g = \fh \oplus \fq$ with
\[ \fh =  \g^\tau 
  =\{x\in\fg \colon \tau (x) = x\} \quad \mbox{ and } \quad
  \fq = \g^{-\tau}=\{x\in\fg\colon \tau(x) =-x\},\] 
the subspace $\fq$ contains $e^{\ad \fh}$-invariant convex cones~$C$
with $h \in C^\circ$ (\cite[Thm.~4.21]{MNO23a}).
We refer to \cite{MNO23a} for a detailed discussion of the connection between
Euler elements and non-compactly causal symmetric spaces. 

The {\it causal homogeneous spaces} we consider in this paper
are symmetric spaces $M = G/H$, where
$G$ is a connected Lie group with Lie algebra $\g$
carrying an involutive automorphism $\tau^G$ corresponding to $\tau$,
and $H$ is an open subgroup of the group $G^\tau$ of
$\tau^G$-fixed points, preserving a pointed generating cone
$C \subeq \fq \cong T_{eH}(M)$ with $h \in C^\circ$.
Then $V_+(gH) := g.C^\circ \subeq T_{gH}(M)$
defines a $G$-invariant field $(V_+(m))_{m \in M}$ of open cones on $M$;
called a {\it causal structure}.
Here $G \times T(M) \to T(M), (g,v) \mapsto g.v$ denotes the
  canonical lift of the $G$-action on $M$ to the tangent bundle.
This brings us to an additional property of a causal
symmetric Lie algebra $(\g, \tau,C)$. It is called
\begin{itemize}
\item {\it compactly causal (cc)}, if all elements $x \in C^0$ are {\it elliptic}
  in the sense that $\ad x$ is semisimple with purely imaginary spectrum.
  An important Lorentzian
  example   is {\it anti-de Sitter space}
  \[\AdS^d = \SO_{2,d-1}(\R)_e/\SO_{1,d-1}(\R)_e.\] 
\item {\it non-compactly causal (ncc)}, if all elements $x \in C^0$ are
  {\it hyperbolic} 
  in the sense that $\ad x$ is diagonalizable over $\R$.
    Here an important Lorentzian example 
  is {\it de Sitter space}  $\dS^d = \SO_{1,d}(\R)_e/\SO_{1,d-1}(\R)_e$. 
\end{itemize}
If $G$ is semisimple with finite center,
then, for compactly causal spaces, there are closed causal curves,
so that no global causal order exists on $M$,
but, for non-compactly causal spaces, there even exists a
global order which is {\it globally hyperbolic} in the sense that
all order intervals are compact (\cite[Thm.~5.3.5]{HO97}).

In Algebraic Quantum Field Theory 
in the sense of Haag--Kastler, one considers 
{\it nets} of von Neumann algebras $\cM(\cO)$
on a fixed Hilbert space $\cH$,
associated to  regions $\cO$ in some space-time manifold~$M$ 
(\cite{Ha96}). 
The hermitian elements of the algebra $\cM(\cO)$ represent
observables  that can be measured in the ``laboratory'' $\cO$. 
Accordingly, one requires {\it isotony}, i.e.,
that 
\[\cO_1 \subeq \cO_2 \Rightarrow 
\cM(\cO_1) \subeq \cM(\cO_2).\] 
One further assumes the existence of a  unitary representation 
$U \colon G \to \U(\cH)$  of a Lie group~$G$, acting as a space-time symmetry 
group on $M$, such that
\[U(g) \cM(\cO) U(g)^* = \cM(g\cO)\quad\text{for } g \in G \quad
\text{(covariance)}.\] In addition, one assumes a $U(G)$-fixed 
unit vector $\Omega \in \cH$, representing typically 
a vacuum state of a quantum field. The domains $\cO \subeq M$ for which 
$\Omega$ is cyclic and separating for $\cM(\cO)$ are of particular 
relevance. For these domains $\cO$, the
Tomita--Takesaki Theorem (\cite[Thm.~2.5.14]{BR87})
yields for the von Neumann algebra 
$\cM(\cO)$ a conjugation (antiunitary involution) $J_\cO$ and a
positive selfadjoint operator $\Delta_\cO$ satisfying
\begin{equation} \label{eq:j1}
J_\cO \cM(\cO) J_\cO = \cM(\cO)' 
\quad \mbox{ and } \quad
\Delta_\cO^{it} \cM(\cO)\Delta_\cO^{-it} = \cM(\cO) \quad \mbox{
  for } \quad t \in \R.
\end{equation}
This defines a standard subspace
$\sV =\Fix(J_\cO\Delta_{\cO}^{1/2})$ 
connecting the Tomita--Takesaki theory to the theory of standard subspaces.
We also obtain the modular automorphism group defined by 
\[\widetilde{\alpha}_t(A) = \Delta_\cO^{-it/2\pi} A \Delta_\cO^{it/2\pi},\quad A \in B(\cH).\] 
It is now an interesting question when this modular group is ``geometric'' 
in the sense that it is implemented by a one-parameter subgroup of~$G$.
This is always the case for nets of operator algebras 
obtained from nets of real subspaces $\sH(\cO)$ satisfying
(Iso), (Cov), (RS) and (BW) by some second quantization functor
(\cite{Si74}). Here the modular group corresponds to the flow defined by
$\alpha_t(m) = \exp(th).m$ on $M = G/H$, where $h$ is an Euler element.
With a view towards the connection with AQFT,
the set of pairs $(h,\tau)$, consisting of an Euler element
$h \in \g$ and an involution $\tau$ with $\tau(h) = h$, has 
been studied in \cite{MN21}.
Although our  construction ignores field theoretic interactions,
it displays already some crucial features of quantum
field theories. This has been explored for the flat case in \cite{NOO21}, 
and in \cite{NO21} for the class of compactly causal symmetric spaces
and unitary highest weight representations $(U,\cH)$ of $G$.
In the present paper we address non-compactly causal spaces and general unitary
representations. For non-compactly causal spaces,
the group $G$ may not have non-trivial
unitary representations in which a non-zero Lie algebra element
has semibounded spectrum
(e.g.\ the Lorentz group $\SO_{1,d}(\R)_e$ for $d > 2$).
Therefore the methods developed in \cite{NO21, NO23b} do not apply.

To deal with the Bisognano--Wichmann property (BW),
we have to specify suitable wedge regions~$W\subeq M$.
As the specific examples in AQFT suggest, the modular flow on $W$
should be timelike future-oriented
because the modular flow should correspond to the
``flow of time'' (see \cite{CR94} and also
\cite{BB99, BMS01, Bo09}, \cite[\S 3]{CLRR22}).
In our context this means that
the {\it modular vector field}
\begin{equation}
  \label{eq:xhdef}
 X_h^M(m) 
 := \frac{d}{dt}\Big|_{t = 0} \alpha_t(m)
=  \frac{d}{dt}\Big|_{t = 0} \exp(th).m 
\end{equation}
should satisfy
\[ X^M_h(m) \in V_+(m) \quad \mbox{ for all } \quad m \in W.\] 
Choosing the Euler element $h$ in such a way that
$h \in C^\circ$, this condition is satisfied in the
base point, so that 
the connected component $W := W_M^+(h)_{eH}$ of the base point
$eH \in M$ in the {\it positivity region}
\begin{equation}\mlabel{def:WM}
 W_M^+(h) := \{ m \in M \colon X^M_h(m) \in V_+(m)  \} 
 \end{equation}
is the natural candidate for a domain for which (BW) could be satisfied.
These ``wedge regions'' have been studied for compactly and
non-compactly causal symmetric spaces
in \cite{NO23b} and \cite{NO23, MNO23b}, respectively. 

We start with a unitary representation $(U,\cH)$ of $G$
that extends to an antiunitary representation of
\[ G_{\tau_h} := G \rtimes \{\1,\tau_h^G\}
  \quad \mbox{ by } \quad U(\tau_h^G) := J,\]
  where $J$ is a conjugation with $J U(g) J =  U(\tau^G_h(g))$ for $g \in G$
  (see \eqref{eq:j-rel}). 
We write $\sV \subeq \cH$ for the standard subspace
specified by 
\begin{equation} \label{eq:mod-obj} 
  \Delta_\sV = e^{2\pi i \cdot \partial U(h)} \quad \mbox{ and } \quad J_\sV = J,
\end{equation}
where $\partial U(h)$ is the skew-adjoint infinitesimal generator
of the one-parameter group $U(\exp th)$. Denote
the open $\pi$-strip in $\C$ by
$\cS_\pi := \{ z \in \C \colon 0 <  \Im z <  \pi\}$.
Then the elements $\xi \in \sV$ are characterized among elements of $\cH$
by the {\it KMS like condition} that the orbit map
$U_h^\xi(t) := U(\exp th)\xi$ has a continuous extension
$U_h^\xi \colon \oline{\cS_\pi} \to \cH$
to the closed strip $\oline{\cS_\pi}$,
which is holomorphic on the interior and satisfies 
\begin{equation}\mlabel{eq:UhJ}
 U_h^\xi(\pi i) = J \xi\quad \mbox{ resp.} \quad
  U_h^\xi(t + \pi i) = J U_h^\xi(t)\quad \mbox{ for }\quad
t \in \R\end{equation}
(cf.\ \cite[Prop.~2.1]{NOO21}). 

Denote by $\cH^\omega \subseteq \cH$ the space of analytic
vectors, endowed with its natural locally convex
topology (cf.\ Section~\ref{sec:anavec}) and by $\cH^\infty \subeq \cH$ the
space of smooth vectors with the standard topology. Then $\cH^{-\omega}$, the space of {\it hyperfunction vectors}, is
its antidual space (the continuous antilinear functionals) and $\cH^{-\infty}$, the space of {\it distribution vectors}, is
the the antidual of $\cH^\infty$. We note that $\cH^{-\omega}\subeq \cH^{-\infty}$ but in general they are not the same. 

We now define subspaces
\begin{equation}
  \label{eq:kms-spaces-intro}
 \cH^{-\infty}_{\rm KMS} \subeq \cH^{-\infty} \quad \mbox{ and } \quad
 \cH^{-\omega}_{\rm KMS} \subeq \cH^{-\omega},
\end{equation}
as the space of hyperfunction, respectively, distribution vectors such that
the orbit map extends continuously, with respect to the weak-$*$-topology, to the the closed strip $\overline{\cS}_\pi$, weak-$*$
holomorphic in the interior, such that \eqref{eq:UhJ}
is satisfied (cf.~Definition \ref{def:kms-subspace}). 
For $G = \R$ our main new result on
  hyperfunction vectors is Theorem~\ref{thm:closed-hyperfunc}, 
  asserting that the subspace $\cH^{-\omega}_{\rm KMS}$ of
  $\cH^{-\omega}$ (cf.~\ref{eq:kms-spaces-intro}) 
  is closed and can be characterized as the annihilator of
  a concrete real subspace of $\cH^\omega$ with respect to the pairing
  induced by $\Im \la \cdot,\cdot \ra$.
Further, 
  Theorem~\ref{thm:E.4} provides a characterization of
  distribution vectors that are limits of analytic vectors
  in terms of asymptotic growths of the norm on curves
  $t \mapsto e^{tH}v$, where $v \in \cH$, and $H = H^*$
  if the selfadjoint generator of the unitary one-parameter group.

We now explain how to find
suitable subspaces $\sE$ of distribution vectors
if $G$ is a connected simple real Lie group.
  Let $K := G^\theta$ be the group of fixed points of a Cartan involution
  $\theta$, so that $\Ad(K) \subeq \Ad(G)$ is maximally compact. 
We consider a finite-dimensional $K$-invariant subspace 
$\cE \subeq \cH$ which is also $J$-invariant and set
\[ \sE_K := \cE^J = \{ v \in \cE \colon Jv = v\}.\]
Note that $\tau_h(K) = K$ implies that $J$ leaves the subspace
$\cH^{[K]}$ of $K$-finite vectors invariant. Therefore $J$-invariant
finite-dimensional $K$-invariant subspaces exist in abundance.
We work with the following two hypotheses:
\begin{itemize}
\item[\rm(H1)] 
  $\cE \subeq \bigcap_{x \in \Omega_{\fp}} \cD(e^{i \partial U(x)})$
  (Proposition~\ref{prop:hyp1}), where
$\Omega_\fp \subeq \fp$ consists of all elements for which the spectral
radius of $\ad x$ is smaller than $\frac{\pi}{2}$. Note that
$th \in \Omega_\fp$ if and only if $|t| < \pi/2$. 
\item[\rm(H2)] 
  In addition,   the limits
  \[ \beta^+(v) := \lim_{t \to \pi/2} e^{-it \partial U(h)} v, \quad v \in \cE, \]
  which exist in the space $\cH^{-\omega}$ of hyperfunction vectors,
  are actually contained  in the space $\cH^{-\infty}$ of distribution vectors.
\end{itemize}

Natural equivariance properties (Proposition~\ref{prop:exten}) 
then imply that  
\begin{equation}
  \label{eq:eh}
  \sE_H := \beta^+(\sE_K) \subeq \cH^{-\infty}
\end{equation}
is a finite-dimensional $H$-invariant subspace, and  
$\sH^M_{\sE_H}$, defined  as in \eqref{eq:pushforward}, 
specifies a net of real subspaces on $M = G/H$.
Our main result (Theorem~\ref{thm:4.9}) states that this net also
satisfies (RS) and (BW). Here it is of key importance that our
  construction is such that the real subspace $\sE_H$ is actually
  contained in the subspace $\cH^{-\infty}_{\rm KMS}$. The hard part of
  our argument is to verify that we also have
  $\sH^M_{\sE_H}(W) \subeq \cH^{-\infty}_{\rm KMS}$, which in turn leads
  to   $\sH^M_{\sE_H}(W) \subeq \sV = \cH_{\rm KMS}$
  (cf.\ Proposition~\ref{prop:4.8}). 

There are two approaches to verify our hypotheses.
For the class of linear simple groups $G$ one can use the
Kr\"otz--Stanton Extension Theorem (see Theorem~\ref{thm:ks04} below)
to verify Hypothesis (H1). Hypothesis (H2)
requires the additional regularity that
the subspace $\sE_H$, which a priori
consists only of hyperfunction vectors, actually
consists of distribution vectors.
For finite coverings of linear groups, this 
can be derived from a generalization
of the van den Ban--Delorme Automatic Continuity Theorem
(Theorem~\ref{thm:autocont}).
So (H1) and (H2) hold in particular if $G$ is linear,
but we expect them to hold in general.
This problem is discussed in Section~\ref{sec:6},
where we argue that 
the Casselman Subrepresentation Theorem extends
to Harish--Chandra modules for which $G$ does not have to be linear. 
Therefore the missing result is a generalization of the
Casselman--Wallach Globalization Theorem~\cite{BK14}
to non-linear groups.

An example that is of particular importance in physics
arises for the group  $G =\SO_{1,d}(\R)_e$ and $H = \SO_{1,d-1}(\R)_e$, for which
$M = G/H = \dS^d$ is $d$-dimensional de Sitter space.
For $d \geq 3$, the simply connected covering
is the spin group $\tilde G\cong \Spin_{1,d}(\R) \subeq \Spin_{1+d}(\C)$.
As this group is linear, our results yield in particular
for every irreducible unitary representation $(U,\cH)$ of this group, 
and any $\sE_K\subeq \cH^{[K],J}$
as above, a corresponding net $\sH^M_{\sE_H}$ on de Sitter space.
For spherical representations, 
  i.e., representations where $\cE$ can be chosen
  as $\cE = \C v_K$ for a $K$-fixed vector~$v_K$, the corresponding subspace 
  $\sE_H = \R v_H\subeq \cH^{-\infty}$ is spanned by a
  non-zero $H$-fixed
  distribution vector $v_H$ (Proposition~\ref{prop:exten}(c),(d)). 
  Through the natural analytic extension process from functions on
  $G/K$ to distributions on $G/H$,
  these $H$-fixed distribution vectors now lead to a $G$-equivariant
  realization of the representation $(U,\cH)$ as a Hilbert space
  of distributions on $G/H$, on which the scalar product is specified
  by a distribution kernel $D \in C^{-\infty}(M \times M,\C)$.
  For irreducible unitary representations, these kernels
  satisfy differential equations (in both variables) coming from the 
  center of the enveloping algebra of $\g$. In the special situation
  of de Sitter space $\dS^d$ and $G = \SO_{1,d}(\R)_e$,
  one obtains eigendistributions of the corresponding
  wave operator. These kernels are precisely
  the perikernels occuring in the work
  of J.~Bros and U.~Moschella in the 1990s 
  (cf.\ \cite{BM96}, \cite{Mo95}, \cite{BEM98}).
  This correspondence also emerges quite naturally in the
  context of reflection positivity on spheres, where
  these kernels can be described explicitly in terms of
  hypergeometric functions (see in particular \cite{NO20}
  and Section~\ref{subsec:sphefunc} below).
  For non-spherical representations, it is much harder to
  match the picture developed in \cite{BM96}
  with our representation theoretic approach.
  For some recent progress in this direction we refer to \cite{FNO24}.

  Identifying the conformal group $\SO_{2,d}(\R)_e$ of Minkowski space
  with a conformal group acting (locally) on de Sitter space $\dS^d$, 
  our construction yields in particular
  conformally covariant nets on (domains in) de Sitter space,
  as they are studied by Guido and Longo in \cite{GL03}.
  We expect that many aspects of their work,
  in particular the ``holographic reconstruction'' of nets on $M$ 
  from nets on ``horizons'',   have counterparts
  for ``conformally flat'' causal symmetric spaces as
  they are classified by Bertram in \cite{Be96}, but this remains
  to be explored in more detail. For $d = 2$ this leads to the
  interesting relation between nets on de Sitter space
  $\dS^2$, anti-de Sitter space $\AdS^2$ and conformal
  nets on the circle. We refer to Subsection~\ref{subsec:psl2},
  and in particular Remark~\ref{rem:herm} 
  for more details on this circle of ideas.

In Section~\ref{sec:rank-one} we show that
Hypothesis (H1) always holds if 
$\dim \cE = 1$ ($\cE = \C v$, $Jv = v$) for real rank-one groups. 
We further show that there exists constants $C > 0$ and $N > 0$ such that
\begin{equation}
  \label{eq:qn1}
  \|e^{it \partial U(h)}v\| \leq C  \Big(\frac{\pi}{2}-t\Big)^{-N} \quad \mbox{
    for } \quad 0 \leq t < \frac{\pi}{2}.
\end{equation}
This implies that Hypothesis (H2) also holds. 
Since all irreducible $K$-sub\-rep\-resentations are one-dimensional
  if $\g = \fsl_2(\R)$, both (H1) and (H2) are satisfied in this case,
  so that that   Theorem~\ref{thm:4.9} also applies to all
irreducible unitary representations
of any connected Lie group $G$ with Lie algebra~$\fsl_2(\R)$.
As $\fsl_2(\R) \cong \so_{1,2}(\R)$, this covers in particular
the situation where $M$ is any covering of the $2$-dimensional
de Sitter space $\dS^2$.
It allows us in particular to  solve 
the $\SL_2$-problem on intersections of standard subspaces,
as formulated in \cite[\S 4]{MN21}, in the affirmative.

The structure of this paper is as follows.
We start in Section~\ref{sec:anavec} with a
discussion of hyperfunction and distribution vectors
for unitary representations $(U,\cH)$ of a Lie group~$G$
and discuss the crucial case $G = \R$ in some detail
in Sections~\ref{subsec:h2} and Section~\ref{subsec:h3}.
In particular Theorem~\ref{thm:E.4} will be used
  in Section~\ref{sec:rank-one} to verify Hypotheses (H1) and (H2)
  for groups with Lie algebra $\fsl_2(\R)$. 
In Section~\ref{sec:3} we consider for a unitary representation
$(U,\cH)$ of a connected simple Lie group
holomorphic extensions of orbits maps $U^v \colon G \to \cH, g \mapsto U(g)v$.
Here we introduce Hypotheses (H1) and (H2) and their consequences.
In Section~\ref{sec:4} we turn to the Reeh--Schlieder (RS) and the
Bisognano--Wichmann (BW) property of the nets of real subspaces
$\sH^M_{\sE_H}$. In particular we prove our main result (Theorem~\ref{thm:4.9}),
which applies in particular if $G$ is linear, i.e., contained in its
complexification. In Section~\ref{sec:rank-one} this result is extended to
covering groups of $\SL_2(\R)$. A possible strategy to extend it
to general connected simple Lie groups is outlined in
Section~\ref{sec:6}. Appendix~\ref{app:C}
contains a proof of a
version of the   van den Ban--Delorme Automatic Continuity Theorem
to $\fh$-finite vectors (Theorem~\ref{thm:autocont}).
Finally, Appendix~\ref{app:b} contains a brief discussion
of results on wedge regions in non-compactly causal spaces that
have been developed in \cite{MNO23b}.

\vspace{4mm}

\nin {\bf Acknowledgment:} 
We are particularly indebted to Job Kuit for several discussions concerning the Automatic Continuity Theorem outlined in Appendix~\ref{app:C}. 
  We also thank Erik van den Ban and Patrick Delorme
  for helpful comments on this matter. 
We further thank Jacques Faraut for communicating a proof of
Proposition~\ref{prop:tempmeas}.
We are also indepted to Tobias Simon for abundant comments
on a first draft of this paper and also for comments by Joachim Hilgert.

Last, but not least, we wish to thank both referees for
very inspiring and most useful reports. Their suggestions have
been very helpful in improving the exposition.

\vspace{4mm}

\nin {\bf Notation:}
\begin{itemize}
\item For $r > 0$ we denote the corresponding horizontal strips in $\C$
  by
  \[ \cS_r := \{ z \in \C \colon 0 < \Im z < r\} \quad \mbox{ and } \quad 
    \cS_{\pm r} := \{ z \in \C \colon |\Im z | < r\}.\]
\item For a unitary representation $(U,\cH)$ of $G$ we write:
  \begin{itemize}
  \item $\cH^{[K]}$ for the space of $K$-finite vectors, where $K \subeq G$
    is a subgroup.
  \item $\partial U(x) = \derat0 U(\exp tx)$ for the infinitesimal
    generator of the unitary one-parameter group $(U(\exp tx))_{t \in\R}$
    in the sense of Stone's Theorem. 
  \item $\dd U \colon \cU(\g_\C) \to \End(\cH^\infty)$ for the representation of
    the enveloping algebra $\cU(\g_\C)$ of $\g_\C$ on
    the space $\cH^\infty$ of smooth vectors. Then
    $\partial U(x) = \oline{\dd U(x)}$ for $x \in \g$. 
  \end{itemize}
\end{itemize}

\section{Analytic and hyperfunction vectors}
\mlabel{sec:anavec}

In this section we discuss preliminaries concerning
hyperfunction and distribution vectors.
Section~\ref{subsec:h1} introduces, for a general Lie group
$G$, the topology 
on the space $\cH^\omega$ of analytic vectors of a unitary
representations, its dual space $\cH^{-\omega}$
and the $G$-representations on these spaces.
In Section~\ref{subsec:h2} we specialize to $G = \R$,
for which very explicit descriptions of these spaces can
be given in terms of spectral theory.
In Section~\ref{subsec:h3} we further 
discuss distribution vectors for $G = \R$.
We focus on the properties of the spaces
$\cH^{-\omega}_{\rm KMS}$ and $\cH^{-\infty}_{\rm KMS}$
(cf.\ \eqref{eq:kms-spaces-intro}). 

For $G = \R$ our main result on
  hyperfunction vectors is Theorem~\ref{thm:closed-hyperfunc}, 
  asserting that the subspace $\cH^{-\omega}_{\rm KMS}$ of
  $\cH^{-\omega}$ (cf.~\ref{eq:kms-spaces-intro}) 
  is closed and can be characterized as the annihilator of
  a concrete real subspace of $\cH^\omega$ with respect to the pairing
  induced by $\Im \la \cdot,\cdot \ra$.
Further, 
  Theorem~\ref{thm:E.4} provides a characterization of
  distribution vectors that are limits of analytic vectors
  in terms of asymptotic growths of the norm on curves
  $t \mapsto e^{tH}v$, where $v \in \cH$ and $H = H^*$
  if the selfadjoint generator of the unitary one-parameter group.

\subsection{The space of analytic vectors}
\mlabel{subsec:h1}

In this subsection we briefly discuss the space of analytic
vectors of a unitary representation of a Lie group.
For analytic vectors of more general representations, we refer to
\cite{GKS11}. 

Let $(U,\cH)$ be a unitary representation of
the connected real Lie group $G$. We write
\[ \cH^\omega = \cH^\omega(U)\subeq \cH \]
for the space of {\it analytic vectors}, i.e.,
those $\xi\in \cH$ for which the orbit map
$U^\xi \colon G \to \cH, g \mapsto U(g)\xi$,  is analytic.

To endow $\cH^\omega$ with a locally convex topology,
we specify subspaces $\cH^\omega_{V}$ by open convex  {$0$-neighborhoods}
$V \subeq \g$ as follows. Let $\eta_G \colon G \to G_\C$ denote the universal
complexification of $G$ and assume that $\eta_G$ has discrete kernel
(this is always the case if $G$ is semisimple). 
We assume that $V$ is so small that the map
\begin{equation}
  \label{eq:eta-g-v}
 \eta_{G,V} \colon G_V := G \times V
 \to G_\C, \quad (g,x) \mapsto \eta_G(g) \exp(ix)
\end{equation}
is a covering. Then we endow $G_V$ with the unique complex manifold
structure for which $\eta_{G,V}$ is holomorphic.

We now write $\cH^\omega_V$ for the set of those analytic vectors
$\xi$ for which the orbit map $U^\xi\colon G \to \cH$ extends to a
holomorphic map
\[ U^\xi_V \colon G_V\to \cH.\]
As any such extension is $G$-equivariant by 
uniqueness of analytic continuation, it must have the form
\begin{equation}
  \label{eq:orb-map}
 U^\xi_V(g,x) = U(g) e^{i\partial U(x)} \xi \quad \mbox{ for }  \quad
 g \in G, x \in V,
\end{equation}
so that
$\cH^\omega_V \subeq \bigcap_{x \in V} \cD(e^{i\partial U(x)}).$ 
The following lemma shows that we even have equality.

\begin{lem} \mlabel{lem:charh-om-v}
If $V \subeq \g$ is an open convex $0$-neighborhood
  for which \eqref{eq:eta-g-v} is a covering, then 
$\cH^\omega_V = \bigcap_{x \in V} \cD(e^{i\partial U(x)}).$ 
\end{lem}

\begin{prf} It remains to show that each
  $\xi \in \bigcap_{x \in V} \cD(e^{i\partial U(x)})$ is contained
  in $\cH^\omega_V$. For that, we first observe that the
  holomorphy of the functions $z \mapsto e^{iz \partial U(x)}v$ on a
  neighborhood of the closed unit disc in $\C$ implies that the
  $\cH$-valued power series
  \[ f_\xi(x) := \sum_{n = 0}^\infty \frac{i^n}{n!} \partial U(x)^n \xi \]
  converges for each $x \in V$.
  Further, \cite[Thm.~1.1]{Go69} implies
  that $\xi \in \cH^\infty$, so that the functions
  $x \mapsto \partial U(x)^n \xi = \dd U(x)^n \xi$
  are homogeneous $\cH$-valued polynomials (cf.\ \cite{BS71}). Thus
  \cite[Thm.~5.2]{BS71} shows that the above series
  defines an analytic function   $f_\xi \colon V\to \cH$.  
  It follows in particular that $\xi$ is an analytic vector,
  and the map
  \[ U^\xi_V \colon  G_V \to \cH, \quad
  (g, x) \mapsto   U^\xi (g,x):=  U(g) e^{i \partial U(x)}\xi \]
is defined. It is clearly equivariant.
We claim that it is holomorphic. As it is locally bounded,
it suffices to show that, for each $\eta \in \cH^\omega$, the function
\[ f \colon G_V \to \C, \quad f(g,x) := \la \eta, U^\xi(g,x) \ra   \]
is holomorphic (\cite[Cor.~A.III.3]{Ne00}). As
\[ f(g,x) = \la U(g)^{-1}\eta, e^{i \partial U(x)}\xi \ra \]
and the orbit map of $\eta$ is analytic, 
$f$ is real analytic. Therefore it suffices to show that it is
holomorphic on some $0$-neighborhood. This follows from the
fact that it is $G$-equivariant and coincides on some $0$-neighborhood
with the local holomorphic extension of the orbit map of~$\xi$.
Here we use that, for $x,y \in \g$ sufficiently small,
the holomorphic extension $U^\xi$ of the $\xi$-orbit map satisfies 
\[ U^\xi(\exp(x * iy)) = U(\exp x) U^\xi(\exp i y)
  = U(\exp x) f_\xi(y) = U^\xi_V(\exp x, y),\]
where $a * b = a + b + \frac{1}{2}[a,b] + \cdots$ denotes the
Baker--Campbell--Hausdorff series. 
\end{prf}

We topologize the space
$\cH^\omega_V$ by identifying it with 
$\cO(G_V, \cH)^G$, the Fr\'echet space of $G$-equivariant holomorphic maps
$F \colon G_V \to \cH$, endowed with the Fr\'echet topology of
  uniform convergence on compact subsets. Now
$\cH^\omega = \bigcup_V \cH^\omega_V$, 
and we topologize $\cH^\omega$ as the locally convex direct limit
of the Fr\'echet spaces $\cH^\omega_V$.
If the universal complexification $\eta_G \colon  G \to  G_\C$
  is injective, it is easy to see that
  we thus obtain the same topology as in \cite{GKS11}. Note that,
for any monotone
basis  $(V_n)_{n \in \N}$ of convex $0$-neighborhoods in $\g$, we
then have
\[ \cH^\omega \cong \indlim \cH^\omega_{V_n},\]
so that $\cH^\omega$ is a countable locally convex limit of
Fr\'echet spaces. As the evaluation maps
\[ \cO(G_V, \cH)^G \to \cH, \quad F \mapsto F(e,0) \] 
are continuous, the inclusion
$\iota \colon \cH^\omega \to \cH$ 
is continuous. 

\begin{rem} Alternatively, one could also use the spaces
  $\cH^{\omega,b}_V$ of those vectors for which the extended orbit
  maps are bounded. They correspond to the spaces
  $\cO(G_V, \cH)^{b,G}$ of $G$-equivariant bounded holomorphic maps,
  which are Banach spaces.

  If $V_1 \subeq V_2$ is relatively compact, then
  $\cH^\omega_{V_2} \subeq \cH^{\omega,b}_{V_1}$, and this implies 
that also
\[ \cH^\omega \cong \indlim \cH^{\omega,b}_{V_n},\]
describing $\cH^\omega$ is a locally convex direct limit of Banach spaces.
\end{rem}

We write $\cH^{-\omega}$ for the space  of continuous antilinear functionals 
$\eta\colon \cH^\omega \to \C$ (called {\it hyperfunction vectors})
and
\[ \la \cdot, \cdot \ra \colon \cH^\omega \times \cH^{-\omega} \to \C \]
for the natural sesquilinear pairing that is linear in the second argument.
We endow $\cH^{-\omega}$ with the weak-$*$-topology. 
We then have natural continuous inclusions
\[ \cH^\omega \into \cH \into \cH^{-\omega}.\] 

Our specification of the topology on $\cH^\omega$
differs from the one 
\cite{GKS11} because we do not want to assume that
the universal complexification $\eta_G \colon G \to G_\C$ is injective, 
but both constructions define the same topology.
Moreover,  the arguments in \cite{GKS11} apply with minor changes to
general Lie group.

Recall that a vector $v\in \cH$ is {\it smooth} if
the orbit map $g\mapsto U^v(g) = U(g) v$ is smooth. Denote
by  $\cH^\infty$ the space
of smooth vectors. The group $G$ acts smoothly on $\cH^\infty $ and we also have
a representation of $\fg$ on $\cH^\infty$ given by
$\dd U(x)v= \frac{d}{dt}\big|_{t=0}U(\exp tx)v$.
The topology is defined by the seminorms $\| \dd U(x)u\|,
x \in \cU(\g)$. The space of {\it distribution vectors},
denoted by $\cH^{-\infty}$, is the conjugate linear dual of $\cH^\infty$. We have
\[\cH^{\omega} \subeq \cH^\infty \subeq \cH \subeq \cH^{-\infty}
  \subeq \cH^{-\omega},\]
where all inclusions are continuous.

\begin{prop} \mlabel{prop:2.3}
  The topology on $\cH^\omega$ has the following properties:
  \begin{itemize}
  \item[\rm(a)] The representation $U^\omega \colon G \to \cL(\cH^\omega)$
    of $G$ on $\cH^\omega$ defines a continuous action. 
  \item[\rm(b)] The orbit maps of $U^\omega$ are analytic. 
  \item[\rm(c)] If $B\subeq G$ is compact, then $U^\omega(B) \subeq
    \cL(\cH^\omega)$ is equicontinuous. 
  \item[\rm(d)] $\cH^\omega$ is complete.
  \item[\rm(e)] For $g \in G$ and a convex $0$-neighborhood $V \subeq \g$,
    we have $U(g) \cH^\omega_V = \cH^\omega_{\Ad(g)V}.$
  \item[\rm(f)] For $\phi \in C_c(G,\C)$
    and $U(\phi)= \int_G \phi(g)\, dg$, we have
    $U(\phi)\cH^\omega \subeq \cH^\omega$, the operator 
\[ U^\omega(\phi) := U(\phi)\res_{\cH^\omega} \colon \cH^\omega \to
\cH^\omega \]
is continuous, and it has a weak-$*$-continuous adjoint 
\[ U^{-\omega} \colon \cH^{-\omega} \to \cH^{-\omega}, \quad
U^{-\omega}(\phi) \eta := \eta \circ U^\omega(\phi^*)\]
that extends the operator $U^{-\infty}(\phi) =\int_G \phi(g) U^{-\infty}(g)\, dg$.
  \end{itemize}
\end{prop}

\begin{prf} (a), (b) follow from \cite[Prop.~3.4]{GKS11}
  and (c) from its proof. Further (d) follows from \cite[Prop.~3.7]{GKS11}.

  \nin (e) follows easily from the relation
  \[ U^{U(g)\xi}(h,x)
  = U(h) e^{i\partial U(x)} U(g) \xi
  = U(hg) e^{i\partial U(\Ad(g)^{-1} x)} \xi
  = U^\xi(hg, \Ad(g)^{-1} x) \]
(see \eqref{eq:orb-map} for the notation). 

  \nin (f) For $v \in \cH^\omega$, the orbit map
  $U^v \colon G \to \cH^\omega$ is continuous and
  $\cH^\omega$ is complete by (b) and (d). Therefore the weak integral
  \[ U(\phi)v := \int_G \phi(g) U(g) v\, dg \]
  exists in $\cH^\omega$, which means that
  $U(\phi)\cH^\omega \subeq \cH^\omega$.

  To see that $U^\omega(\phi)$ is continuous, let
  $B := \supp(\phi)$. Then $U^\omega(B)$ is equicontinuous by (c),
  i.e., for any $0$-neighborhood $W_1 \subeq \cH^\omega$
  (w.l.o.g.\ absolutely convex and closed),
  there exists a {$0$-neighborhood} $W_2 \subeq \cH^\omega$
  with $U^\omega(B) W_2 \subeq W_1$. As $W_1$ is closed and absolutely
  convex, it follows that
  \[ U^\omega(\phi)W_2 \subeq \|\phi\|_1 \cdot W_1\]
  and hence that $U^\omega(\phi)$ is continuous.

  This implies in particular the existence of the weak-$*$-continuous
  operator $U^{-\omega}(\phi)$. To see that this operator extends the operator 
  \[ U^{-\infty}(\phi) \colon \cH^{-\infty} \to \cH^{-\infty}, \quad
  U^{-\infty}(\phi) \eta = \int_G \phi(g) U^{-\infty}(g)\eta\, dg,\]
  we evaluate it on some $\eta \in \cH^{-\infty} \subeq \cH^{-\omega}$.
  So let $\xi \in \cH^\omega \subeq \cH^\infty$. Then
  \begin{align*}
  \la \xi, U^{-\omega}(\phi)\eta \ra
&=  \int_G  \phi(g) \la \xi, U^{-\omega}(g)\eta \ra\, dg 
=  \int_G  \phi(g) \la \xi, U^{-\infty}(g)\eta \ra\, dg \\
&=  \int_G  \phi(g) \la U(g^{-1})\xi, \eta \ra\, dg 
=   \la U(\phi^*)\xi, \eta \ra 
=   \la \xi, U^{-\infty}(\phi)\eta \ra.
  \end{align*}
  We conclude that
  $U^{-\infty}(\phi)\eta\res_{\cH^\omega} = U^{-\omega}(\phi)\eta$,
  i.e., the inclusion $\cH^{-\infty} \into \cH^{-\omega}$
  maps  $U^{-\infty}(\phi)\eta$ to $U^{-\omega}(\phi)\eta$. 
  Here we use that $\cH^\omega$ is dense in $\cH^\infty$
  (Theorem~\ref{thm:anal-dense-inf} below).
\end{prf}

\begin{thm} \mlabel{thm:anal-dense-inf}
  The subspace $\cH^\omega$ is dense in the Fr\'echet space
  $\cH^\infty$. 
\end{thm}

The density of $\cH^\omega$ in $\cH$
for finite-dimensional Lie groups follows from
Nelson (\cite[Thm.~4]{Nel59}) and from 
G\aa{}rding (\cite{Ga60}) by convolution with heat kernels 
(cf.\ \cite[Sect.~4.4.5]{Wa72}). However, the
theorem does not follow from this weaker result. 

\begin{prf}
Let $x_1, \ldots, x_n$ be a linear basis of $\g$ and
$\Delta_N := \sum_{j = 1}^n x_j^2 \in \cU(\g)$ 
    be the {\it Nelson Laplacian} in the universal enveloping algebra.
  For the positive selfadjoint operator $A := (\1 - \oline{\dd U(\Delta_N)})^{1/2}$,
  we have
  \[ \cH^\omega(A)  = \cH^\omega \subeq \cH^\infty = \cH^\infty(A)
= \bigcap_{n \in \N} \cD(A^n),\]
where $\cH^\omega(A)$, respectively $\cH^\infty (A)$, denotes the space of
analytic, respectively smooth vectors for
the unitary one-parameter group $(e^{it A})_{t \in \R}$.
Here the first equality is Goodman's Theorem
  (\cite[Thm.~2.1]{Go69}) 
  and the second equality follows
  from \cite[Thm.~3, Cor.~9.3]{Nel59} because
  $A$ and $A^2$ have the same set of smooth vectors.
So it suffices to show that $\cH^\omega(A)$ is dense in
the Fr\'echet space $\cH^\infty(A)$. 
This reduces the problem to the case of $1$-parameter groups.

\nin{\bf Case $G = \R$:} Let $U_t = e^{itA}$ be a unitary one-parameter group. 
We consider the normalized Gaussians 
$\gamma_n(t) := \sqrt{\frac{n}{\pi}} e^{-n t^2}$ 
  and recall that 
  $v_n := U(\gamma_n) v \in \cH^\infty$ converges to $v$ in the Fr\'echet space
  $\cH^\infty$ (cf.~\cite[Prop.~3.3(ii) and \S 5]{BN23}).   The relation
  \[ e^{izA} U(\gamma_n) v = U(\delta_z * \gamma_n) v  \]
  and the fact that the map $ \C \to L^1(\R),
  z \mapsto \delta_z * \gamma_n$ 
  is holomorphic shows that $v_n$ is an analytic vector.
  Hence  the assertion follows.

\end{prf}

\begin{rem} If $(\pi, \cH)$ is irreducible and $G$ is semisimple, then
  the subspace $\cH^{[K]}$ of $K$-finite vectors is contained in
  $\cH^\omega$ (\cite{HC53}). 
  The density of
  $\cH^{[K]}$ in $\cH^{\infty}$ with respect to
the Fr\'echet topology on $\cH^\infty$
follows from the Peter--Weyl Theorem, applied
to the continuous action of $K$, resp., the compact group
$\oline{U(K)\T}$ on $\cH^\infty$ (cf.~\cite[Cor.~4.5]{Ne10}; 
see also
  \cite{HC53}).
\end{rem}

\subsection{Analytic vectors for one-parameter groups}
\mlabel{subsec:h2}

After the discussion of analytic vectors in the preceding subsection,
we now take a closer look at the very special case $G = \R$. 
Let $U_t = e^{itH}$ be a unitary $1$-parameter group on $\cH$,
where $H = H^*$ is selfadjoint. Let $P_H$ denote the
spectral measure of $H$, so that we obtain for
each $v \in \cH$ a finite positive Borel measure
$\mu_v := \la v, P_H(\cdot) v \ra$ on~$\R$. 
For $v \in \cH$, we then have
\[ \la v, U_t v \ra = \int_\R e^{it\lambda}\, d\mu_v(\lambda) \]
and $v$ is smooth if and only if the measure $\mu_v$ possesses moments of all order.
Likewise $v$ is analytic if and only if there exists an $r > 0$
with \[ \int_\R e^{\pm r \lambda}\, d\mu_v(\lambda) < \infty.\] 
and then $U^v(x + iy) = e^{(ix - y)H}v$ is defined for
  $|y| < r/2$ by spectral calculus (\cite[Lemma~A.2.5]{NO18}). 
Accordingly, we have
\begin{equation}
  \label{eq:5.1}
  \cH^\omega = \bigcup_{r > 0} \cH^\omega(r), \quad \mbox{ where }\quad
  \cH^\omega(r) = \{v \in \cH \colon v \in \cD(e^{\pm r H})\}.
\end{equation}
The space $\cH^\omega(r)$
consists of all elements whose orbit map extends to
the closed strip $\oline{\cS_{\pm r}}$, where
  \[ \cS_{\pm r} = \{ z \in \C\colon |\Im z|  < r\} \] 
(\cite[Lemma~A.2.5]{NO18}),   
and we topologize $\cH^\omega(r)$ by the embedding
\begin{equation}
  \label{eq:anal-graf-emb}
  \cH^\omega(r) \into \cH^{\oplus 2}, \quad
  v \mapsto (e^{r H}v, e^{-rH}v).
\end{equation}

\begin{rem} To see that this topology coincides with the one
  introduced in Section~\ref{subsec:h1} for general Lie groups, we
  use the embedding
  \[ \iota \colon \cH^\omega(r) \to \cO_\partial(\cS_{\pm r}, \cH)^\R, \quad
  v \mapsto U^v, \quad U^v(z) = e^{izH}v \]
  to the space   $\cO_\partial(\cS_{\pm r}, \cH)^\R$
  of equivariant continuous functions
  $\oline{\cS_{\pm r}} \to \cH$ that are holomorphic on the interior.
  All these maps are bounded and satisfy
  \[ \|U^v\|_\infty \leq \max \{ \|e^{rH}v\|, \|e^{-rH}v\|\},\]
  showing that $\iota$ is continuous. Further, $\iota$
  is bijective (\cite[Lemma~A.2.5]{NO18}) and
  $\iota^{-1}(F) = F(0)$ is also continuous.  
\end{rem}

From \eqref{eq:anal-graf-emb} we derive
that the dual space $\cH^{-\omega}(r)$ of $\cH^\omega(r)$
  consists of all
antilinear functionals of the form
\begin{equation}
  \label{eq:hyperfunc-sum}
  \alpha(v) =
  \la e^{rH}v, w_1 \ra 
  + \la e^{-rH}v, w_2 \ra \quad \mbox{ for } \quad w_1, w_2 \in \cH.
\end{equation} 
Informally, we think of $\alpha$ as a sum
$e^{rH} w_1 +  e^{-rH} w_2.$ 
Note that
\begin{equation}
  \label{eq:L.1} \cH^{-\omega} = \bigcap_{r > 0} \cH^{-\omega}(r)
\end{equation}
carries a natural Fr\'echet space structure.
Further \eqref{eq:hyperfunc-sum} implies that, for
$\eta \in \cH^{-\omega}$ and any $t > 0$, we have
$e^{-t|H|} \cH \subeq \cH^\omega$
because the functions $\lambda \mapsto e^{\pm r\lambda-t|\lambda|}$
  are bounded for $|r| < t$, and
\begin{equation}
  \label{eq:eta-into-h} \eta \circ e^{-t|H|} \in \cH.
\end{equation}

\begin{ex} \mlabel{ex:2.7} Consider $\cH = L^2(\R,\mu)$ and 
  $(U_t f)(\lambda) = e^{it\lambda} f(\lambda)$. We claim that
  \[ \cH^{-\omega} =\{ f \colon \R \to \C \colon f\ \mbox{measurable},  (\forall t > 0)\
  e^{-t|\lambda|}f(\lambda) \in L^2 \}.\]
Let $r > 0$. Then every  $f \in \cH^{-\omega}$ can be written
  as   $\eta = e_r f_1 + e_{-r} f_2$ with $f_1, f_2 \in L^2$  and
  $e_r(\lambda) = e^{r \lambda}$. For $t > r$ we then have
  $e^{-t|\lambda|} \eta \in L^2$.

  If, conversely, $\eta \colon \R \to \C$ is a measurable function with
  $e^{-t|\lambda|} \eta \in L^2$ for all $t > 0$, then we write
  $\eta = \eta_+ + \eta_-$ with $\eta_\pm$ supported on $\pm [0,\infty)$.
    Then $\eta_+ = e_t(e_{-t} \eta_+)$ and $e_{-t} \eta_+ \in L^2$, and
    likewise $\eta_- = e_{-t} (e_{t} \eta_-)$ with $e_{t} \eta_- \in L^2$.
\end{ex}

\begin{lem} \mlabel{lem:k.1}
  Let $v\in \cH^\omega $ and $r > 0$
  such that $v \in \cD(e^{tH})$ for
  $0 \leq t < r$.
  Then the following assertions hold:
  \begin{itemize}
  \item[\rm(a)]   $\eta := \lim_{t \to r} e^{tH} v$
  exists in the space $\cH^{-\omega}$ with respect to the
  weak-$*$-topology. 
\item[\rm(b)] The orbit map of $\eta \in \cH^{-\omega}$ extends to a
  weak-$*$-bounded, weak-$*$-continuous map 
 $ U^\eta \colon \oline{\cS_r} \to \cH^{-\omega}.$ 
  \end{itemize}
\end{lem}

\begin{prf} (a) Let  $w \in \cH^\omega  = \bigcup_{r > 0} \cH^\omega(r)$.
  Then there exists an $\eps \in (0,r)$ with 
  $w \in \cD(e^{\eps H})$. Let 
$w_\eps := e^{\eps H}w.$ 
  Then
  \[ \la w, e^{t H} v \ra
   = \la e^{-\eps H} w_\eps, e^{tH} v \ra 
   = \la w_\eps, e^{(t-\eps)H} v \ra,\]
     and since the orbit map 
     \[ U^v \colon \{ z \in \C \colon 0 < \Re z < r \}  \to \cH, \quad 
     \quad z \mapsto e^{z H} v \]
     is holomorphic, we obtain
     \[ \lim_{t \to r} \la w, e^{t H} v \ra
     = \la w_\eps, e^{(r-\eps)H} v \ra.\]
        On $\cH^\omega(\eps)$ we  obtain
   \[ \eta(w) \leq \|w_\eps\| \cdot \|e^{(r-\eps)H} v\|, \]
   which implies that $\eta \in \cH^{-\omega}(\eps)$. 
   This shows that $\eta \in \cH^{-\omega}$ and that 
     $e^{tH}v \to \eta$ in the weak-$*$-topology of
     $\cH^{-\omega}$.

   \nin (b) For $z = x + i y$ with $0 \leq y \leq r$, we put
   \[ U^\eta(z) :=
   \begin{cases}
     U_x^{-\omega} \eta & \text{ for } y = 0 \\
     e^{(r+i z)H} v  & \text{ for } 0 < y \leq r.
   \end{cases}\] 
 This map is a holomorphic $\cH$-valued on the open strip
 $\{ 0 < y < r\}$,  weak-$*$-continuous on the semi-closed strip
 $\{ 0 \leq y < r\}$ and 
 bounded on any strip of the form
 $\{ 0 \leq y < r'\}$ for $0 < r' < r$.
      To verify (b), let $w \in \cD(e^{\eps H})$, 
   $w_\eps := e^{\eps H}w$ as in (a), 
   and $r ' \in (r-\eps, r)$ for $0 < \eps < r$. Then
   \begin{align*}
 \la w, U_t^{-\omega} \eta \ra
&   =  \la U_{-t} w, \eta \ra
   =  \la e^{-itH} e^{-\eps H} w_\eps, \eta \ra
   =  \la e^{-itH} w_\eps, e^{(r-\eps) H} v \ra\\
&   =  \la w_\eps, e^{(r-\eps + i t) H} v \ra
  = \lim_{z \to t}  \la w_\eps, e^{(r - \eps + i z) H} v \ra
= \lim_{z \to t}  \la w, e^{(r + i z) H} v \ra\\
&= \lim_{z \to t}  \la w, U^\eta(z) \ra
   \end{align*}
   shows that $U^\eta$ is also continuous in $\R$.
   The boundedness  $z \mapsto \la w, U^\eta(z) \ra$
   follows from
   \[ |\la w, U^\eta(z) \ra|
   = |\la w_\eps, e^{(r - \eps + i z) H} v \ra|
   \leq \|w_\eps\| \cdot \max \{ \|e^{(r-\eps)H}v\|, \|v\|\}\]
 for $0 \leq \Im z \leq r-\eps$.
\end{prf}

\begin{defn} \mlabel{def:kms-subspace}
  Let $(V_t)_{t\in \R}$ be a strongly continuous
  one-parameter group  of topological isomorphisms of a locally
  convex space $\cX$ and $J \colon \cX \to \cX$ a continuous antilinear
  involution. Then we write
  $\cX_{\rm KMS} \subeq \cX$ for the real linear subspace of those
  elements $\xi \in \cX$ for which the orbit map
  $V^\xi \colon \R\to \cX, t \mapsto V_t\xi$, extends to a
  continuous map $V^\xi \colon \oline{\cS_\pi} \to \cX$, holomorphic on $\cS_\pi$,
that satisfies
  \[ V^\xi(\pi i) = J \xi,\quad \mbox{ resp., } \quad
    V^\xi(\pi i + t) = J V^\xi(t) \quad \mbox{ for } \quad t \in \R.\]
\end{defn}

\begin{lem} \mlabel{lem:fixed-gen}
  If $\xi \in \cX_{\rm KMS}$, then $V^\xi(\pi i/2)$ is fixed by $J$.   
\end{lem}

\begin{prf} The two maps
  $V^\xi \colon \oline{\cS_\pi} \to \cX$ and 
$z \mapsto J V^\xi(\pi i + \oline z)$ 
are both continuous on $\oline{\cS_\pi}$ and holomorphic on $\cS_\pi$.
As they coincide on $\R$,
uniqueness of analytic continuation shows they coincide,  i.e., 
  \[  V^\xi(\pi i + \oline z)  = J V^\xi(z) \quad \mbox{ for } \quad
    z \in \oline{\cS_\pi}.\]
It follows in particular that  $J$ fixes $V^\xi(\pi i/2)$.   
\end{prf}

\begin{prop} \mlabel{prop:2.9}
{\rm(a)}  If $J$ is a conjugation on $\cH$ commuting with $U_\R$,  and
  $v \in \cH^J$ is such that the orbit map $U^v \colon \R \to \cH$
  extends to a holomorphic map $\cS_{\pm \pi/2} \to\cH$, then
  \[ \eta := \lim_{t \to \pi/2} U^v(-it) \in \cH^{-\omega}_{\rm KMS},\]
where $\cH^{-\omega}_{\rm KMS}$ is defined around
    \eqref{eq:kms-spaces-intro}.

\nin {\rm(b)}  If, conversely, $\eta \in \cH^{-\omega}_{\rm KMS}$, then
  \[ U^\eta(z)\in \cH \quad \mbox{ for } \quad z \in \cS_\pi,\]
  and, for $v := U^\eta(\pi i/2) \in \cH^J$, we have
$ \eta = \lim_{t \to \pi/2} U^v(-it).$
\end{prop}

\begin{prf} (a)  Lemma~\ref{lem:k.1} shows that the holomorphic map
$ U^v  \colon \cS_{\pm \pi/2} \to \cH \subeq \cH^{-\omega}$ 
  extends to a weak-$*$-continuous, weak-$*$-bounded map
$U^v  \colon \oline{\cS_{\pm \pi/2}} \to \cH^{-\omega}.$ 
  For $\eta := U^v(-\pi i/2)$, we then have
  \[ U^\eta(z) = U^v(z - \pi i/2) \quad \mbox{ for } \quad
  z \in \oline{\cS_\pi}.\]
  As $J$ commutes with $U_\R$ and fixes $v$, we have
  \[ J U^v(z) = U^v(\oline z)
    \quad \mbox{ and thus } \quad
  J U^\eta(z)
  = U^v(\oline z + \pi i/2)
  = U^\eta(\oline z + \pi i).\]

 \nin (b) For the converse, let $\eta \in \cH^{-\omega}$.
  In view of \eqref{eq:eta-into-h},
  $\eta_t := \eta \circ e^{-t|H|} \in \cH$. We claim that $\eta$ vanishes on
  all analytic vectors $\xi$ orthogonal to the cyclic subspace 
  $\cK$ generated by these vectors~$\eta_t$. In fact, there exists a
  $t > 0$ with $e^{t|H|}\xi \in \cH$, so that
  \[ \eta(\xi) = (\eta \circ e^{-t|H|})(e^{t|H|} \xi)
    \in \la \cK, \cK^\bot \ra = \{0\}.\]
  By analytic continuation, the same holds for every
  $U^\eta(z)$, $z \in \cS_\pi$, hence in particular for
  $\eta' := U^\eta(\pi i/2)$, which is fixed by~$J$
  (Lemma~\ref{lem:fixed-gen}). Then, for each $t > 0$, 
  the functional   $\eta' \circ e^{-t|H|} \in \cH$ is fixed by~$J$. 
  
  We may therefore assume that the representation
  $(U,\cH)$ is cyclic with a $J$-fixed generator.
  By Bochner's Theorem it is therefore of the form $\cH = L^2(\R,\mu)$ 
  with
  \[ (U_t f)(\lambda) = e^{it\lambda} f(\lambda)
    \quad \mbox{ and } \quad (Jf)(\lambda) = \oline{f(-\lambda)},\]
  where  we use that the $J$-fixed generator corresponds to the constant
function $1$, so that
$J1 = 1$ and $J e_{it} = J(U_t 1) = U_t 1 = e_{it}$ implies the
asserted form of~$J$. In particular, the measure $\mu$ is symmetric.
Now $\eta$ is represented by a function $\eta \colon \R\to \C$ with the
  property that $e^{-t|\lambda|} \eta(\lambda)$ is $L^2$ for every $t > 0$
  (Example~\ref{ex:2.7}),   and the KMS condition turns into
  \[ \oline{\eta(-\lambda)} = (J\eta)(\lambda) = e^{-\pi \lambda}
  \eta(\lambda).\]
  We write $\eta = \eta_+ + \eta_-$ with $\eta_\pm$ supported
  in $\pm [0,\infty)$.
    For $z = x + iy \in \cS_\pi$ we then have $0 < y < \pi$, so that
    we have for any $r \in (0,y)$: 
    \[ e_{-y} \eta_+
    = e_{-y} e_r e_{-r} \eta_+
    = e_{r-y} \underbrace{e_{-r} \eta_+}_{\in L^2} \in L^2.\]
    We likewise find for $\lambda \leq 0$: 
    \[ (e_{-y} \eta_-)(\lambda)
    = e^{-\lambda y} \eta_-(\lambda)
    = e^{\lambda (\pi -y)} e^{-\pi \lambda}\eta_-(\lambda)
    = e^{\lambda (\pi-y)} \oline{\eta_+(-\lambda)}.\]
    As $\pi - y > 0$ and $\eta_+ \in \cH^{-\omega}$,
    it follows that $e_{-y} \eta_- \in L^2$. This shows that 
    $U^\eta(z) = e_{iz} \eta \in L^2$ for $z \in \cS_\pi$.

  For $v := U^\eta(\pi i/2)$ we then have
      $U^v(-it) = U^\eta((\frac{\pi}{2}- t)i)$, so that the last assertion
      follows from the continuity of $U^\eta$ on the closed strip
      $\oline{\cS_\pi}$.       
\end{prf}

The argument in the proof of the preceding proposition
can also  be used to obtain:
\begin{cor} \mlabel{cor:2.12} If $\eta \in \cH^{-\omega}_{\rm KMS}$
  and the element $J\eta \in \cH^{-\omega}$
  defined by $(J\eta)(\xi) := \oline{\eta(J\xi)}$
  is also contained in $\cH^{-\omega}_{\rm KMS}$, then
  \[ \eta \in \cH \quad \mbox{ and } \quad U_t \eta = \eta \quad
    \mbox{ for all } \quad t \in \R.\] 
\end{cor}

\begin{prf} As in the proof of Proposition~\ref{prop:2.9},
  we may assume that
    \[ (U_t f)(\lambda) = e^{it\lambda} f(\lambda)
  \quad \mbox{ and } \quad (Jf)(\lambda) = \oline{f(-\lambda)}.\]
For $\eta \in \cH^{-\omega}_{\rm KMS}$, we then have the relation
$(J\eta)(\lambda) = e^{-\pi \lambda} \eta(\lambda).$ 
The same argument, applied to $J\eta$, shows that
\[ \eta(\lambda) = J(J\eta)(\lambda)
  = e^{-\pi \lambda} (J\eta)(\lambda)
  = e^{-2\pi \lambda} \eta(\lambda).\]
This is only possible if $\eta$ is supported in $\{0\}$, i.e.,
if the cyclic representation generated by $\eta$ is trivial.
So $\eta = U_t \eta$ for every $t \in \R$ and 
$\eta = U^\eta(\pi i/2) \in \cH$ (Proposition~\ref{prop:2.9}).   
\end{prf}

\begin{rem} The containment in the
    complex subspace $\cH^{-\omega}_{\rm KMS} + i \cH^{-\omega}_{\rm KMS}$
    need not give any information.
    In fact, if $U$ is norm-continuous, then $\cH^{-\omega} = \cH$
    and
    $\cH^{-\omega}_{\rm KMS} + i \cH^{-\omega}_{\rm KMS} = \cH$.
  \end{rem}

  \begin{rem} \mlabel{rem:ana-molli}
    We now introduce a construction that will be used in the
    proof of the theorem below. 
    
\nin  (a) On $\R$ we consider the analytic
  $L^1$-functions
  \[ \gamma_a(t) := \frac{\sqrt{a}}{\sqrt\pi}e^{-at^2}, \quad a > 0.\]
  The (left) translation action
  \[ (\lambda_x \phi)(t) := \phi(t-x) \]
  defines a continuous isometric action of $\R$ on $L^1(\R)$,
  and the functions $\gamma_a$ are entire vectors, i.e.,
  their orbit map extends to a holomorphic map
  \[ \lambda^{\gamma_a} \colon \C \to L^1(\R), \quad
    \lambda^{\gamma_a}(z)(t) = \frac{\sqrt{a}}{\sqrt\pi} e^{-a(t-z)^2}
    = \frac{\sqrt{a}}{\sqrt\pi}e^{-at^2 + a 2 t z - a z^2}.\]

  \nin (b) Let $(U_t)_{t \in \R}$ be a unitary one-parameter group.
  Then the continuity of the bilinear map
  \[ L^1(\R) \times \cH \to \cH, \quad
  (\psi, \xi) \mapsto U(\psi) \xi \]
  implies that $U(\gamma_a)\cH \subeq \cH^\cO \cong \cO(\C,\cH)^\R$
  (the space of {\it entire vectors}), so that we obtain a linear map
  \[ U(\gamma_a) \colon \cH \to \cH^\cO.\]
  We also note that the map
  \[ \C \times \cH \to \cH,\quad
  (z, \xi) \mapsto U(\lambda^{\gamma_a}(z)) \xi \]
  is holomorphic, so that we obtain a natural map
  \[ \Upsilon \colon \cH \to \cO(\C,\cH), \quad
  \Upsilon(\xi)(z) := U(\lambda^{\gamma_a}(z)) \xi.\]
  As the orbit map $\lambda^{\gamma_a}$ in $L^1(\R)$ is locally bounded,
  $\Upsilon$ is continuous.
  Identifying $\cH^\cO$ with the closed subspace $\cO(\C,\cH)^\R
  \subeq \cO(\C,\cH)$ of equivariant maps (by evaluation in $0$),
  it follows that 
  $U(\gamma_a) \colon \cH \to \cH^\cO$ is continuous.

  \nin (c) For $r > 0$, the natural map 
  $\cH^\cO \to \cH^\omega(r)$ is continuous
  (cf.\ \eqref{eq:5.1}). Hence, for each $r > 0$, the linear map
  \[ U(\gamma_a) \colon \cH \to \cH^\omega(r)  \]
  is continuous by (b).
  As a consequence, it induces a continuous linear map 
  \[ U(\gamma_a) \colon \cH \to \cH^\omega. \]
  So it has a weak-$*$-continuous adjoint map 
  \begin{equation}
    \label{eq:min-om-adj}
 U^{-\omega}(\gamma_a) = U(\gamma_a)^\sharp \colon \cH^{-\omega} \to \cH.
  \end{equation}
\end{rem}

\begin{rem} For the following theorem, we recall from the introduction
  that standard subspaces
  $\sV \subeq \cH$ are specified by pairs $(\Delta_\sV, J_\sV)$ of a positive 
  selfadjoint operator $\Delta_\sV$ and a conjugation $J_\sV$ satisfying
  $J_\sV \Delta_\sV J_\sV = \Delta_\sV^{-1}$ via
  \[ \sV = \Fix(J_\sV\Delta_\sV^{1/2}).\]
  These operators are uniquely determined by $\sV$ 
  as the factors in the polar decomposition $T_\sV = J_\sV \Delta_\sV^{1/2}$ of
  the (closed) {\it Tomita operator}
  \[ T_\sV \colon \sV + i \sV \to \cH,\quad x + iy \mapsto x - iy \quad \mbox{
      for } \quad x,y \in \sV.\]

For the dual standard subspace
\begin{equation}
  \label{eq:v'}
  \sV' = \{ w \in \sH \colon \Im \la w, \sV \ra = \{0\}\}
\end{equation}
  we have $\Delta_{\sV'} = \Delta_\sV^{-1}$ and $J_{\sV'} = J_\sV$.
\end{rem}

\begin{thm} \mlabel{thm:closed-hyperfunc} 
  Let $J$ be a conjugation commuting with $U_\R$ and
  $\sV \subeq \cH$ be the standard subspace with
  \[ J_\sV := J \quad \mbox{ and } \quad
  \Delta_\sV := e^{-2\pi H}.\]
  Then
  \[ \cH^{-\omega}_{\rm KMS} =
  \{ \eta \in \cH^{-\omega} \colon (\forall v \in \sV' \cap
  \cH^\omega)\ \eta(v) \in \R\}.\]
\end{thm}

This theorem implies
in particular that $\cH^{-\omega}_{\rm KMS}$ is weak-$*$-closed. 

\begin{prf}  ``$\subeq$'': {\bf Step 1:} For each entire vector
  $w \in \cH^\cO \cong \cO (\C,\cH )^\R$ we show that 
  \begin{equation}
    \label{eq:anal-rel1}
    \la w, U^\eta(z) \ra = \la U^w(-\oline z), \eta \ra
    \quad \mbox{ for }\quad z \in \oline{\cS_\pi}.
  \end{equation}
  First we observe that the orbit map
  $U^w \colon  \C \to \cH$ is holomorphic and $\R$-equivariant.
  Hence  $U^w(z + z') = U^{U^w(z)}(z')$ shows that
  $U^w$ is also holomorphic as a map
  $\C \to \cH^\cO  $, hence in particular as a map
  to $\cH^\omega$. Therefore both sides of \eqref{eq:anal-rel1}
  are continuous on $\oline{\cS_\pi}$ and holomorphic on the interior
  (the expression on the left by assumption and the expression on the
  right  by the preceding argument).
  As they coincide on $\R$, the claim follows.

  \nin {\bf Step 2:} For $\eta \in \cH^{-\omega}_{\rm KMS}$,
  $v \in \sV' \cap  \cH^\omega$ and
    $\gamma_n$ as in Remark~\ref{rem:ana-molli}, 
      we have
  $v_n := U(\gamma_n)v  \in \cH^\cO \cap \sV'$. Hence Step 1 implies that
  \[ \la U^{v_n}(\oline z), U^\eta(z) \ra
  = \la U^{v_n}(\oline z - \oline z), \eta \ra 
  = \la v_n, \eta \ra \quad \mbox{ for } \quad z \in \oline{\cS_\pi}.\]
  Evaluating this relation in $z = \pi i$, we arrive at 
  \[ \eta(v_n) = \la U^{v_n}(-\pi i), U^\eta(\pi i) \ra
  = \la J v_n, J\eta \ra = \oline{\la v_n,\eta\ra}
  = \oline{\eta(v_n)},\]
where we have used that $v_n \in \sV'$ to obtain
  $U^{v_n}(-\pi i) = J v_n$.

  \nin {\bf Step 3:} $\eta(v_n) \to \eta(v)$:
  We have to show that $v_n \to v$ in $\cH^\omega$.
  So let $r > 0$ with $v \in \cD(e^{\pm rH})$.
  It suffices to show that $v_n \to v$ in $\cH^\omega(r)$,
  but this follows from
  \[ e^{\pm rH}v_n
  =   e^{\pm rH} U(\gamma_n) v  
  =  U(\gamma_n)  e^{\pm rH}  v   \to e^{\pm rH}  v.\] 

  Clearly, Steps 2 and 3 imply $\eta(v) \in \R$,
  so that ``$\subeq$'' holds.

  \nin    ``$\supeq$'': This direction is less straight forward.
   Let $\eta \in \cH^{-\omega}$ be such that
    $\eta(v) \in \R$ for all $v \in \sV' \cap  \cH^\omega$.
    We consider the sequence $\eta_n := U^{-\omega}(\gamma_n)\eta$ in $\cH$
    (cf.\ \eqref{eq:min-om-adj}). 
    Then
    \[ \eta_n(\sV') = \eta(U(\gamma_n)\sV') \subeq
    \eta(\sV'\cap \cH^\omega) \subeq \R \]
    implies $\eta_n \in \sV$ for every $n \in \N$.
    So we have in particular
    $\eta_n \in \cH^{-\omega}_{\rm KMS}$.

    Further, $\eta_n \to \eta$ in the weak-$*$-topology follows from
    $v_n = U(\gamma_n)v \to v$ in $\cH^{\omega}$ for every $v \in \cH^\omega$
    in the topology of $\cH^\omega$ (see Step 3 above).

    Next we consider the Hilbert space $\cH^\omega(r)$,
    identified with a closed subspace of $\cH^{\oplus 2}$ as in
    \eqref{eq:anal-graf-emb} via
    $v \mapsto (e^{rH}v, e^{-rH}v)$.
    The closedness follows from the fact that the range of this
    map is the graph of the selfadjoint operator $e^{-2rH}$.
    On this Hilbert space we obtain by
    $U^r_t := U_t\res_{\cH^\omega(r)}$ a unitary one-parameter group
    for which the above embedding is equivariant.
    As $\cH^{-\omega}$ is a projective limit of the dual spaces
    $\cH^{-\omega}(r)$, it suffices to verify that the restriction
    $\eta^r := \eta\res_{\cH^\omega(r)}$ is contained in 
    $\cH^{-\omega}(r)_{\rm KMS}$.

    As $J$ commutes with $U_\R$, it also defines a conjugation
    on the spaces $\cH^\omega(r)$ and $\cH^{-\omega}(r)$.
    Let $\sV_r 
= \cH^{-\omega}(r)_{\rm KMS} 
\subeq \cH^{-\omega}(r)$ be the corresponding
    standard subspace. We have to show that $\eta \in \sV_r$.
    As $\eta_n \to \eta$ also holds weakly in the Hilbert space
    $\cH^{-\omega}(r)$ and     $\eta_n \in \sV_r$, we obtain
    $\eta \in \sV_r$.   
\end{prf}

\begin{cor} \mlabel{cor:n5}
  $\cH \cap  \cH_{\rm KMS}^{-\omega} = \sV.$
\end{cor}

\begin{prf} In view of Theorem~\ref{thm:closed-hyperfunc},
  it suffices to observe that 
  $\sV' \cap \cH^\omega$ is dense in~$\sV'$.
  This follows from the fact that
  $U(\gamma_n)v \to v$ for $v \in \sV'$ and 
  $U(\gamma_n)v \in \sV' \cap \cH^\omega$.
\end{prf}

\subsection{Distribution vectors for one-parameter groups}
\mlabel{subsec:h3}

In this section we turn to distribution vectors for
unitary one-parameter groups. We shall characterize them
in terms of approximations by holomorphic extensions
of orbits maps in $\cH$ to strips in~$\C$
and the asymptotics of the norm in boundary points. 

For a  positive Borel measure $\nu$ on $\R$,
we consider its Laplace transform 
\[  \cL(\nu)(t) = \int_\R e^{-tx}\,  d\nu(x).\]
We assume that there exists a $\delta > 0$ with
$\cL(\nu)(t) < \infty$ for $t \in (0,\delta]$.
We are interested in the asymptotics for $t \to 0$.
If $\nu$ is unbounded, then $1 \not\in L^2(\R,\nu)$ and
\[ \lim_{t \to 0} \cL(\nu)(t) = \infty\]
by the Monotone Convergence Theorem
(see \cite[\S V.4]{Ne00} for more general arguments of this type). 


The following proposition describes different types of
asymptotic behavior. 

  \begin{prop}  \mlabel{prop:2.15}
    We consider the measure $\mu_s$ on $[1,\infty)$ with the density $x^{-s}$, $s \in \R$.   This measure is finite if and only if $s > 1$,
    and, for $s \leq 1$,
    the asymptotics of $\cL(\mu_s)$ for $t \to 0$ is given by
    \[ t^{s-1} \ \mbox{ for }\  s < 1 \quad \mbox{ and } \quad
      |\log t|\ \mbox{ for }\  s = 1.\]
  \end{prop}

  \begin{prf} For $s > 1$ we have
    $\lim_{t \to 0} \cL(\mu_s)(t) = \mu_s(\R) > 0.$ 
    For $s < 1$ we have
    \[ \cL(\mu_s)(t)
    = \int_1^\infty e^{-tx}\frac{dx}{x^s}
    = t^{s-1} \int_t^\infty e^{-x}\frac{dx}{x^s} \]
    and therefore
\[\lim_{t \to 0}  t^{1-s} \cL(\mu_s)(t) 
=  \int_0^\infty e^{-x}\frac{dx}{x^s}  \in (0,\infty). \]
For $s = 1$ we obtain
\begin{align*}
 \cL(\mu_1)(t)
&= \int_1^\infty e^{-tx}\frac{dx}{x}
 =  \int_t^\infty e^{-x}\frac{dx}{x}
 = \big| e^{-x} \log(x) \Big]_t^\infty
- \int_t^\infty (-e^{-x}) \log x\, dx \\
& = - e^{-t} \log t + \int_t^\infty e^{-x} \log x\, dx, 
\end{align*}
so that
$\lim_{t \to 0} \frac{\cL(\mu_1)(t)}{|\log t|}= 1$
because $\int_0^\infty e^{-x} \log(x)\, dx$ exists.
  \end{prf}

  \begin{rem}    Suppose
    that $v$ is a $K$-fixed vector in a unitary representation
    of a semisimple Lie group $G$ with
    $K \subeq G$ maximal compact modulo 
      the center and 
    \[ \phi_v(t) := \la v, U(\exp t h) v \ra, \quad v \in \cH^K \]
    for an Euler element $h \in \fa \subeq \fp$.
Writing $\phi_v(t) = \hat\mu_v(t)$ for a finite positive
measure on $\R$ (Bochner's Theorem), we have
\[ \phi_v(z) := \int_\R e^{iz x}\, d\mu_v(x) \quad \mbox{ for } \quad
z \in \cS_{\pm \frac{\pi}{2}}.\]  
The asymptotics of $\phi_v(i(\frac{\pi}{2}-\eps))$
for spherical functions $\phi_v$ on rank-one groups 
found in \cite[Thm.~5.1]{KSt04} are of the form
$|\log\eps|$ and $\eps^{-s}$, $s > 0$.
\end{rem}

  \begin{prop} \mlabel{prop:tempmeas}
    Let $\mu$ be a positive Borel measure for which there exists a
    $\delta > 0$  with 
 \[ \cL(\mu)(t) = \int_\R e^{-tx}\, d\mu(x) < \infty
   \quad \mbox{ for } \quad t \in (0,\delta]. \]
 Then $\mu$ is tempered, i.e., there exists an $n \in \N$
 for which
 \[ \int_\R (1 + \lambda^2)^{-n}\, d\mu(\lambda) < \infty,\] 
 if and only if there exists an $N \in \N$ and $C > 0$ such that
 \begin{equation}
   \label{eq:EN}
   \cL(\mu)(t) \leq C t^{-N} \quad \mbox{ for } \quad 0 < t \leq \delta.
\tag{${\rm E}_{\rm N}$} \end{equation}
  \end{prop}

        \begin{prf}
    Our assumption implies that the interval
    $(-\infty, 1]$ has finite measure and since
      \[\lim_{t \to 0} \int_{-\infty}^1 e^{-tx}\, d\mu(x) =
      \int_{-\infty}^1 \, d\mu(x) \]
      by the Monotone Convergence  Theorem (applied to the integral over $\R_-$),
      we may w.l.o.g.\ assume that $\mu$ is supported by $[1,\infty)$. 

    First we assume that $\mu$ is tempered and that 
    $\int_1^\infty\, \frac{d\mu(x)}{x^N} < \infty$ for
an    $N \in \N$. As
    the function $e^{-tx} x^N$ assumes its maximal value on $\R_+$ for
    $x_0 = \frac{N}{t}$, we have
    \[ \cL(\mu)(t)
    = \int_1^\infty e^{-tx} \, d\mu(x)
    = \int_1^\infty e^{-tx}x^N \frac{d\mu(x)}{x^N} 
    \leq  e^{-N} \frac{N^N}{t^N}  \int_1^\infty \frac{d\mu(x)}{x^N}.\]
  It follows that $\cL(\mu)(t) \leq C t^{-N}$ for some $C > 0$.

    Suppose, conversely, that $\cL(\mu)(t) \leq C t^{-N}$ for some $C > 0$
    and $0 < t \leq \delta$. For $M \in \N$, we consider the measure
    \[ d\mu_M(x) := \frac{d\mu(x)}{x^M} \quad \mbox{ on }\quad [1,\infty).\]
 We have to show that one of these measures is finite.
    Then the Laplace transforms $\cL(\mu_M)$ exist on
    $(0,\delta]$ for some $\delta < 1$ 
    and its $M$-fold derivative is 
    \[ \cL(\mu_M)^{(M)} = (-1)^M \cL(\mu).\]
    In particular, $\cL(\mu_1)' = - \cL(\mu)$.
        First we assume that $N \geq 2$. Then we have
    \begin{align*}
      \cL(\mu_1)(t)
 &  
      = \cL(\mu_1)(\delta) +  \int_t^\delta \cL(\mu)(x)\, dx
      \leq  \cL(\mu_1)(\delta) + C  \int_t^\delta \frac{dx}{x^N} \\
& =  \cL(\mu_1)(\delta) + C \Big[ \frac{-1}{N-1} \frac{1}{x^{N-1}}\Big|_t^\delta \\ 
        & =  \cL(\mu_1)(\delta) - \frac{C}{(N-1) \delta^{N-1}}  + \frac{C}{(N-1) t^{N-1}}
        \leq C_1 \frac{1}{t^{N-1}}
    \end{align*}
    for some $C_1 > 0$ and every $t \in (0,\delta]$. 
    Iterating this argument, we see that there exists $C_{N-1} > 0$ with
    \[ \cL(\mu_{N-1})(t) \leq \frac{C_{N-1}}{t} \quad \mbox{ for }\quad 0 < t \leq \delta.\]

    This leaves us with the case $N = 1$, where
    $\cL(\mu)(t) \leq \frac{C}{t}$. Then
    \[ \cL(\mu_1)(t)
    \leq \cL(\mu_1)(\delta) + C \int_t^\delta \frac{dx}{x}
    \leq \cL(\mu_1)(\delta) + C (\log \delta - \log t)
    \leq  C'|\log t|  \]
    for some $C' > 0$ and every $t \in (0,\delta]$. 
    For the measure $\mu_2$ we thus obtain
    \begin{align*}
  \cL(\mu_2)(t)
&   = \cL(\mu_2)(t) = \cL(\mu_2)(\delta) + \int_t^\delta\, \cL(\mu_1)(x)\, dx 
 \leq \cL(\mu_2)(\delta) + \int_t^\delta C' |\log x|\, dx \\
&      \leq \cL(\mu_2)(\delta) + \int_0^\delta C' |\log x|\, dx.
    \end{align*}
 We conclude that $\cL(\mu_2)$ is bounded on $(0,\delta]$, which,
      by monotone convergence, implies that $\mu_2$ is finite, hence that
      $\mu$ is tempered.
    \end{prf}

    The preceding proof provides the 
      following quantitative information
on condition \eqref{eq:EN} from Proposition~\ref{prop:tempmeas}: 
    \[ \int_1^\infty\frac{d\mu(x)}{x^N} < \infty \quad 
      \Rarrow \quad {\rm(E_N)} \quad \Rarrow \quad 
      \int_1^\infty\frac{d\mu(x)}{x^{N+1}} < \infty.\] 
    
  The preceding proposition has an important consequence:
  \begin{thm} \mlabel{thm:E.4}
    Let $H = H^*$ be a selfadjoint operator on the complex
    Hilbert space~$\cH$ and
    $v \in \cH$ such that
    \[ v \in \cD(e^{tH})\quad \mbox{  for } \quad t \in [0,b).\]
      If there exist $C > 0$ and $N > 0$ such that
      \[ \|e^{tH}v\| \leq \frac{C}{(b-t)^N} \quad \mbox{ for } \quad t \in [0,b),\]
      then
      $\eta := \lim_{t \to b-} e^{tH}v$ exists in the space $\cH^{-\infty}$
      of distribution vectors for the unitary
        one-parameter group $(e^{itH})_{t \in \R}$.     
      \end{thm}

      If $v \in \cD(e^{bH})$, then we may take $N = 0$ and the limit
      $e^{bH}v = \lim_{t \to b-} e^{tH}v$ exists in $\cH$.
  
  \begin{prf} We may assume that $v$ is a cyclic vector, so that 
    $\cH \cong L^2(\R, \mu)$ with the multiplication operator
   $(Hf)(\lambda) = \lambda f(\lambda)$ (Spectral Theorem). 
   To simplify matters, we may assume that 
       $v(\lambda) = e^{-b\lambda}$. Then $\mu$ is an unbounded measure
    whose Laplace transform exists for $t \in [-2b, 0)$.  
Our assumption then means that 
      \[ \cL(\mu)(-2(b-t)) = \|e^{tH}v\|^2  \leq \frac{C^2}{(b-t)^{2N}}
        \quad \mbox{ for }\quad t \in [0,b).\]
 As we may replace $N$ by $\lceil N \rceil \in \N$,
        Proposition~\ref{prop:tempmeas} now implies that
        the measure $\mu$ is tempered, and hence that the constant
        function $\eta = 1$ is a distribution vector for the unitary 
        one-parameter group $(e^{itH})_{t \in \R}$
        (\cite[Cor.~10.8]{NO15}).         

A vector $f \in L^2(\R,\mu)$ is smooth if and only
    if $\lambda^n f(\lambda)$ is $L^2$ for every $n \in \N$.
    We want to show that, for $e_t(\lambda) := e^{t\lambda}$,
    we have $\lim_{t \to 0} e_{-t} = 1$  
    in the space of distribution vectors, i.e., that
    \[ \lim_{t \to 0} \int_\R e^{-t\lambda} f(\lambda)\,
      d\mu(\lambda) \to \int_\R f(\lambda)\,  d\mu(\lambda)  \]
    for every smooth vector~$f$.
 Choose $N$ such that $(1 + \lambda^2)^{-N}$ is $\mu$-integrable. Then
    \[ \int_\R e^{-t\lambda} f(\lambda)\, d\mu(\lambda)
      =  \int_\R e^{-t\lambda} f(\lambda) (1 + \lambda^2)^N
      \, \frac{d\mu(\lambda)}{(1 + \lambda^2)^N}.\] 
    Splitting the integral into the integration over $\R_-$ and $\R_+$,
    the correct limit behavior over $\R_-$ follows by dominated convergence
    and the fact that $e_{-t}$ is integrable for some $t > 0$.
    Over $\R_+$ the corresponding statement follows from the fact that
    $f(\lambda) (1 + \lambda^2)^N$ is $L^2$ with respect to $\mu$,
    hence also with respect to the finite measure
    $\frac{d\mu(\lambda)}{(1 + \lambda^2)^N}$, and thus also
    $L^1$ with respect to this measure. Now we can argue with
    dominated convergence.     
  \end{prf}
      
  \section{Hyperfunction boundary values}
  \mlabel{sec:3}

In this section we consider for a unitary representation
$(U,\cH)$ of a connected simple Lie group
holomorphic extensions of orbits maps $U^v \colon G \to \cH, g \mapsto U(g)v$.
If $G\subeq G_\C$, i.e., $G$ is linear,
then the Kr\"otz--Stanton Extension Theorem
implies that $U^v$ extends to the domain
$G \exp(i\Omega_\fp) K_\C \subeq G_\C$, where
$\Omega_\fp \subeq \fp$ consists of all elements for which the spectral
radius of $\ad x$ is smaller than $\frac{\pi}{2}$.
This result is best expressed as the extendability of certain
vector bundle maps
\[ G \times_K \cE \to \cH \]  to a holomorphic
vector bundle $\bE$ over the crown $\Xi$ of $G/K$
(\cite{AG90}, \cite{GK02b}).
As this makes also sense for non-linear groups,
we call this extension property 
Hypothesis (H1) and discuss its consequences. The main conclusion
is the existence of limit maps
\[ \beta^\pm \colon \cH^{[K]} \to \cH^{-\omega, [H]}_{\rm KMS}, \quad
  \beta^\pm(v) := \lim_{t \to \pi/2} e^{\mp it \partial U(h)} v \]
from $K$-finite vectors to $H$-finite hyperfunction vectors. 
Combining this with the Automatic Continuity
Theorem~\ref{thm:AutomaticContinuity}, we obtain
finite-dimensional $H$-invariant subspaces
$\sE_H \subeq \cH^{-\infty}$ that define nets of real subspaces
$(\sH^{G/H}_{\sE_H}(\cO))_{\cO \subeq G/H}$ for which
we shall see in the next section that the Reeh--Schlieder and
the Bisognano--Wichmann property are both satisfied.

Concretely, we consider the following setup: 
\begin{itemize}
\item $G$ is a connected semisimple Lie group
\item $\eta_G \colon  G \to G_\C$ is the universal complexification of $G$;
  its kernel is discrete. 
\item $h \in \g$ is an Euler element, i.e., $(\ad h)^3 = \ad h$. 
\item $\theta$ is a Cartan involution on $G$ and its Lie algebra $\g$,
  and   $\g = \fk \oplus \fp$ is the eigenspace decomposition,
  and we assume that $\theta(h) = -h$. 
\item $(\g,\tau)$ is ncc, 
  where $\tau = \tau_h \theta$, $\tau_h = e^{\pi i \ad h}\res_{\g}$,
  and the $\tau$-eigenspace decomposition is denoted $\g = \fh\oplus \fq$ 
  (see  \cite[Thm.~4.21]{MNO23a} and the introduction);
  we also write
  \[ \fh_\fk = \fh \cap \fk, \quad \fh_\fp = \fh \cap \fp, \quad
    \fq_\fk = \fq \cap \fk, \quad \fq_\fp = \fq \cap \fp.\]
\item $C \subeq \fq$ is the {\bf maximal} $\Inn(\fh)$-invariant
  cone containing $h$ (cf.\ \cite[\S 3]{MNO23a}). 
\item $H \subeq G^\tau$ is an open $\theta$-invariant
  subgroup for which $\Ad(H)C = C$
  (which is equivalent to $H_K= H \cap K$ fixing $h$ \cite[Cor.~4.6]{MNO23a}).
If $G$ is given, this means that $H_{\rm min} \subeq H \subeq H_{\rm max}$,
    where $H_{\rm min} = G^\tau_e$ is connected and
    $H_{\rm max} = K^{\tau,h} \exp(\fh_\fp)$.    
\item  $M = G/H$ is the corresponding ncc symmetric space.
\item $G_{\tau_h} := G \rtimes \{\1,\tau_h\}$ is
  the corresponding graded Lie group.
\item $(U,\cH)$ an {\bf irreducible} antiunitary representation of $G_{\tau_h}$,
  i.e., $J := U(\tau_h)$ is a conjugation on $\cH$
  and $U(G) \subeq \U(\cH)$. 
\end{itemize}

\begin{rem} (Existence of antiunitary extensions) 
  If $(U,\cH)$ is an irreducible unitary representation
  of $G$, then by \cite[Thm.~2.11(d)]{NO17} exactly one of the
    following cases occurs:
    \begin{itemize}
    \item It extends to an antiunitary representation
      of~$G_{\tau_h}$ on $\cH$, where the commutant is $\R \1$.
      We refer to \cite[Thm.~4.24]{MN21} for a discussion of the case
      $G = \tilde\SL_2(\R)$, asserting that every irreducible
      unitary representation of this group extends to an
      antiunitary representation of~$G_{\tau_h}$.
    \item $U \oplus (U^* \circ \tau_h^G)$
      extends to an irreducible antiunitary representation $V$ of
      $G_{\tau_h}$ on $\cH \oplus \cH^*$ by
      $V(\tau_h)(v,\alpha) = (\Phi^{-1}\alpha, \Phi v)$, where
      $\Phi \colon \cH \to \cH^*$ is given by
      $\Phi(v)(w) = \la v,w\ra$. Its commutant is~$\C$
      if $U^* \circ \tau_h^G \not\cong U$ and 
      the quaternions $\H$ if $U^* \circ \tau_h^G \cong U$. 
    \end{itemize}
\end{rem}


Let $(U,\cH)$ be a unitary representation of the
semisimple Lie group $G$, fix a Cartan involution~$\theta$, and
write $\g = \fk \oplus \fp$ for the corresponding Cartan decomposition
of the Lie algebra. Then
\[ \Omega_\fp := \{ x \in \fp \colon \Spec(\ad x) \subeq (-\pi/2, \pi/2) \} \]
is an open convex subset of $\fp$.

Let $\eta \colon G \to G_\C$ denote the universal complexification of $G$.
Then $G_\C$ carries a uniquely determined antiholomorphic involution
$\sigma$ with $\sigma \circ \eta = \eta$, i.e.,
it induces on the Lie algebra $\g_\C$ the complex 
conjugation with respect to $\g$. The holomorphic involution
$\theta_\C$ on $G_\C$ induced by a Cartan involution $\theta$ on $G$ via
$\theta_\C \circ \eta = \eta \circ \theta$ then commutes with $\sigma$.
We put
\[ K_\C := (G_\C^{\theta_\C})_e = \la \exp \fk_\C \ra.\] 
Hence $\sigma$ preserves the complex symmetric subspace
\[ G_\C^{-\theta} = \{ g \in G_\C \colon \theta_\C(g) = g^{-1}\} \]
of $G_\C$ 
whose identity component is isomorphic to $G_\C/G_\C^\theta$ via the
$G_\C$-action by $g.h := g h \theta_\C(g)^{-1}$.

We consider the crown domain
of the Riemannian symmetric space
$G/K$: 
\begin{equation}
  \label{eq:crown1}
  \Xi := G \times_K i\Omega_\fp = (G \times i\Omega_\fp)/\sim
  \quad \mbox{ with } \quad
  (g, ix) \sim (gk, \Ad(k)^{-1}ix), k \in K. 
\end{equation}
The complex structure on this domain is determined by the
requirement that the map 
\[ q \colon  \Xi \to G_\C^{-\theta},
\quad q([g,i x]) \mapsto g.\Exp(ix) = g\exp(2ix)\theta(g)^{-1}, \]
which is a covering of an open subset of $G_\C^{-\theta_\C}$,
is holomorphic.

\begin{rem} \mlabel{rem:xia} 
   The construction of $\Xi$ only depends on the
  Lie algebra $\g$ of $G$. It produces the same manifold
  for the simply connected covering group $\tilde G$ and for the
  adjoint group $G/Z(G) \cong \Ad(G)$. This is due to the fact that
  $K = G^\theta$ is always connected.
\end{rem}

For each finite-dimensional unitary $K$-representation
$(\sigma,\cE)$, we obtain a vector bundle
$\bE \to \Xi$ by
\begin{equation}
  \label{eq:qbe}
 q_\bE \colon \bE := (G \times i \Omega_\fp) \times_K \cE \to \Xi, \quad
 [g,ix,v] \mapsto [g,ix].
\end{equation}

\begin{prop} \mlabel{prop:2.8} 
  $\bE \to \Xi$ is a holomorphic vector bundle.
If, in addition, $J\cE = \cE$,
    then the antiholomorphic involution $\oline\tau_h$
    on $\Xi$ induced by $\tau_h$ lifts to
    an antiholomorphic involution $\oline\tau_\bE$ on $\bE$, given  by
\begin{equation}
  \label{eq:tau-be}
  \oline\tau_\bE([g,ix,v]) := [\tau_h(g), -i \tau_h(x), Jv].
\end{equation}
   \end{prop} 

\begin{prf} In view of Remark~\ref{rem:xia},
  we may assume that $G$ is simply connected,
  which, by polar decomposition, implies that $K$
  and its universal complexification $K_\C$ are simply connected as well. 
  Then $G_\C$ is also simply connected and
  the universal complexification of $K$ is the natural map
  $\eta_K \colon K \to \tilde K_\C$ whose existence follows from the
  simple connectedness of $K$.

  Moreover, the fixed point group $G_\C^{\theta_\C}$
  is connected by (\cite[Thm.~IV.3.4]{Lo69}), hence equal to $K_\C$. 
  Thus $(G_\C^{-\theta})_e \cong G_\C/K_\C$, and we have a natural covering map
  $\Xi \to G.\Exp(i\Omega_\fp) \subeq G_\C/K_\C$ of an open domain of $G_\C/K_\C$,
  that is diffeomorphic to $\eta_G(G) \times_{\eta_G(K)} i\Omega_\fp$.
  Remark~\ref{rem:xia}
  implies that this domain is biholomorphic to $\Xi$,
  and we may thus consider $\Xi$ as a domain in $G_\C/K_\C$.   

  Let $q \colon G_\C \to G_\C/K_\C$ denote the quotient map and
  \[ \Xi_{G_\C}:= q^{-1}(\Xi) = q^{-1}(G.\Exp(i\Omega_\fp))
    = G \exp(i\Omega_\fp) K_\C .\]
  This is an open subset of $G_\C$ that is right $K_\C$-invariant,
  so that $\Xi_{G_\C}$ is a $K_\C$-principal bundle over~$\Xi$.
  As $\Xi$ is contractible (it is an affine bundle over the contractible space
  $G/K$), the natural homomorphism 
  $\pi_1(K_\C) \to \pi_1(\Xi_{G_\C})$ is an isomorphism
  by the long exact homotopy sequence for fiber bundles.
  We conclude that the simply connected covering
  $\tilde\Xi_{G_\C}$ is  a holomorphic $\tilde K_\C$-principal
  bundle over $\Xi$. 

  Extending the representation $\sigma^\cE \colon K \to \U(\cE)$
  to a holomorphic representation $\sigma^\cE_\C \colon \tilde K_\C \to
  \GL(\cE)$, we obtain a holomorphic vector bundle
  \[ \bE' := \tilde\Xi_{G_\C} \times_{\tilde K_\C} \cE \to \Xi.\]

  As $G$ is simply connected, the $G$-action on $\Xi_{G_\C}$ by
  left multiplications 
  lifts to the simply connected covering $\tilde\Xi_{G_\C}$
  and the map   $\exp \colon i\Omega_\fp \to  \Xi_{G_\C}$ lifts uniquely
  to a map
  \[ \tilde\exp \colon i\Omega_\fp \to \tilde\Xi_{G_\C} \quad \mbox{ with } \quad
  \tilde\exp(0) = \tilde e,\] 
 where $\tilde e$ is the base point in $\tilde\Xi_{G_\C}$ over $e \in G_\C$.
  As lifts are uniquely determined, once the image of the base
  point is fixed, $\tilde\exp \colon  i\Omega_\fp \to \tilde\Xi_{G_\C}$
  is equivariant with respect to the conjugation action by $K$.   
  We thus obtain a well-defined map
  \[ \Psi \colon \bE \to \bE', \quad
  [g,ix,v] \mapsto [g\tilde\exp(ix), v]. \]
  This is an equivalence of complex vector bundles
  over $\Xi$, so that we obtain a holomorphic vector bundle
  structure on $\bE$.

  Now we turn to the antiholomorphic involution $\oline\tau_h$ of~$\Xi$.
  As $\tau_h(\fk) = \fk$ and $\tau_h(\Omega_\fp) = \Omega_\fp$,
  the antiholomorphic
  involution $\oline\tau_h(g) = \exp(\pi i h) \sigma^\cE(g) \exp(-\pi i h)$
  of $G_\C$ preserves the subgroups $K_\C$, $G$ and the subset
  $\exp(i\Omega_\fp)$, hence also $\Xi_{G_\C}$.
  Therefore it induces an antiholomorphic involution on
  $\tilde\Xi_{G_\C}$.
  If $J\cE = \cE$, then
  the relation $J U(k) = U(\tau_h(k)) J$ for $k \in K$
  implies that $J \sigma^\cE_\C(k) = \sigma^\cE_\C(\oline\tau_h(k)) J$ for
  $k \in \tilde K_\C$, which leads to an antiholomorphic involution
$[m,v] \mapsto [\oline \tau_h(m), Jv]$ 
  on $\bE'$. Now the assertion follows from the fact that
  $\Psi$ intertwines this involution with~$\oline\tau_\bE$.
\end{prf}

The following hypothesis will turn out to be crucial for
the construction of local nets over ncc symmetric spaces. 

\nin{\bf Hypothesis (H1)} {\it 
  For the finite-dimensional
  $K$-invariant subspace $\cE \subeq  \cH$, the map
  \[ G \times_K \cE \to \cH, \quad [g,v] \mapsto U(g) v \]
  extends to a holomorphic map
  \begin{equation}
    \label{eq:cE-extend}
 \Psi_\cE \colon \bE = (G \times i \Omega_\fp) \times_K \cE \to \cH, \quad
 [g,ix,v] \mapsto U(g) e^{i\partial U(x)} v.
  \end{equation}}

\begin{thm} \mlabel{thm:ks04}
  {\rm(Kr\"otz--Stanton Extension Theorem;  \cite[Thm.~3.1]{KSt04})}
    If $U$ is irreducible, $G$ is simple and 
  $\eta_G \colon G \to G_\C$ is an embedding,
  then {\rm Hypothesis~(H1)} is satisfied.
\end{thm}

The Kr\"otz--Stanton Extension Theorem immediately generalizes
to connected linear semisimple groups because
their irreducible representations factorize as tensor products
of irreducible representations of the simple normal subgroups.

We expect that Hypothesis~(H1) also holds if
$\eta_G$ is not injective. In Section~\ref{sec:rank-one} below,
we shall verify this for $G = \tilde\SL_2(\R)$,
and more generally for $\g = \so_{1,d}(\R)$, $d \geq 2$.

The following reformulation of Hypothesis~(H1) approaches
it from a different angle:
\begin{prop} \mlabel{prop:hyp1}
  {\rm Hypothesis~(H1)} is equivalent to
$\cE \subeq \bigcap_{x \in \Omega_\fp} \cD(e^{i \partial U(x)}).$ 
\end{prop}

\begin{prf} Clearly, Hypothesis~(H1) implies that
any $v \in \cE$ is contained in $\cD(e^{i \partial U(x)})$
  for $x \in \Omega_\fp$. Suppose, conversely, that this is the case.
  Then we obtain a map
  \[ f \colon G \times i \Omega_\fp \times \cE \to \cH, \quad
  f(g,ix, v) := U(g) e^{i\partial U(x)} v.\]
  First, we argue as in the proof of Lemma~\ref{lem:charh-om-v}
  to see that this map is analytic. Then we observe that, for $k \in K$,
  \begin{align*}
 f(gk, \Ad(k)^{-1}ix, U(k)^{-1}v)
&  = U(gk) e^{i\partial U(\Ad(k)^{-1}x)} U(k)^{-1}v\\
    &  = U(g) e^{i\partial U(x)} v = f(g,ix,v),
  \end{align*}
  so that $f$ factors through a well-defined real-analytic map
  \[ F \colon \bE \to \cH, \quad [g,ix,v] \mapsto U(g) e^{i \partial U(x)} v.\]
  As this map is real-analytic, it suffices to
  verify its holomorphy in a neighborhood of $[e,0,v]$.
  As $v$ is an analytic vector, we have for $x,y \in \fp$ sufficiently small 
  \[ U(\exp x) e^{i \partial U(y)} v = U^v(\exp(x * iy)),\]
  where $a * b = a + b + \frac{1}{2}[a,b] + \cdots$ denotes the
Baker--Campbell--Hausdorff series. 
  Comparing with the construction of the complex structure on $\bE$
  in the proof of Proposition~\ref{prop:2.8}, it now follows
  that~  $F$ is holomorphic.   
\end{prf}

\begin{rem} (a) Hypothesis~(H1) is a condition on the
  representation $U$ and the $K$-invariant subspace~$\cE$.
  If $\eta \colon G \to G_\C$ is injective,
  then  the inclusion $K \into K_\C$ is the universal complexification
  of $K$ and 
  $\Xi_{G_\C} = G \exp(i\Omega_\fp) K_\C$. Hence Hypothesis~(H1) 
  reduces to the statement that, for every $K$-finite vector
  $v \in \cH$, the orbit map $U^v \colon  G \to \cH$ extends to a holomorphic
  map $\Xi_{G_\C} \to \cH$. 

  The preceding discussion shows that Hypothesis~(H1) is also
  satisfied if $\ker(\eta_G) \subeq \ker(U)$ because it then
  factors through a representation of a group $G$ that embeds in its
  complexification. 
  
  \nin (b) Theorem~\ref{thm:ks04} has a slightly involved history.
In \cite{KSt04} it is based on Conjecture A, asserting that
\begin{equation}
  \label{eq:crown-inc}
  G \exp(i\Omega_\fp)  \subeq N_\C A_\C K_\C,
\end{equation}
where $G = NAK$ is an Iwasawa decomposition and
$N_\C$, $A_\C$, $K_\C$ are the corresponding complex integral subgroups of~$G_\C$.
The inclusion \eqref{eq:crown-inc}
was proved for classical groups in \cite{KSt04}.
The general case was obtained by A.~Huckleberry
in \cite[Prop.~2.0.2]{Hu02}, 
using a certain strictly plurisubharmonic function,
and with structure theoretic
methods by T.~Matsuki in \cite[Thm.]{Ma03}.
It also follows from the Complex Convexity Theorem
in \cite{GK02a}  which provides finer information.
Another argument for the inclusion \eqref{eq:crown-inc},
based on holomorphic extension of eigenfunctions of the
Laplace--Beltrami operator, is given by
B.~Kr\"otz and H.~Schlichtkrull in \cite[Cor.~3.3]{KrSc09}.

\nin (c) In this context, it is interesting to observe
  that the Kr\"otz--Stanton Extension theorem is optimal
  with respect to the domain to which analytic orbit maps
  may extend. Refining techniques by
  R.~Goodman \cite{Go69}, the following result has been obtained in
  \cite{BN23}. Consider the $2$-dimensional
  group $\Aff(\R)_e\cong \R \rtimes \R_+$ with Lie algebra $\R y +\R h$ 
  where $[h,y] = y$, and a unitary representation
  $(U,\cH)$
  satisfying the non-degeneracy condition $\ker(\partial U(y)) = \{0\}$.
  Then   any analytic vector $v$ for which
  $e^{it \partial U(h)}v$ is defined and contained
  in $\cH^\omega$ for $|t| \leq \pi/2$ is zero. 

  As the Lie algebra $\g$ of a non-compact semisimple Lie group
  contains many copies of the non-abelian $2$-dimensional Lie algebra,
  this observation implies that, if $\cH^G = \{0\}$, and $v \in \cH^\omega$
  and $x \in \fp$ are such that 
  $e^{it \partial U(x)}v$ is defined and an analytic 
  vector for $|t| \leq 1$, then all eigenvalues of $\ad x$ are $< \pi/2$.
  Hence the domain $\Omega_\fp \subeq \fp$ is the maximal domain
  for which a result as Theorem~\ref{thm:ks04} may hold.
\end{rem}

\begin{rem} \mlabel{rem:hyp1-invol}
If $J\cE = \cE$ and Hypothesis~(H1) is satisfied, then 
\begin{equation}
  \label{eq:j-tauh}
  J \circ \Psi_{\cE} = \Psi_{\cE} \circ \oline\tau_\bE  
\end{equation}
(cf.~\eqref{eq:cE-extend}) holds for the antiholomorphic involution $\oline\tau_\bE$ on $\bE$, 
given by
\begin{equation}
  \label{eq:tau-be2}
  \tau_\bE([g,ix,v]) = [\tau_h(g), -i \tau_h(x), Jv]
\end{equation}
as  in Proposition~\ref{prop:2.8}.
\end{rem}

Several parts of the following proposition follow from
\cite[Thm.~2.1.3]{GKO04} if
$\eta_G$ is injective and $G_\C$ is simply connected.
Our formulation does not require this assumption.
The main information is contained in~(e).

\begin{prop}
  \mlabel{prop:exten}
  Under {\rm Hypothesis~(H1)}, 
  for any $K$-finite vector $v \in \cH^{[K]}$, the limits   
\[ \beta^\pm(v) := \lim_{t \to \pi/2} e^{\mp it \partial U(h)} v \]
exist in $\cH^{-\omega}(U_h)$, where $U_h(t) = \exp(th)$ for $t \in \R$.
Moreover, the following assertions hold:
\begin{itemize}
\item[\rm(a)] The maps $\beta^\pm \colon  \cH^{[K]} \to \cH^{-\omega}(U_h)
  \subeq \cH^{-\omega}$ are injective.   
\item[\rm(b)] The limits $\beta^\pm(v)$ exist in the topology
  of uniform convergence on subsets of the form
  $U(C)\xi$, where $C \subeq G$ is compact and $\xi \in \cH^\omega$.
\item[\rm(c)] The automorphism
  $\zeta := e^{-\frac{\pi i}{2} \ad h} \in \Aut(\g_\C)$ 
  satisfies
  \begin{equation} \label{eq:zetakh}
    \zeta(\fh_\fk + i \fq_\fk) = \fh, 
  \end{equation}
  hence in particular $\zeta(\fk_\C) = \fh_\C$.
\item[\rm(d)]  We have the intertwining relation 
  \[ \beta^\pm \circ \dd U(x) = \dd U^{-\omega}(\zeta^{\pm 1}(x)) \circ \beta^\pm \colon
  \cH^{[K]} \to \cH^{-\omega, [H]}\quad \mbox{ for } \quad
  x \in \g_\C.\]
\item[\rm(e)] If $\cE \subeq \cH^{[K]}$ is finite-dimensional and $J$-invariant,
  then the finite-dimensional
  real subspaces $\beta^\pm(\cE^J) \subeq \cH^{-\omega}$ are
  $U^{-\omega}(H)$-invariant.
  We further have
  \[ J \beta^\pm(v) = \beta^\mp(v) \quad \mbox{ for } \quad
  v \in \cH^{[K]}.\] 
\end{itemize}
\end{prop}

\begin{prf} For $|t| < \frac{\pi}{2}$, we have
  $th \in \Omega_\fp$, so that
  Hypothesis~(H1) implies that
  \begin{equation}
    \label{eq:vdom}
 v \in \cD(e^{it \partial U(h)}) \quad \mbox{ for } \quad
 |t| < \frac{\pi}{2}.
  \end{equation}
  Hence the existence of the limits $\beta^\pm(v)$ in the
  weak-$*$-topology
  on $\cH^{-\omega}(U_h)$ follows from Lemma~\ref{lem:k.1}.
  
\nin  (a) (cf.~\cite[Thm.~2.1.3]{GKO04}) Suppose that
$\beta^+(v) = 0$. As $v \in \cH^{[K]} \subeq \cH^\omega$ (\cite{HC53})
is contained in $\cD(e^{t i \partial U(h)})$ for
$|t| < \frac{\pi}{2}$, the function
\[ f \colon \R \to \C, \quad f(t) := \la v, e^{t \partial U(h)} v \ra \]
extends analytically to the strip $\cS_{\pm \pi}$. Our assumption implies that 
\[\quad  f\big(-\tfrac{\pi i}{2} + t\big)
= \beta^+(v)(e^{t \partial U(h)} v) = 0
\quad \mbox{  for } \quad t \in \R, \]
so that $f = 0$ by analytic continuation, and thus
$0 = f(0) = \|v\|^2$ leads to $v = 0$.

\nin (b) (cf.\ \cite[p.~653]{GKO04})
The corollary to \cite[Thm.~3.3.1]{Tr67} (Banach--Steinhaus Theorem)
implies that pointwise convergence of a sequence on a Fr\'echet space
implies uniform convergence on compact subsets.
To see that the limit defining $\beta^+(v)$ is uniform 
on subsets of the form $U(C)w$, $w \in \cH^\omega$, $C \subeq G$ compact,
we recall from Proposition~\ref{prop:2.3}, that the orbit map
$U^w \colon  G \to \cH^\omega$ is continuous. 
Actually its proof in \cite[Prop.~3.4]{GKS11} shows that
the restriction of the orbit map to $C$ factors
through a continuous map $U^w \colon C \to \cH^\omega_V$,
for some $V$ as in Lemma \ref{lem:charh-om-v}. 
Therefore the Banach--Steinhaus Theorem applies.

\nin (c) We have $\ker(\ad h) = \fh_\fk \oplus \fq_\fp = {\g}^{\tau_h}$ and
$\g^{-\tau_h} = \fh_\fp \oplus \fq_\fk = \g_1(h) \oplus \g_{-1}(h)$.
As $\theta(h) = -h$, we have $\theta(\g_1(h)) = \g_{-1}(h)$.
So $\fq_\fk = \{ x+ \theta(x) \colon x \in {\g}_1(h)\}$.
This shows that
\[ \zeta(\fq_\fk) 
=\{ i(x - \theta(x)) \colon x \in \g_1(h)\} 
=i \fh_\fp,\]
and therefore $\zeta(\fh_\fk + i \fq_\fk)) = \fh_\fk + \fh_\fp = \fh$,
which entails $\zeta(\fk_\C) = \fh_\C$
(cf.\ \cite[Thm.~5.4]{NO23}). 

\nin (d) For a $K$-finite vector $v$, we have 
  \begin{align*} 
    \dd U^{-\omega}(\zeta^{\pm 1}(x))\beta^\pm(v)
    &= \lim_{t \to \pi/2} \dd U(\zeta^{\pm 1}(x)) e^{\mp it \partial U(h)} v \\
&    = \lim_{t \to \pi/2} e^{\mp it \partial U(h)}
    \dd U(e^{\pm it \ad h}\zeta^{\pm 1}(x))  v \\
    & = \beta^\pm\big(\dd U(e^{\pm i\frac{\pi}{2} \ad h}\zeta^{\pm 1}(x))  v\big)
    = \beta^\pm(\dd U(x)  v). 
  \end{align*}
  Here we use that $t \mapsto
  \dd U\big(e^{it \ad h}\zeta(x)\big)v$ is a continuous curve
  in a finite-dimensional subspace. 

  \nin (e) As $J i\partial U(h) J = - i \partial U(h)$,
  we have $J \beta^\pm(v) = \beta^\mp(J v)$ for $v \in \cH^{[K]}.$ 
  The relation
  \begin{equation}
    \label{eq:jdu-rel}
    J \partial U(z) J = \partial U(\tau_h(\oline z))
      \end{equation}
  shows that, on $\cH^{[K]}$, the operator
  $\dd U(z)$ for $z \in \fh_\fk + i \fq_\fk$, commutes with $J$.
  By (d), $\beta^\pm$ intertwines these operators with
  $\dd U^{-\omega}(\fh)$. Hence the subspaces
  $\beta^\pm(\cE^J)$ are $\dd U^{-\omega}(\fh)$-invariant.
  The subspace $\cE^J \subeq \cE$ is invariant under the
  subgroup~$K^{\tau_h}$. 
  As $H = H_K \exp(\fh_\fp)$ with
  $H_K \subeq K^\tau = K^{\tau_h}$ and 
  $H_K \subeq K^h$ (\cite[Lemma~4.11]{MNO23b}),
  the $K^h$-equivariance of $\beta^\pm$ entails that
  the 
$\dd U^{-\omega}(\fh)$-invariant 
subspaces $\beta^\pm(\cE^J)$ are invariant under~$U^{-\omega}(H)$.
\end{prf}

  \begin{rem} As a consequence of \eqref{eq:jdu-rel}, 
the real subspace $\cH^{[K],J}$ of the $\g$-module
    $\cH^{[K]}$ is invariant under the Lie subalgebra
    \[ \g^{\tau_h} + i \g^{-\tau_h}
      = \fh_\fk + \fq_\fp + i \fh_\fp + i \fq_\fk.\]
Further, 
    \[ \zeta^{\pm}(\g^{\tau_h} + i \g^{-\tau_h}) 
      = \g^{\tau_h} + \g^{-\tau_h} = \g,\]
    so that the equivariance relation Proposition~\ref{prop:exten}(d)
implies that $\beta^\pm(\cH^{[K],J})$ are
$\g$-submodules of~$\cH^{-\omega}$.
\end{rem} 

\begin{rem} The intertwining relation in
  Proposition~\ref{prop:exten}(d) can be strengthened
   as follows. We consider the dense subspace 
   \[ \cD := \bigcap_{|t| < \pi/2} \cD(e^{it \partial U(h)}) \]
   of $\cH$ as a subspace of the locally convex space  $\cH^{-\omega}$,
   on which the Lie algebra acts by continuous operators.
   Hypothesis~(H1) asserts that $\cH^{[K]} \subeq \cD$,
   but $\cD$ may be substantially larger as we shall see in
   Section~\ref{subsec:6.2}. In particular, it is invariant
   under~$U(G^h)$, where $G^h$ is the centralizer of $h$ in $G$. 

The existence of the limits 
    \[ \beta^\pm \colon \cD \to \cH^{-\omega}(U_h) \subeq \cH^{-\omega} \]
    in the weak-$*$-topology follows from Lemma~\ref{lem:k.1}.
    The same argument as in the proof of Proposition~\ref{prop:exten}(a)
    implies that both maps $\beta^\pm$ are injective.

    To establish an equivariance relation for the $\g$-action, 
    let $\cH^{-\omega}_{\rm ext} \subeq \cH^{-\omega}$ be the subspace of
    those hyperfunction vectors for which the orbit map
    $U^\eta \colon \R \to \cH^{-\omega}$ extends to a weak-$*$-continuous
    map on $\oline{\cS_\pi}$ that is holomorphic on the interior.
    For any $\eta \in \cH^{-\omega}_{\rm ext}$ and $x \in \g$, we have 
    \begin{equation}
      \label{eq:3.11}
U^\omega(\exp th) \dd U^{-\omega}(x)\eta
= \dd U^{-\omega}(e^{t \ad h} x) U^\omega(\exp th)\eta.
    \end{equation}
      Since $\dd U^{-\omega}(\g)$ consists of continuous operators on
      $\cH^{-\omega}$ and
\[   \dd U^{-\omega}(e^{t \ad h} x) 
= e^t \dd U^{-\omega}(x_1) + \dd U^{-\omega}(x_0)  
+ e^{-t} \dd U^{-\omega}(x_{-1}),\]
where $x = x_1 + x_0 + x_{-1}$ is the decomposition of $x$ into
$\ad h$-eigenvectors,
the right hand side of \eqref{eq:3.11} extends to a weak-$*$-continuous map 
on $\oline{\cS_\pi}$ that is holomorphic on the interior.
It follows
      that $\eta' := \dd U^{-\omega}(x)\eta \in \cH^{-\omega}_{\rm ext}$
      with
      \[ U^{\eta'}(z) = \dd U^{-\omega}(e^{z \ad h} x) U^\eta(z)\quad \mbox{ for }
        \quad
      z \in \oline{\cS_\pi}.\]
    For $z = \pi i/2$ and $\eta \in \cH$, we obtain in particular
    \begin{equation}
      \label{eq:gen-equiv-rel}
 \beta^\pm \circ \dd U^{-\omega}(x)
      = \dd U^{-\omega}(\zeta^{\pm 1}(x)) \circ \beta^\pm
      \colon \cD \to \cH^{-\omega}.
    \end{equation}
  \end{rem}

  \begin{lem} \mlabel{lem:3.11}
 $\cH^{-\omega}_{\rm KMS} \subeq \cH^{-\omega}$ is a real 
   $\g$-invariant subspace.
  \end{lem}

  \begin{prf} Let $\eta \in \cH^{-\omega}_{\rm KMS}$
    and $x\in \g_{\pm 1}(h)$. Then
    \[
      U^{-\omega}_h(t) \dd U^{-\omega}(x)\eta 
=     e^{\pm t} \dd U^{-\omega}(x)U^{-\omega}_h(t) \eta 
=     e^{\pm t} \dd U^{-\omega}(x) U^\eta_h(t).\]
As the operator $\dd U^{-\omega}(x) \colon \cH^{-\omega} \to \cH^{-\omega}$
is continuous and $U^\eta_h$ extends to a continuous map
$\oline{\cS_\pi} \to 
\cH^{-\omega}$ 
that is holomorphic on $\cS_\pi$, the same holds for the
orbit map of $\eta' := \dd U^{-\omega}(x)\eta$.

Moreover, this extension satisfies
\begin{align*}
 J U_h^{\eta'}(t)
&  = e^{\pm t} \dd U^{-\omega}(\tau_h(x)) J U^\eta_h(t)
  = e^{\pm t} \dd U^{-\omega}(-x) U^\eta_h(t + \pi i)\\
&  = e^{\pm (t + \pi i)} \dd U^{-\omega}(x) U^\eta_h(t + \pi i)
                  = U_h^{\eta'}(t+ \pi i).
\end{align*}
This proves the lemma because~$\g = \g_1(h) + \g_0(h) + \g_{-1}(h)$
and $\cH^{-\omega}_{\rm KMS}$ is obviously invariant under
  $\dd U^{-\omega}(\g_0(h))$.
  \end{prf}

\nin {\bf Hypothesis (H2)} {\it 
  {\rm  Hypothesis~(H1)} is satisfied for $\cE$ and,
  in addition,
\[ \sE_H := \beta^+(\sE_K) \subeq \cH^{-\infty}.\]}
Proposition~\ref{prop:exten}(e) then implies that
  the subspace $\sE_H$ is $U^{-\infty}(H)$-invariant.
\\

The following generalization 
of the van den Ban--Delorme Automatic Continuity
Theorem is a key tool to show that
Hypothesis~(H2) follows from Hypothesis~(H2). 
We indicate a proof of this result in
Appendix~\ref{app:C} (Theorem~\ref{thm:AutomaticContinuity}).
Note that, by Proposition~\ref{prop:exten}(e),
  the finite-dimensional subspaces $\beta^\pm(\cE^J)$
  of $\cH^{-\omega}$ are $\fh$-invariant, hence consist of $\fh$-finite
  distribution vectors, so that the following theorem
  implies that $\beta^\pm(\cE^J) \subeq \cH^{-\infty}$. 

\begin{thm}  {\rm(Automatic Continuity Theorem)} \mlabel{thm:autocont}
 If $\eta_G \colon G \to \eta_G(G) \subeq G_\C$ is a finite
    covering  and $(U,\cH)$ is irreducible, then every
  $\fh$-finite hyperfunction vector is a distribution vector, i.e.,
  \[ (\cH^{-\omega})^{[\fh]} \subeq \cH^{-\infty}.\] 
\end{thm}

\begin{prop} \mlabel{prop:3.21} Under {\rm Hypothesis~(H2)}, 
  the real subspaces $\beta^\pm(\cH^{[K],J}) \subeq \cH^{-\infty, [H]}$ 
are $\g$-submodules with trivial intersection.
\end{prop}

\begin{prf} The real subspace
  $\beta^+(\cH^{[K],J}) \subeq \cH^{-\infty,[H]}$
  is contained in $\cH^{-\omega}(U_h)_{\rm KMS}$ and
  $\beta^-(\cH^{[K],J}) = J \beta^+(\cH^{[K],J})$, so that
  Corollary~\ref{cor:2.12} shows that
  \[ \beta^+(\cH^{[K],J}) \cap \beta^-(\cH^{[K],J}) \subeq
    \ker(\partial U(h))\subeq \cH.\]
  Moore's Theorem (\cite{Mo80}) and the unboundedness of $e^{\R \ad h}$
  show that we have ${\ker(\partial U(h)) = \{0\}}$, so that  this intersection
  is trivial.
\end{prf}

We conclude this section by listing some open research problems.
\begin{prob} The map $\beta^\pm$ provides a realization
  of the representation $(U,\cH)$ from the perspective of $H$-finite
  distribution vectors, i.e., the vector bundle
  $\bE_M := G \times_H (\sE_H)_\C$, constructed from
    the finite-dimensional $H$-invariant subspace
    $\sE_H = \beta^+(\sE_K)\subeq \cH^{-\infty}$, permits a $G$-equivariant
    injection
    \[ \cH^{-\infty} \into C^{-\infty}(M,\bE_M).
  \begin{footnote}{We plan to discuss these realizations in more detail
      in the future.}   \end{footnote}\] 
  We thus obtain in particular a realization of the unitary
  representation $(U,\cH)$ in distributional sections of the vector
  bundle~$\bE_M$.
  It would be interesting to see more concrete expressions
  of the corresponding bundle-valued distribution kernel
  specifying this Hilbert space. This is of particular importance
  to understand the locality properties of the corresponding
  net $\sH^M_{\sE_H}$ on~$M$.
  
 If $(U,\cH)$ is irreducible,
  then $\cH^{[K]}$ is an irreducible $\g$-module
  (\cite{HC53}). However, Proposition~\ref{prop:3.21}
shows that the space $(\cH^{-\infty})^{[H]}$ of $\fh$-finite distribution
vectors is never an irreducible real $\g$-module.
 \end{prob}

  \begin{prob} Does $\eta \in \cH^{-\infty}_{\rm KMS}$ imply that
    $U^\eta(z) \in \cH$ for $z \in \cS_\pi$?
    The same question can be asked for $\eta \in \cH^{-\omega}_{\rm KMS}$.
    This is true
    if $\eta \in \cH^{-\omega}(U_h)_{\rm KMS}$ (Proposition~\ref{prop:2.9}),
    but it is not clear for every~$\eta \in \cH^{-\infty}_{\rm KMS}$.    
  \end{prob}

  \begin{prob} Is the subspace $\cH^{-\omega}(U_h)$ invariant under
    $\dd U^{-\omega}(\g)$? This is clearly the case for
    the Lie subalgebra $\g_0(h) \subeq \g$. Dualy,
    one may ask if the subspace $\cH^\omega(U_h)$ of $U_h$-analytic
    vectors in $\cH$ is invariant under $\g$, but this requires that
    $\cH^\omega(U_h) \subeq \cH^\infty$. Is this always the case
    for irreducible unitary representations of $G$?    
  \end{prob}

  \section{Reeh--Schlieder and Bisognano--Wichmann property}
  \mlabel{sec:4}

For linear groups $G\subeq G_\C$ we obtain 
from the Kr\"otz--Stanton Extension Theorem~\ref{thm:ks04},
Proposition~\ref{prop:exten}(e)
and Theorem~\ref{thm:autocont} that 
$\beta^\pm(\cH^{[K]}) \subeq \cH^{-\infty}$.
With $\sE_H \subeq \cH^{-\infty}$ as in Proposition~\ref{prop:exten},  
we thus obtain a {\it net of real subspaces}
$\sH_{\sE_H}^M$ on $M = G/H$.
In this section we show that this net has the
Reeh--Schlieder and the Bisognano--Wichman property.

\subsection{The Reeh--Schlieder property}
\mlabel{subsec:4.1}

In this section we assume Hypothesis~(H2), 
so that the real subspaces $\sH_{\sE_H}^G(\cO)$
are defined for every open subset $\cO \subeq G$.
In this section we show that, whenever $\cO \not=\eset$, 
this subspace is total.

We shall need the following lemma (\cite[Lemma~3.7]{NO21}):
\begin{lem} \mlabel{lem:cartes}
Let $X$ be a locally compact space and 
$f \colon X \to \cH^{-\infty}$ be a weak-$*$-continuous map. 
Then the following assertions hold: 
\begin{itemize}
\item[\rm(a)] $f^\wedge \colon X \times \cH^\infty \to \C, f^\wedge(x,\xi) := f(x)(\xi)$, 
is continuous. 
\item[\rm(b)] If, in addition, $X$ is a complex manifold and 
$f$ is antiholomorphic, then $f^\wedge$ is antiholomorphic. 
\end{itemize}
\end{lem}

We shall use the following result from
\cite[Thm.~6.4, Cor.~6.8]{BN23}:

\begin{thm}  \mlabel{prop:2-12-BN23} For the real bilinear form
  $\cH^\infty \times \cH^{-\infty} \to \R,
  (\xi, \psi) \mapsto \Im \la \xi, \psi \ra,$ 
 $\cH^{-\infty}_{\rm KMS}$ is the annihilator of
  $\sV' \cap \cH^\infty$ and
  \begin{equation}
    \label{eq:bn-distrib-kms}
    \cH^{-\infty}_{\rm KMS} \cap \cH = \sV.
  \end{equation}
\end{thm}

In our context, this result has the following interesting consequence.
Whenever $\eta\in \cH^{-\omega}(U_h)_{\rm KMS}$ is a distribution
vector, then the extended orbit map
$U^\eta \colon \oline{\cS_\pi} \to \cH^{-\infty}$ is
automatically weak-$*$-continuous,
and holomorphic on the interior, as an $\cH^{-\infty}$-valued map. 

\begin{cor} \mlabel{cor:3.24}
  $\cH^{-\omega}(U_h)_{\rm KMS} \cap \cH^{-\infty} \subeq \cH^{-\infty}_{\rm KMS}$.
\end{cor}

\begin{prf} We recall from Theorem~\ref{thm:closed-hyperfunc}
  that 
  \[ \cH^{-\omega}(U_h)_{\rm KMS}
  =  \{ \eta \in \cH^{-\omega}(U_h) \colon (\forall v \in \sV' \cap
  \cH^\omega(U_h))\ \eta(v) \in \R\}.\]
  For $\eta \in\cH^{-\omega}(U_h)_{\rm KMS} \cap \cH^{-\infty}$ we therefore
  have $\eta(v) \in \R$ for $v \in \sV' \cap  \cH^\omega(U_h)$.
  In view of Theorem~\ref{prop:2-12-BN23}, we have to show that
  $\eta(v) \in \R$ holds for any $v \in \sV' \cap \cH^\infty$.

  Now let $v \in \sV' \cap \cH^\infty$.
  To see that we also have $\eta(v) \in \R$,
  we consider the normalized Gaussians 
$\gamma_n(t) = \sqrt{\frac{n}{\pi}} e^{-n t^2}$ 
  and recall from \cite[Prop.~3.5(ii) and \S 5]{BN23} that 
  $v_n := U_h(\gamma_n) v \in \cH^\infty$ converges to $v$ in the Fr\'echet space
  $\cH^\infty$. As $\sV$ is $U_h$-invariant, so ist
  $\sV' = J\sV$. Further, 
  $U_h$ commutes with $J$ and $\gamma_n$ is real, so that 
  $v_n \in \sV'$, hence $\eta(v_n) \in \R$ because
  $v_n \in\sV' \cap \cH^\omega(U_h)$.
  Finally the assertion follows from $\eta(v_n) \to  \eta(v)$.
\end{prf}

 \begin{cor} $\sV^\infty := \sV \cap \cH^{\infty}$ is a
      $\g$-invariant subspace of $\cH^\infty$ and 
$\cH^{-\infty}_{\rm KMS}$ is a $\g$-invariant subspace of $\cH^{-\infty}$.     
\end{cor}

\begin{prf} For the first assertion we observe that
  $\sV = \cH_{\rm KMS}$ (\cite[Prop.~2.1]{NOO21}) implies 
  \[ \sV^\infty
  =   \cH_{\rm KMS} \cap \cH^\infty 
  \subeq \cH^{-\omega}_{\rm KMS} \cap \cH^\infty,\]
and the space on the right is $\g$-invariant by Lemma~\ref{lem:3.11}.
Therefore the $\g$-invariant subspace of $\cH^\infty$ generated by
$\sV^\infty$ is contained in 
\[ \cH^{-\omega}_{\rm KMS} \cap \cH^\infty
  \subeq \cH^{-\omega}_{\rm KMS} \cap \cH = \sV\]
(Corollary~\ref{cor:n5}). Therefore $\sV^\infty$ is $\g$-invariant.

Replacing $h$ by $-h$, it follows that
$\sV' \cap \cH^\infty$ is also $\g$-invariant, so that the
$\g$-invariance of $\cH^{-\infty}_{\rm KMS}$, the annihilator
of $\sV' \cap \cH^\infty$, follows from Theorem~\ref{prop:2-12-BN23}. 
\end{prf}

\begin{lem} \mlabel{lem:2.2}
  Under {\rm Hypothesis~(H2)}, for any $v \in \sE_K = \cE^J$,
  the extended orbit map
  \[ U_h^v \colon \oline{\cS_{-\frac{\pi}{2}, \frac{\pi}{2}}}
  = \Big\{ z \in \C \colon |\Im z| \leq \frac{\pi}{2}\Big\} \to \cH^{-\infty}, \quad
  z \mapsto e^{-z \partial U(h)} v \]
    is continuous and holomorphic on the interior, with respect to the
  weak-$*$-topology on $\cH^{-\infty}$.
\end{lem}

\begin{prf}
  From Hypothesis~(H2), 
  we know that $\sE_H \subeq \cH^{-\infty}$.
  Further,
  \[ \cH^{-\omega}(U_h)_{\rm KMS} \cap \cH^{-\infty} \subeq \cH^{-\infty}_{\rm KMS} \]
  by Corollary~\ref{cor:3.24} and
$\sE_H \subeq \cH^{-\omega}(U_h)_{\rm KMS}$ 
  by Lemma~\ref{lem:k.1}. 
  For $\eta := \beta^+(v) \in \sE_H$, we then have
  by Proposition~\ref{prop:2.9} 
\[ U_h^v(z)  = U_h^\eta\Big(z + \frac{\pi i}{2}\Big), \] 
  so that the assertion follows from $\eta\in \cH^{-\infty}_{\rm KMS}$.
\end{prf}

\begin{thm} {\rm(Reeh--Schlieder Theorem)} \mlabel{thm:RS}
  Let $(U,\cH)$ be a unitary representation of $G$ and
  $\sF_K \subeq \cH^{[K]}$ be a real linear subspace. Suppose that
  \begin{itemize}
  \item $\sF_K \subeq \bigcap_{|t| < \pi/2} \cD(e^{it \partial U(h)})$, that the
    limit $\beta^+(v) = \lim_{t \to \pi/2} e^{-it \partial U(h)}v$
    exists for every $v \in \sF_K$
    in the weak-$*$-topology of     $\cH^{-\infty}$, and that
  \item   $U(G)\sF_K$ is total in $\cH$.
  \end{itemize}
  Then, for every non-empty open subset
  $\cO \subeq G$ and $\sF_H := \beta^+(\sF_K)$,
  the subspace $\sH_{\sF_H}^G(\cO)$ is total
in $\cH$.
\end{thm}

\begin{prf} Let $\xi \in \cH$ be orthogonal to
  $\sH_{\sF_H}^G(\cO)$ and $\cO \not=\eset$. We have to show that $\xi = 0$.

  \nin {\bf Smooth case:} First we assume that $\xi$ is a smooth vector,
  so that the orbit map
  \[ U^\xi \colon G \to \cH^\infty, \quad g \mapsto U(g)\xi \]
  is smooth, hence in particular
  continuous with respect to the Fr\'echet topology on $\cH^\infty$
  (\cite[Cor.~4.5]{Ne10}). Let
$U_h^v(z) = e^{z \partial U(h)} v$ be as above. 
Then Lemma~\ref{lem:cartes} implies that the map
  \[ F \colon \oline{\cS_{-\frac{\pi}{2},\frac{\pi}{2}}}  \times G \to \C, \quad
F(z,g) := \la U(g)^{-1}\xi, U_h^v(-z) \ra \]
    is continuous for every $K$-finite vector $v \in \sF_K$.
    We conclude that, for every relatively compact open subset $\cO_c \subeq G$, the map 
    \[ F^\wedge  \colon \oline{\cS_{-\frac{\pi}{2},\frac{\pi}{2}}} \to C(\cO_c,\C), \quad
     F^\wedge(z)(g) := F(z,g) =  \la U(g)^{-1}\xi, U_h^v(-z) \ra \]
      is continuous with respect to the topology of uniform convergence on
      the Banach space $C^b(\cO_c,\C)$ of bounded continuous functions on
      $\cO_c$, hence in particular locally bounded.
      Moreover, for every fixed $g \in \cO_c$, the map
      $z \mapsto F^\wedge(z)(g)$
      is holomorphic on the interior of the strip,
      so that \cite[Cor.~A.III.3]{Ne00} implies that
      $F^\wedge$ is holomorphic on the interior of the strip.

      Suppose that $\oline{\cO_c} \subeq \cO$
      and choose $\eps > 0$ in such a way that
      \begin{equation}
        \label{eq:wackel}
        \oline{\cO_c \exp([-\eps,\eps]h)} \subeq \cO.
      \end{equation}
      We claim that $F^\wedge(t + \pi i/2) = 0$ for $t \in \R$. In fact,
      \[ U_h^v(-t - \pi i/2)
      = U(\exp(-th)) \beta^+(v).\]
      For $g \in \cO_c$ we therefore have
      \begin{align*}
 F^\wedge\Big(t + \frac{\pi i}{2}\Big)(g)
&      = \la U(g)^{-1} \xi,  U^{-\infty}(\exp(-th)) \beta^+(v) \ra\\
        &      = \la \xi,  U^{-\infty}(g\exp(-th)) \beta^+(v) \ra.
      \end{align*}
    This is a continuous function of $g$, and we now show that it
    vanishes on $\cO_c$ for every $t$ in $[-\eps,\eps]$.
    For every test function $\phi \in C^\infty_c(\cO_c,\R)$
      and the translated test function
      $\psi(g) = \phi(g \exp(th))$, which is supported in $\cO$,
      we have 
      \begin{align*}
\int_G \phi(g) F^\wedge\Big(t + \frac{\pi i}{2}\Big)(g)\, dg
& =\int_G \phi(g)\la \xi,  U^{-\infty}(g\exp(-th)) \beta^+(v) \ra\, dg\\
        &=\la \xi, \int_G  \phi(g)U^{-\infty}(g\exp(-th))\, dg\ \beta^+(v) \ra\\
& =\la \xi, U^{-\infty}(\psi) \beta^+(v) \ra
 \in \la \xi, \sH^G_{\sF_H}(\cO) \ra = \{0\}.
      \end{align*}
      
      We conclude that $F^\wedge$ vanishes on the boundary interval
      $\frac{\pi i}{2} + [-\eps,\eps]$.
      The Schwarz Reflection Principle now implies that
      $F^\wedge = 0$. For $z = 0$, we obtain
      in particular that
      $\xi \bot U(\cO_c) \sF_K$. Since $\sF_K$
      consists of analytic vectors,
      $U(\cO_c) \sF_K$ and $U(G) \sF_K$
        generate the same closed subspace. We thus obtain  
      $\xi \bot U(G)\sF_K$ and hence that $\xi = 0$.

      \nin {\bf General case:} Now we drop the assumption that $\xi$ is a smooth vector.
      We pick a non-empty, relatively compact subset $\cO_c \subeq \cO$ and observe that there
      exists a compact symmetric
      $e$-neighborhood $\cO_e \subeq G$ with $\cO_e \cO_c \subeq \cO$.  Then
      \[
      \la \xi, U(C^\infty_c(\cO_e,\R)) \sH_{\sF_H}^G(\cO_c) \ra
      \subeq \la \xi, \sH_{\sF_H}^G(\cO) \ra = \{ 0\},\]
      so that, for any test function
      $\psi \in C^\infty_c(\cO_e,\C)$, the smooth vector
      $U(\psi)^*\xi$ is orthogonal to $\sH_{\sF_H}^G(\cO_c)$,
      hence zero by the smooth case. As 
      $C^\infty_c(\cO_e,\C)$ contains an approximate identity of
      $G$, this implies  that~$\xi = 0$.   
\end{prf}

\subsection{The Bisognano--Wichmann property}
\mlabel{subsec:4.2}

Let $q \colon G \to G/H$ be the quotient map and
\[ W^G := q^{-1}(W) \quad \mbox{ for } \quad
W := W_M^+(h)_{eH},\] 
where the positivity domain $W_M^+(h)_{eH}$ is defined in \eqref{def:WM}.
We want to show that
\[ \sH_{\sE_H}^M(W) = \sH_{\sE_H}^G(W^G) = \sV,\]
where $\sV$ is the standard subspace specified as in 
\eqref{eq:mod-obj} by  
\[   \Delta_\sV = e^{2\pi i \partial U(h)} \quad \mbox{ and } \quad J_\sV = J.\]

\subsubsection{A realization as a space of holomorphic functions}

In this subsection we describe a realization of
the representation $(U,\cH)$ in the space 
$\cO(\bE^{\rm op})$ of holomorphic functions on the holomorphic vector 
  bundle   $\bE^{\rm op}$ over $\Xi^{\rm op}$,
  obtained by flipping the sign of the complex structure on $\bE$ 
  (cf.~Proposition~\ref{prop:2.8}).
We therefore have by Hypothesis~(H1) an
antiholomorphic map 
 \begin{equation}
    \label{eq:cE-extendb}
 \Psi_\cE \colon \bE^{\rm op} \to \cH, \quad
 [g,ix,v] \mapsto U(g) e^{i\partial U(x)} v 
  \end{equation}
that is fiberwise antilinear. This leads to a complex linear map 
\[ \Phi \colon \cH \to \cO(\bE^{\rm op}), \quad
  \Phi(\xi)(v) := \la \Psi_\cE(v), \xi \ra.\]
As $U(G)\cE$ spans a dense subspace of $\cH$, it 
maps $\cH$ isometrically onto a
reproducing kernel Hilbert space
$\cH_Q \subeq \cO(\bE^{\rm op})$
whose kernel $Q \colon \bE^{\rm op} \times \bE^{\rm op} \to \C$ is 
sesquiholomorphic. For the
evaluation functionals $Q_w$, $w \in \bE$, we find with
\[ \la \Psi_\cE(w), \xi \ra = \Phi(\xi)(w)
= \la Q_w, \Phi(\xi) \ra
= \la \Phi^{-1}(Q_w), \xi \ra
 \]
the relation $\Phi^{-1}(Q_w) = \Psi_\cE(w)$, resp.,
\begin{equation}
  \label{eq:2.4.1} Q_w = \Phi(\Psi_\cE(w)), \quad w \in \bE.
\end{equation}

The map $\Phi$ is $G$-equivariant with respect to the natural
$G$-action on $\cO(\bE^{\rm op})$: 
\[ \Phi(U(g)\xi)(w)
= \la U(g)^{-1}\Psi_\cE(w), \xi \ra
= \la \Psi_\cE(g^{-1}.w), \xi \ra
= \Phi(\xi)(g^{-1}.w).\] 

If $J\cE = \cE$, then we obtain with Remark~\ref{rem:hyp1-invol}
and the antiholomorphic involution
\[ \tau_\bE([g,ix,v]) = [\tau_h(g), -i \tau_h(x), Jv]\]
on $\bE$ the relation 
\begin{align*} 
 \Phi(J\xi)([g,ix,v])
&= \la \Psi_\cE([g,ix,v]), J\xi \ra
= \oline{\la J\Psi_\cE([g,ix,v]), \xi \ra}\\
&= \oline{\la \Psi_\cE([\tau_h(g),-i \tau_h(x), J v]), \xi \ra}
= \oline{\Phi(\xi)([\tau_h(g),-i \tau_h(x), J v])}, 
\end{align*}
so that
\begin{equation}
  \label{eq:4.2}
  \Phi(J\xi) = \oline{\Phi(\xi) \circ \tau_\bE}.
\end{equation}

We now obtain 
\[ J Q_w
= J\Phi(\Psi_\cE(w))
= \Phi(J \Psi_\cE(w))
= \Phi(\Psi_\cE(\tau_\bE(w))) 
= Q_{\tau_\bE(w)}.\]
We conclude that $w \in \bE^{\tau_\bE}$ implies
$Q_w \in \cH_Q^J$.
For $w = [g,ix,v]$, the relation
$\tau_\bE(w) = w$ is equivalent to
\[ [g,ix,v] = [\tau_h(g),-i\tau_h(x), Jv],\]
which means that
$q_\bE(w) \in \Xi^{\oline\tau_h}$ and that
$Jv = v$ holds in the fiber $\bE_{[g,ix]}$ over
the $\oline\tau_h$-fixed point $[g,ix] \in \Xi$. 

For $m \in \Xi^{\oline{\tau}_h}$ the modular flow
$\alpha_{it}(m) = \exp(ith).m$ is defined in $\Xi$ for $|t| < \frac{\pi}{2}$
(cf.\ \cite[\S 8]{MNO23b}) and the same holds for its lift
to the bundle $\bE$, given by 
\[ \alpha_t([g,ix,v]) = [\exp(th)g,ix,v].\]
So we obtain for each $w \in \bE^{\tau_\bE}$ a holomorphic extension
\[ \alpha^w \colon \cS_{\pm \pi/2} \to \bE^{\rm op}.\] 
For each $w \in \bE^{\tau_\bE}$ we thus obtain by
Hypothesis~(H1)  with Lemma~\ref{lem:k.1} a limit 
\begin{equation}
  \label{eq:2.13} \beta^\pm(\Psi_\cE(w))
= \lim_{t \to \mp \pi/2} e^{i t \partial U(h)} \Psi_\cE(w)
\in \cH^{-\omega}(U_h)_{\rm KMS},
\end{equation}
where $\cH^{-\omega}(U_h)_{\rm KMS}$ is 
  defined as in Definition~\ref{def:kms-subspace} with respect to the
  weak-$*$-topology on $\cH^{-\omega}(U_h)$
  (see also \eqref{eq:kms-spaces-intro} in the introduction).

  \subsubsection{Verification of the Bisognano-Wichmann property}

  In this section we show that the  Bisognano-Wichmann property
    introduced in the introduction holds
    for the net $\sH^M_{\sE_H}(\cO)$. The main results are collected in
    Theorem~\ref{thm:4.9}.

\begin{lem} \mlabel{lem:2.16}
Under {\rm Hypothesis~(H2)}, the map
\[ \Upsilon := \beta^+ \circ \Psi_\cE \colon \bE^{\tau_\bE} \to \cH^{-\omega}(U_h)_{\rm KMS} \] 
has the following properties:
\begin{itemize}
\item[\rm(a)] It is weak-$*$-analytic.   
\item[\rm(b)] It is $G^h_e$-equivariant. 
\item[\rm(c)] The automorphism $\zeta = e^{-\frac{\pi i}{2} \ad h}$
  of $\g_\C$ maps $i\Omega_{\fh_\fp} := i\Omega_\fp \cap \fh$ to
\[ \Omega_{\fq_\fk} := \{ x \in \fq_\fk \colon \Spec(\ad x) \subeq (-\pi/2,\pi/2)i\},\]  and
  for $x \in \Omega_\fp$ and $v \in \sE_K$, 
  we have
  \begin{equation}
    \label{eq:ana-ext1}
    U^{-\omega}(\exp \zeta(i x))  \beta^+(v) 
    =  \beta^+(e^{i \partial U(x)}v). 
  \end{equation}
\item[\rm(d)] $U^{-\omega}(W^G) \beta^+(\sE_K)\subeq \cH^{-\omega}(U_h)_{\rm KMS}$. 
\end{itemize}
\end{lem}

\begin{prf} (a) Let  $\xi \in \cH^\omega(U_h)$ and $w \in\bE^{\tau_\bE}$.
  Then there exists an $\eps \in (0,\pi/2)$ with 
  $\xi \in \cD(e^{-i \eps \partial U(h)})$. Let 
$\xi_\eps := e^{-i \eps \partial U(h)}\xi.$ 
  Then
  \begin{align*}
\la \xi, \Upsilon(w) \ra
&  = \la e^{i \eps \partial U(h)}  \xi_\eps, \Upsilon(w) \ra 
 = \la \xi_\eps, e^{i \eps \partial U(h)}  \Upsilon(w) \ra\\
& = \la \xi_\eps, \Upsilon(\alpha_{i\eps} w) \ra
  = \la \xi_\eps, \Psi_\cE(\alpha_{i(\eps - \pi/2)} w) \ra, 
  \end{align*}
  so that the assertion follows from the analyticity of
  $\Psi_\cE$ on $\bE$ and the analyticity of the map
  \[   \alpha_{i(\eps - \pi/2)} \colon \bE^{\tau_\bE} \to \bE.\]
 Finally we observe that, by Theorem~\ref{thm:closed-hyperfunc},
    the subspace $\cH^{-\omega}(U_h)_{\rm KMS}$ is closed in 
    $\cH^{-\omega}(U_h)$ and since $\Upsilon$ takes values in this
    subspace, it is real
    analytic as an $\cH^{-\omega}(U_h)_{\rm KMS}$-valued map.
  
    \nin (b) follows from the fact that the action of the subgroup
    $G^h_e = G^{\tau_h}_e$ on $\bE$ preserves the fixed point set 
  $\bE^{\tau_\bE}$ and commutes with $U_h(\R)$.

  \nin (c) First we recall from Proposition~\ref{prop:2.3}(b)
  that the orbit maps of $U^\omega$ are analytic, hence that the
  orbit maps of $U^{-\omega}$ are weak-$*$-analytic. Hence,  
  for a fixed $v \in \cE^J$,   the map
  \begin{equation}
    \label{eq:3.12.1}
 \Omega_\fp \to \cH^{-\omega}(U), \ \ 
 x \mapsto
 U^{-\omega}(\exp \zeta(i x))  \beta^+(v) 
= U^{-\omega}(\exp \zeta(i x)) \Upsilon([e,0,v])
  \end{equation}
  is weak-$*$-analytic.  
  Further, (a) and the weak-$*$-continuity of the inclusion
  $\cH^{-\omega}(U_h) \to \cH^{-\omega}(U)$ imply  that the map
  \begin{equation}
    \label{eq:3.12.2}
 \Omega_\fp \to \cH^{-\omega}(U), \quad
 x \mapsto \Upsilon([e,i x,v])
  \end{equation}
  is analytic.
  The $n$-th order terms in the Taylor expansion of
  \eqref{eq:3.12.1} in $0$ are given by
  \[ \frac{1}{n!}\dd U^{-\omega}(\zeta(i x))^n \beta^+(v) \]
and for   \eqref{eq:3.12.2} by
\[ \frac{1}{n!} \beta^+(\dd U(i x)^n v) 
= \frac{1}{n!} \dd U^{-\omega}(i \zeta(x))^n \beta^+(v),\]
where we have used Proposition~\ref{prop:exten}(d)
for the last equality.
We conclude that both maps have the same Taylor expansion
in $[e,0,v]$, hence that they coincide because their domain
  is connected.
This proves \eqref{eq:ana-ext1}.

\nin (d) As the right hand side of \eqref{eq:ana-ext1}
is contained in the subspace
$\cH^{-\omega}(U_h)_{\rm KMS}$ by \eqref{eq:2.13},
\eqref{eq:ana-ext1} implies that
\begin{equation}
  \label{eq:4.8}
   U^{-\omega}(\exp(\Omega_{\fq_\fk}))  \beta^+(\sE_K)
   \subeq \cH^{-\omega}(U_h)_{\rm KMS}.
\end{equation}
Now $W^G = G^h_e \exp(\Omega_{\fq_\fk}) H$
by \eqref{eq:fiber-diffeo} in Appendix~\ref{app:b}.
implies 
\begin{align*}
 U^{-\omega}(W^G) \beta^+(\sE_K)
&= U^{-\omega}(G^h_e) U^{-\omega}(\exp(\Omega_{\fq_\fk})) U^{-\omega}(H) \beta^+(\sE_K)\\
 &\ {\buildrel {\ref{prop:exten}(e)} \over =}\ U^{-\omega}(G^h_e) U^{-\omega}(\exp(\Omega_{\fq_\fk}))
 \beta^+(\sE_K)\\
&\subeq  U^{-\omega}(G^h_e) \cH^{-\omega}(U_h)_{\rm KMS} 
\subeq  \cH^{-\omega}(U_h)_{\rm KMS}. 
\qedhere\end{align*}
\end{prf}

The following proposition  leads to a characterization
of those subspaces $\sE$ with $\sH_\sE^M(W) \subeq \sV$ in
Corollary~\ref{cor:3.24}. 

\begin{prop} \mlabel{prop:4.8} For an open subset
  $\cO \subeq G$ and a real subspace $\sE \subeq \cH^{-\infty}$,
  the following are equivalent:
  \begin{itemize}
  \item[\rm(a)]   $\sH_\sE^G(\cO) \subeq \sV$. 
  \item[\rm(b)] For all $\phi \in C^\infty_c(\cO,\R)$ we have
    $U^{-\infty}(\phi)\sE \subeq \sV$. 
  \item[\rm(c)] For all $\phi \in C^\infty_c(\cO,\R)$ we have
    $U^{-\infty}(\phi)\sE \subeq \cH^{-\infty}_{\rm KMS}.$ 
  \item[\rm(d)]  $U^{-\infty}(g) \sE \subeq \cH^{-\infty}_{\rm KMS}$
    for every $g \in \cO$.
  \end{itemize}
\end{prop}

\begin{prf}  It is clear that (a) implies (b) by the   definition
  of $\sH^G_\sE(\cO)$. (b) implies (c)  because     $U^{-\infty}(\phi)\sE \subeq \cH$
  and  
  $\sV = \cH \cap \cH^{-\infty}_{\rm KMS}$
  (Theorem~\ref{prop:2-12-BN23}). 
  
  For the implication (c) $\Rightarrow $ (d), let $(\delta_n)_{n \in \N}$ be a $\delta$-sequence in $C^\infty_c(G,\R)$.
  Then $U(\delta_n)\xi \to \xi$ in $\cH^\infty$ and hence also in $\cH^{-\infty}$.
  It follows in particular that
  \[U^{-\infty}(\delta_n * \delta_g) \eta
  =  U^{-\infty}(\delta_n) U^{-\infty}(g) \eta \to  U^{-\infty}(g) \eta
  \quad \mbox{ for } \quad \eta\in \cH^{-\infty}.\] 
  Hence $\oline\sV \subeq \cH^{-\infty}_{\rm KMS}$, resp., the closedness of
  $\cH^{-\infty}_{\rm KMS}$, shows that (c) implies (d).
  Here we use that $\delta_n * \delta_g \in C^\infty_c(\cO,G)$
  for $g \in \cO$ if $n$ is sufficiently large. 

As the $G$-orbit maps in $\cH^{-\infty}$ are continuous
and $\cH^{-\infty}_{\rm KMS}$ is closed, hence stable under integrals over
compact subsets and $U^{-\infty}(C_c^\infty (\cO,\R))\cH^{-\infty} \subset \cH^\infty$, we see that (d) implies (a).
\end{prf}

\begin{prop} \mlabel{prop:BW-incl}
  Under {\rm Hypothesis~(H2)},
  the space $\sE_H = \beta^+(\sE_K)$
  consists of distribution vectors and 
$\sH_{\sE_H}^G(W^G) \subeq \sV$.   
\end{prop}

\begin{prf}
By Lemma~\ref{lem:2.16}(d), we have
$U^{-\omega}(W^G) \sE_H \subeq \cH^{-\omega}(U_h)_{\rm KMS}.$
If, in addition, $\sE_H \subeq \cH^{-\infty}$, then
Corollary~\ref{cor:3.24}  
yields 
\[ U^{-\infty}(W^G) \sE_H = 
    U^{-\omega}(W^G) \sE_H
    \subeq \cH^{-\omega}(U_h)_{\rm KMS} \cap \cH^{-\infty}
\subeq \cH^{-\omega}_{\rm KMS} \cap \cH^{-\infty}
    =  \cH^{-\infty}_{\rm KMS},\]  
and therefore Proposition~\ref{prop:4.8}(a),(d) imply the assertion.    
\end{prf}

We are now ready to state our main theorem for the case where
the simple Lie group $G$ is linear, i.e., $\eta : G\to G_\C$
is injective. 

\begin{thm} \mlabel{thm:4.9}
  Let $(U,\cH)$ be an irreducible antiunitary
  representation of
  \[ G_{\tau_h} := G \rtimes \{\1,\tau_h\},\] 
  let $\cE$ be a finite-dimensional subspace invariant under
  $K$ and $J$, and $\sE_K := \cE^J$.
  If $G$ is linear, then {\rm Hypothesis (H2)} is satisfied, so that
  $\sE_H = \beta^+(\sE_K) \subeq \cH^{-\infty}$.
  Then the net $\sH^M_{\sE_H}$ on the non-compactly causal
  symmetric space satisfies {\rm (Iso), (Cov), (RS)} and
  {\rm (BW)}, where $W = W_M^+(h)_{eH}$ is the connected component
  of the positivity domain of $h$ on $M$, containing the base point.
\end{thm}

\begin{prf} (Iso) and (Cov) are trivially satisfied and
  (RS) follows from Theorem~\ref{thm:RS}. It remains to
  verify (BW). 

  As $W$ is invariant under the modular flow on $M$, the subspace
  $\sH := \sH_{\sE_H}^M(W) = \sH_{\sE_H}^G(W^G)$ is invariant under
  $U(\exp th) = \Delta_\sV^{-it/2\pi}$ for
$t \in \R$. Further, $\sH \subeq \sV$
by Proposition~\ref{prop:BW-incl},
so that   $\sH$ is separating,
hence standard because it is also  cyclic by the
Reeh--Schlieder Theorem~\ref{thm:RS}. 
Then \cite[Lemma~3.4]{NO21} implies that
$\sH = \sV$ because 
$\sH$ is invariant under the modular group $U(\exp \R h)$ of $\sV$.
\end{prf}

  \begin{rem} (``Independence'' of the net from $H$)
    In the context of Theorem~\ref{thm:4.9},
    the real subspace
    $\sE_H \subeq \cH^{-\infty}$ is invariant under
    $U^{-\infty}(H)$.
    For any open subset $\cO_G \subeq G$ we therefore have
    \[ \sH^G_{\sE_H}(\cO) = \sH^G_{\sE_H}(\cO H)  \]
    by \cite[Lemma~2.11]{NO21}.
    Hence the inclusions $H_{\rm min} \subeq H \subeq H_{\rm max}$
    from the introduction to Section~\ref{sec:3} imply that
    \[ \sH^G_{\sE_H}(\cO) 
      = \sH^G_{\sE_H}(\cO H) = \sH^G_{\sE_H}(\cO H_{\rm max}).\]
    Here we use that the real subspace $\sE \subeq \cH^{-\infty}$ is 
      invariant under $H_{\rm max}$ because Proposition~\ref{prop:exten}(e) also
      applies to $H_{\rm max}$, which is the maximal choice for $H$. 
    For the covering 
    \[  q_m \colon G/H     \to   M_{\rm min} := G/H_{\rm max} \] 
   it therefore follows that the net  
   $\sH^M_{\sE_H}$ on $M$ can be recovered from its pushforward
   $\sH^{M_{\rm min}}_{\sE_H}$ to $M_{\rm min}$ because 
   \[ \sH^M_{\sE_H}(\cO)
     =\sH^M_{\sE_H}(q_m^{-1}(q_m(\cO)))  
     =\sH^{M_{\rm min}}_{\sE_H}(q_m(\cO)) \]
   for any open subset $\cO \subeq M$.
    
\end{rem}

\subsection{Locality}
\mlabel{subsec:locality}

In the context of quantum field theories on space-time manifolds 
one considers for nets $\cO \mapsto \sH(\cO)$ of real subspaces the
{\it locality condition} 
\[ \cO_1 \subeq \cO_2' \quad \Rarrow \quad \sH(\cO_1) \subeq \sH(\cO_2)', \]
where $\sH(\cO_2)' := \sH(\cO_2)^{\bot_\omega}$ is the symplectic orthogonal
space with respect to $\omega = \Im \la \cdot,\cdot \ra$ and
$\cO_2'$ is the {\it causal complement of $\cO_2$}, i.e.,
the largest open subset that cannot be connected to $\cO_2$ with
causal curves.

Presently, we do not know if this strong locality
condition is satisfied for our nets $\sH^M_{\sE_H}$ on
non-compactly causal symmetric spaces $M = G/H$.
However, there is some information that can be formulated as follows. 
Let $W := W_M^+(h)_{eH}$ be the wedge region in $M$ associated
to the Euler element $h$. To see natural candidates for its
``causal complement''  $W'$, we note that our construction implicitly
uses an embedding of $M$ into the ``boundary'' of the crown domain
$\Xi$. We actually have two natural embeddings, that can be described by
\[ M_\pm := \alpha_{ \mp \pi i/2}(G/K) \subeq \partial \Xi \]
whenever $\Xi$ is embedded into a complex homogeneous space
$G_\C/K_\C$ and
\[ \alpha_z(m) = \exp(zh).m.\] 
Note that these embeddings are highly non-unique and have
different geometric properties.
In particular, the boundary $G$-orbits $M_\pm$ in $\partial \Xi$
may coincide or not.
Such embeddings are also studied
by Gindikin and Kr\"otz in \cite{GK02b} and in \cite[Thm.~5.4]{NO23},
where $\Xi$ is identified with a natural tube domain of~$M =G/H$.

Identifying $M$ with $M_+$, 
a natural candidate for the ``causal complement''  $W'$ is the domain
\[ W' = \alpha_{\pi i}(W) \subeq M_-.\]
For locality issues one now has to distinguish
between the two cases $M_+ = M_-$ and $M_+\not= M_-$. 
A~simple example with $M_+ \not= M_-$ arises for 
\[ M_+ = \R_+, \quad
  \mbox{ where } \quad \Xi = \C_+ = \{ z \in \C \colon \Im z> 0\} \] 
is the complex upper half plane (the crown of the Riemannian
space $i \R_+$), $\alpha_t(z) = e^t z$, and  $M_- = \R_- = (-\infty,0)$.
If $M = \R^{1,n}$ is $(n+1)$-dimensional Minkowski space,
$G = \R^{1,n} \rtimes \SO_{1,n}(\R)_e$ is the connected Poincar\'e group,
$W = \{ (x_0,\bx) \colon x_1 > |x_0|\}$ is the standard right wedge and
$\alpha_t$ is the corresponding group of boosts, then 
$\alpha_{\pi i}(W) = -W = W'$, $\Xi = \R^{1,n} + i V_+$
($V_+ \subeq \R^{1,n}$ is the open future light cone),
and $M_+ = M_-$. 
Although these examples do not come from simple groups,
they represent the two different situations in an
elementary way.

Our construction produces nets of real subspaces
$\sH^{M_+}_{\sE_H}$ on $M_+$ and
$\sH^{M_-}_{J\sE_H}$ on $M_-$. 
For the wedge regions we then derive from the (BW) property that
\[ \sH^{M_+}_{\sE_H}(W) = \sV \quad \mbox{ and } \quad
  \sH^{M_-}_{J\sE_H}(W') = J\sV = \sV'.\]
Let
\[ \cW(M_+) := \{ g.W \colon g \in G \} \]
denote the {\it wedge space of $M_+$} and,
likewise
\[ \cW(M_-)  := \{ g.W' \colon g \in G \} \]
the wedge space of $M_-$
(cf.\ \cite{MN21, MNO23b}).
Putting $(g.W)' := g.W'$ for $g \in G$, we obtain by covariance and isotony
the following property: 
\begin{itemize}
\item[\textrm($L_W$)] Wedge-Locality: If there exists a wedge domain
  $W_1$ such that   $\cO_1 \subeq W_1$ and $\cO_2 \subeq W_1'$, then 
  $\sH_\sE^M(\cO_1) \subeq \sH^{M_-}_\sE(\cO_2)'$.
\end{itemize}
In fact, $\cO_1 \subeq W_1 = g.W$ and $\cO_2 \subeq W_1' = g.W'$ lead to
\[ \sH_\sE^M(\cO_1) \subeq U(g)\sV
\quad \mbox{ and } \quad  \sH_\sE^{M_-}(\cO_2) \subeq U(g)\sV' = (U(g)\sV)'.\] 
Of course, this is most interesting if $M_+ = M_-$, so that duality
relates domains in the same homogeneous space. 
We plan to explore the locality condition  in subsequent work. 

\section{Covering groups of  $\SL_2(\R)$} 
\mlabel{sec:rank-one}

In this section we show by different methods that
Hypothesis~(H2) is also  satisfied for all connected
Lie groups $G$ with Lie algebra $\g = \fsl_2(\R)$.
Our argument is based on the observation
that, for any simple Lie group $G$ of real rank one
and any $K$-eigenvector $v$ in an irreducible $G$-representation
$(U,\cH)$, we have
\[ v \in \bigcap_{x \in \Omega_\fp} \cD(e^{i \partial U(x)}), \]
which by Proposition~\ref{prop:hyp1} is
Hypothesis~(H1) for $\cE = \C v$.
We further show that there exists a constant $C > 0$ such that
\begin{equation}
  \label{eq:qn1b}
  \|e^{it \partial U(h)}v\| \leq C  \Big(\frac{\pi}{2}-t\Big)^{-N} \quad \mbox{
    for } \quad 0 \leq t < \frac{\pi}{2}.
\end{equation}
This implies that the limit $\beta^+(v)$ is contained
in $\cH^{-\infty}(U_h) \subeq \cH^{-\infty}$
(Theorem~\ref{thm:E.4}). 
The inclusion $U(W^G) \beta^+(v) \subeq \cH^{-\infty}_{\rm KMS}$
is obtained by verifying that actually
\[ e^{i \partial U(\Omega_\fp)} v \subeq
  \bigcap_{|t| < \pi/2} \cD(e^{i t\partial U(h)}) \] 
and that \eqref{eq:qn1b} extends to any~$e^{i \partial U(x)} v$
for $x \in \Omega_\fp$ (Section~\ref{subsec:6.2}). 
Then Theorem~\ref{thm:4.9} applies and shows that
the net defined by $\sE_H := \R \beta^+(v)$
satisfies (Iso), (Cov), (RS) and (BW).

We address these issues as follows.
Let $G$ be a connected simple Lie group of real rank one,
$\g = \fk + \fp$ be a Cartan decomposition
and $\fa \subeq \fp$ a maximal abelian subspace, which is $1$-dimensional
by assumption. Accordingly, the restricted root system is either
$\Sigma(\g,\fa) = \{ \pm \alpha\}$ or 
$\Sigma(\g,\fa) = \{ \pm \alpha, \pm \alpha/2\}$.
We choose a basis element $h\in \fa$ normalized by $\alpha(h) = 1$.
We now consider an irreducible unitary representation
$(U,\cH)$ of $G$ and assume that $v \in \cH$ is a normalized
$K$-eigenvector. We write
\[ \chi \colon  K \to \T \]
for the corresponding character. Then
\[ \phi \colon G \to \C, \quad \phi(g) := \la v, U(g) v \ra \]
is called a {\it $\chi$-spherical function} 
and a {\it spherical function} if $\chi = 1$. It satisfies
\[ \phi(k_1 g k_2) = \chi(k_1) \phi(g) \chi(k_2) \quad \mbox{ for } \quad
  g \in G, k_1, k_2 \in K.\]
As $G = KAK = K \exp(\R h) K$, it
is determined by its values on the one-parameter group
defined by $a_t := \exp(th)$. We shall see below
that the values $\phi(a_t)$ are given by hypergeometric
functions and that \eqref{eq:qn1b} will follow from estimates for
\begin{equation}
  \label{eq:sphe-norm}
 \|e^{it \partial U(h)}v\|^2 = \phi(a_{2it}) \quad \mbox{ for } \quad
 |t| < \pi/2.
\end{equation}

Clearly, non-trivial characters $\chi$ occur only if
the Lie algebra $\fk$ has non-trivial center, i.e.,
if $\g$ is hermitian ($G/K$ is a bounded complex symmetric domain).
Then the rank-one condition implies $\g \cong \su_{1,n}(\C)$ for some
$n \in \N$. For $n = 1$, we have $\su_{1,1}(\C) \cong \fsl_2(\R)
\cong \so_{1,2}(\R)$. The only rank-one Lie algebras containing Euler elements
are those for which $\alpha/2$ is not a root, i.e.,
$\g = \so_{1,n}(\R)$. Note that only $\so_{1,2}(\R)$ is hermitian,
and, for $n > 2$, the character $\chi$ is trivial.

As the spherical functions can be written as Gauss hypergeometric
functions ${}_2F_1$ we start our discussion by
a short overview of those (Section~\ref{subsec:5.1}).
We then apply this to the spherical functions.
For $\chi$-spherical functions we use results of Shimeno (\cite{Sh94}).
For $\chi = 0$ the estimates for the spherical functions
$\phi(a_{it})$ for $t \to \pm \pi i$ also follow from 
\cite[Thm.~7.2]{KO08}.

We actually expect that the results
that we use in this section for rank-one spaces
hold more generally: 
\begin{conj} Let $(U,\cH)$ be an irreducible unitary
  representation  of the connected real simple Lie group
  $G$, $h \in \g$ an Euler element,
  and $v \in \cH^{[K]}$ be a $K$-finite vector.
  Then there exists a constant $C > 0$  and $N > 0$ such that
\begin{equation}
  \label{eq:qn-conj}
  \|e^{it \partial U(h)}v\| \leq C  \Big(\frac{\pi}{2}-t\Big)^{-N} \quad \mbox{
    for } \quad 0 \leq t < \frac{\pi}{2}.
\end{equation}
\end{conj}

\subsection{Growth estimates and hypergeometric functions}
\mlabel{subsec:5.1}

\subsubsection{The hypergeometric functions}

Our standard references for this section are \cite[Ch.~II]{Er53}
and \cite[Ch.~9]{LS65}. 

The Gau\ss{} hypergeometric functions, or simply the
{\it hypergeometric function}, $\hgfabc $ is given by the series
\begin{equation}\label{def:hgf}
 {}_2 F_1(\beta, \alpha ;\gamma; z)=\hgfabc =\sum_{k=0}^\infty \frac{ (\alpha)_k (\beta)_k}{k! (\gamma)_k}\, z^k,\quad \gamma \not= 0, -1, -2, \ldots 
\end{equation}
where 
\[(a)_k =a(a+1)\cdots (a+k-1), \quad k =0, 1,2, \ldots.\]
We note the following, see \cite[Ch. II]{Er53} and \cite[Ch. 9]{LS65} 
\begin{lemma} 
\begin{itemize} 
\item[\rm (a)] If $\alpha $ or $\beta$ is contained in $-\N_0$, 
  then the series in \eqref{def:hgf} is finite and $\hgf$ is a polynomial function.
\item[\rm (b)] The series in \eqref{def:hgf} converges absolutely for $|z|<1$.
\item[\rm (c)] The hypergeometric function $\hgfabc$ is the unique solution to the differential equation
\[z(1-z)u^{\prime\prime} +(\gamma - (\alpha + \beta +1)z)u^\prime - \alpha\beta u=0\]
which is regular at $z=0$ and takes the value $1$ at that point.
\end{itemize}
\end{lemma}

The following can be found in \cite[p.~241,\ 246,\ 245]{LS65}. For the second
theorem we use in our formulation that
$\arg (1-z)< \pi $ is equivalent to $z\in \C\setminus [1,\infty)$. 

\begin{theorem} The hypergeometric function $\hgfabc$
  extends holomorphically to the slit plane $\C \setminus [1,\infty)$,
  and for fixed $z\in \C\setminus [1,\infty)$, the
  function
\[ (\alpha,\beta,\gamma) \mapsto \frac{\hgfabc}{\Gamma (\gamma)}  \]
is an entire function of $\alpha,\beta$ and $\gamma$.
\end{theorem}

In particular, the functions 
$\gamma \mapsto \hgfabc$ are meromorphic with simple poles at
$\gamma \in - \N_0$. 
Some more facts that we will use are collected in the following  theorem.

\begin{thm} \label{thm:gamma-limit} The following assertions hold:
  \begin{itemize}
  \item[\rm(i)] 
    If $\Re (\gamma - \alpha - \beta) >0$
   and $\gamma \not \in -\N_0$, then
   \[\lim_{t\to 1-}\ {}_2F_1(\alpha,\beta;\gamma;t)
          = \frac{\Gamma (\gamma)
      \Gamma (\gamma - \alpha-\beta )}{\Gamma (\gamma - \alpha)\Gamma(\gamma -\beta)}.\]
\item[\rm(ii)] 
  For $z \in \C \setminus [1,\infty)$, we have
\[\hgfabc= (1-z)^{\gamma - \alpha -\beta} \hgf (\gamma - \alpha ,\gamma - \beta; \gamma ; z).\]
\item[\rm(iii)] 
  If $\gamma =\alpha +\beta
 \not\in -\N_0$, then 
 \[\lim_{t\to 1-}\frac{{}_2F_1(\alpha ,\beta ;\gamma ;t)}
{-\log (1-t)}    
   = \frac{\Gamma (\alpha + \beta)}{\Gamma (\alpha ) \Gamma (\beta)} 
   = \frac{\Gamma (\gamma)}{\Gamma (\alpha ) \Gamma (\beta)}.\]
  \end{itemize}
\end{thm}

\begin{prf} (i) is \cite[p.~244 (9.3.4)]{LS65}.

  \nin (ii) is  \cite[p.~248]{LS65}. Here
  we use that $\arg (1-z) < \pi$
  if and only if $z\in \C\setminus [1,\infty)$
  (see also \cite[(15.4.20,15.4.22]{DLMF}).

  \nin (iii)  We use \cite[(15.4.21)]{DLMF} with $\gamma =\alpha + \beta$.
\end{prf}

If $\Re (\gamma - \alpha -\beta)<0$, then we can
use the hypergeometric identity (ii) to evaluate limits as in (i) because 
\[\gamma - (\gamma - \alpha) - (\gamma - \beta) = - \gamma +\alpha + \beta\]
has positive real part.


\subsubsection{Spherical functions}
\label{subsec:sphefunc}
  
We start with a discussion of spherical functions
of real rank-one groups (see the introduction to this section for notation).
Following the notation in \cite[p.~1158]{OP13}  
we fix a positive root $\alpha$ such that $\alpha$
and possibly $\alpha/2$ are the only 
positive roots and normalize $h \in \fa$ by $\alpha (h)= 1$ and identify $\fa_\C$ and  $\C$
using the isomorphism $x\mapsto \alpha (x)$ with inverse $z\mapsto zh$.

Let
\[ m_\alpha := \dim \fg_\alpha, \quad m_{\alpha/2} := \dim \fg_{\alpha/2}
  \quad \mbox{ and } \quad
  \rho :=\frac{1}{2}\left( m_\alpha +\frac{1}{2} m_{\alpha/2}\right).\]
We also let $a_t = \exp th$ and consider a spherical function
\[ \phi(g) = \la v, U(g) v\ra \]
for an irreducible unitary representation and a
$K$-fixed unit vector~$v$.
According to \cite[p. 1158]{OP13} and \cite[Ch.~IV, Ex.~B.8]{Hel84},
there exists a $\lambda \in \C$ such that 
\begin{equation}
  \label{eq:op13}
\phi(a_t) =   \varphi_\lambda(a_{t}) := {}_2F_1\Big(\rho + \lambda,\rho -\lambda;
\frac{m_{\alpha/2} + m_\alpha +1 }{2};-\sinh^2(t/2)\Big).
\end{equation}
By \eqref{def:hgf}, this function is constant $1$ if
  $\lambda = \pm \rho$.

The result of Kostant stated below 
  (see \cite[Thm.~1]{FJK79} for a rather direct proof) characterizes 
  the values of $\lambda$ that occur in our context, i.e.,
  those for which the corresponding spherical function is positive definite.
  
  \begin{thm} {\rm(Kostant's Characterization Theorem)}
    \mlabel{thm:kostant}
    Suppose that $\Sigma = \{  \alpha\}$ or $\{ \alpha, \alpha/2\}$. 
Let 
  \[  s_0 :=
    \begin{cases}
      \rho =\frac{1}{2}\left( m_\alpha +\frac{1}{2} m_{\alpha/2}\right)
      & \text{ if } m_{\alpha/2} = 0 \\      
\frac{1}{2}(1+ \frac{m_{\alpha/2}}{2}) \leq \rho & \text{ if } m_{\alpha/2} >  0.
    \end{cases}
\]    
Then the spherical function $\phi_\lambda$ is positive definite if and only if
\[  \lambda \in i\R \cup ([-s_0, s_0] \cup \{ \pm \rho\}).\]
\end{thm}

Note that $\phi_{\pm \rho} \equiv 1$ by \eqref{eq:op13},  and this function
is trivially positive definite.

{As they appear in \eqref{eq:op13}, we recall in the following table the
root multiplicities for the real rank-one simple Lie algebras
(cf.\ \cite{Hel78}):  
\begin{center}
\begin{tabular}{|l|l|l|l|l|}\hline
  $\g$ & $\so_{1,n}(\R)$ & $\su_{1,n}(\C)$ &  $\fu_{1,n}(\H)$
  &$\ff_{4(-20)}$   \\ \hline
$m_\alpha$ & $n-1$ & $1$ &  $3$ &  $7$ \\
$m_{\alpha/2}$ & $0$  & $2(n-1)$  & $4(n-1)$ & $8$\\
\hline 
\end{tabular} 
  
\end{center}
\vspace{4mm}}
It follows in particular that $m_\alpha = 1$ occurs only for
$\g = \su_{1,n}(\C)$. Furthermore, the only real rank-one
  Lie algebras  containing an Euler element are
  the Lie algebras  $\so_{1,n}(\R)$
  (see Table~3 in \cite[\S 4]{MNO23a}).
  Note that $\su_{1,1}(\C) \cong \so_{1,2}(\R)$.

\begin{theorem} The following assertions hold for
  $\lambda \not=\pm \rho$: 
\begin{itemize}
\item[\rm (1)] If $m_\alpha >1$, 
  then
    \[\lim_{t\to \pi-} \cos(t/2)^{m_\alpha -1}\cdot 
    \varphi_{\lambda} (a_{it}) = 2^{\frac{m_\alpha -1}{2}}
    \frac{\Gamma\big( \frac{m_{\alpha/2} + m_\alpha +1 }{2}\big)
      \Gamma\big(\frac{m_\alpha -1}{2}\big)}
    {\Gamma(\rho-\lambda)
      \Gamma(\rho+\lambda)}. \]
\item[\rm (2)]
  If $m_\alpha =1$, then 
\[\lim_{t \to \pi-} \frac{\varphi_{\lambda } (a_{it})}{-\log (\pi - t)}
  = \frac{2\cdot \Gamma\Big(1 + \frac{m_{\alpha/2}}{2}\Big)}{\Gamma (\rho-\lambda)
    \Gamma (\rho+\lambda)}.\] 
\end{itemize}
\end{theorem}

\begin{proof} We put
\[ a:= \frac{1}{4}(2+m_{\alpha/2}), \quad
  b:=\frac{1}{2}(m_\alpha -1) \quad \mbox{ and } \quad
  c:= \frac{1}{2}(m_{\alpha/2}+m_\alpha+1).\]
Using that 
\begin{align*}
  c  - (\rho + \lambda) - (\rho -\lambda) &
= c- 2 \rho  = \frac{1-m_\alpha}{2} = -b,  \quad \mbox{ and } \\
 c - (\rho \pm \lambda)=  \frac{m_{\alpha/2} + m_\alpha +1}{2}
                                            - (\rho \pm \lambda)
  & =\frac{2+m_{\alpha/2}}{4} \mp \lambda = a \mp \lambda, 
\end{align*}
we get with \eqref{eq:op13}
and Theorem~\ref{thm:gamma-limit}(ii): 
\begin{align}
\varphi_{\lambda} (a_{it}) 
&= {}_2F_1\Big(\rho + \lambda,\rho -\lambda;c;\sin^2(t/2)\Big)\nonumber\\
&= 2^{b}\cos^{-2b}(t/2)\cdot {}_2F_1\big(a - \lambda,
a + \lambda;c ; \sin^2(t/2)\big).\label{eq:sh}
\end{align} 

Now (1) follows from \eqref{eq:sh} and  Theorem \ref{thm:gamma-limit}(i)
  because
  \[ c - a =
    \frac{1}{2}(m_{\alpha/2}+m_\alpha+1)
    -  \frac{1}{4}(2 +m_{\alpha/2})
    =  \frac{m_{\alpha/2}+2 m_\alpha}{4} \] 
  and   
\[   c - (a - \lambda) - (a + \lambda)
  = c - 2a = \frac{m_{\alpha}-1}{2}. \] 
Note that, by Theorem~\ref{thm:kostant}, we have for $\lambda \in \R$ that 
\[ |\lambda|
  \leq \rho = \frac{1}{4}(2 m_\alpha + m_{\alpha/2})\]  
with equality only for $\lambda = \pm \rho$.
So the $\Gamma$-functions in the denominator are singular
only for $\lambda = \pm \rho$, and in  this
case $\phi_{\pm\rho} \equiv 1$. 

\nin For (2), we observe that
$m_\alpha = 1$ 
implies that 
$\rho = \frac{1}{2} + \frac{m_{\alpha/2}}{4}$
and $c = 2\rho$. 
So we have 
\[
\lim_{t \to \pi-} \frac{\varphi_{\lambda} (a_{it})}{-\log(\pi -t)}
= \lim_{t \to \pi-}
\frac{{}_2F_1\big(\rho + \lambda,\rho -\lambda;c;\sin^2(t/2)\big)}
{-\log(\pi -t)} .\] 
Now (2) follows from Theorem~\ref{thm:gamma-limit}(iii)
and the fact that
\begin{align*}
&  \lim_{t \to \pi-}\frac{-\log(1 - \sin^2(t/2))}{-\log(\pi - t)}
=  \lim_{t \to \pi-}\frac{\log(\cos^2(t/2))}{\log(\pi - t)}\\
&=  \lim_{t \to \pi-}\frac{\frac{1}{\cos^2(t/2)} (-\cos(t/2)\sin(t/2))}
    {\frac{-1}{\pi -t}} \\
&=  \lim_{t \to \pi-}\frac{\pi - t}{\cos(t/2)} 
     =  \lim_{t \to \pi-}\frac{-1}{-\sin(t/2)\frac{1}{2}}
     =  2.\qedhere
\end{align*}
\end{proof}

\subsubsection{$\chi$-spherical functions}

Now we turn to $\chi$-spherical functions where $\chi$ is non-trivial,
so that $\g \cong \su_{1,n}(\C)$.
This case was treated in   \cite{Sh94}. The general
case for $\chi \not=1$ and $G/K$ a bounded symmetric domain was discussed in
\cite{HS94}.  We use the normalization from \cite[pp. 383]{Sh94} and set
\[
  h :=\frac{1}{2}\begin{pmatrix}  0 & 0 & 1\\ 0 & \0_{n-1} & 0 \\ 1 & 0 & 0\end{pmatrix}
  \quad\text{and}\quad 
\wz =  i\begin{pmatrix}  1& 0 & \\ 0 &  -\frac{1}{n}\1_n \end{pmatrix} .\]
Then $\fa =\R h$ 
is maximal abelian in $\fp$ and $\fz(\fk) =\R \wz$ is the center of $\fk$.
Here the normalization of $\alpha $ is that
$\alpha(h) = 1$ 
and the positive roots are
$\alpha$ and, if $n\ge 2$, also $\alpha/2$. The
unitary characters of $K$ are parametrized by  
\[\chi_\ell (\exp t \wz)= e^{i \ell t},\quad \ell \in \R.\]
In view of \cite[p. 384]{Sh94},
the $\chi_\ell$-spherical functions take on 
$a_t=\exp (th)$ the form by 
\[\varphi_{\ell,\lambda} (a_{t}) = (\cosh t/2)^{-\ell} \cdot
  {}_2F_1 \Big(\frac{n-\ell +\lambda}{2}, \frac{n-\ell  -\lambda}{2};
  n;-\sinh^2(t/2)\Big)\]
and hence
\[(\cos t/2)^{\ell}\varphi_{\ell,\lambda} (a_{it})
  = {}_2F_1 \Big(\frac{n-\ell +\lambda}{2}, \frac{n-\ell  -\lambda}{2}; n;
\sin^2(t/2)\Big)\]

\begin{theorem} Assume that $\ell \not= 0$. Then we have 
  \[\lim_{t\to \pi-} \cos^{|\ell|}(t/2)\phi_{\ell, \lambda} (\exp it h)
    = 
    \frac{\Gamma(|\ell |) (n-1)!}
    {\Gamma (\frac{1}{2}(n+|\ell | -\lambda ))
      \Gamma  (\frac{1}{2}(n+|\ell | +\lambda ))}  .\]
 \end{theorem}

 \begin{proof}  
   Assume first that $\ell >0$. Then  $\Re (\gamma -\alpha - \beta)=\ell >0$
   and the claim follows from  Theorem~\ref{thm:gamma-limit}(i). 
    For $\ell <0$ we use Theorem~\ref{thm:gamma-limit}(ii)
   to write
 \begin{align*}
   \hgf \left(\frac{1}{2}(n-\ell\right.&\left. +\lambda),
\frac{1}{2}(n-\ell -\lambda);n;\sin^2(t/2)\right)\\
 &=\cos^{2\ell } (t/2)\cdot \hgf \left(\frac{1}{2}(n+\ell -\lambda),
 \frac{1}{2}(n + \ell +\lambda);n;\sin^2(t/2)\right)\\
 &=\cos^{-2|\ell | } (t/2)\cdot \hgf \left(\frac{1}{2}(n- |\ell | -\lambda),
 \frac{1}{2}(n -|\ell | +\lambda);n;\sin^2(t/2)\right)
 \end{align*}
 and the claim follows as above.
 \end{proof}

\begin{cor} \mlabel{cor:5.1}
  Let $(U,\cH)$ be an irreducible unitary representation
  of the rank-one group~$G$.
  Any $K$-eigenvector is contained
  in $\bigcap_{x \in \Omega_\fp} \cD(e^{i\partial U(x)})$
  and {\rm Hypothesis~(H2)} is satisfied.
\end{cor}

\begin{prf} The first assertion follows from
  \[ \cD(e^{i\partial U(\Ad(k)x)}) = U(k) \cD(e^{i\partial U(x)})
  \quad \mbox{ for } \quad k \in K, x \in \fp \]
  and $\Omega_\fp = \Ad(K) (-\pi/2,\pi/2)h$ for any Euler element
  $h \in \fp$. It follows in particular that, for each
  $K$-eigenvector $v$, the hyperfunction vector
  $\beta^+(v) \in \cH^{-\omega}$ is defined.
  Further the preceding discussion implies the existence
  of an $N > 0$ such that 
  \[ \sup_{|t| < \pi/2}
  \Big(\frac{\pi}{2}-t\Big)^N \| e^{i t \partial U(h)} v\| 
    = \sup_{|t| < \pi/2}
  \Big(\frac{\pi}{2}-t\Big)^N \phi(a_{2it})^{1/2}  < \infty.\]
This in turn implies with Theorem~\ref{thm:E.4} that
$\beta^+(v) \in \cH^{-\infty}(U_h) \subeq \cH^{-\infty}$, i.e.,
that Hypothesis~(H2) is also satisfied.
\end{prf}

\subsection{More general asymptotics} 
\mlabel{subsec:6.2}
We keep our assumption that $\g$ is of real rank-one. 
Let $v \in \cH$ be a $K$-eigenvector fixed by $J$ 
for which
\begin{equation}
  \label{eq:qN} q_N(v) := \sup_{|t| < \pi/2}
  \Big(\frac{\pi}{2}-t\Big)^N \| e^{i t \partial U(h)} v\| < \infty
\end{equation}
(Corollary~\ref{cor:5.1}). We then consider $\cE = \C v$ and
  $\sE_K = \R v$.

\begin{lem} The functions
\[ \ell([g,i x]) := \|U(g)e^{i\partial U(x)}v\|  =  \|e^{i\partial U(x)}v\|
  \quad \mbox{ on } \quad \Xi \cong G \times_K i \Omega_\fp\]
and
\begin{equation}
  \label{eq:ell1}
  \delta \colon \Xi \to \Big[0,\frac{\pi}{2}\Big), \quad 
  \delta([g,i x]) = \|\ad x\| 
\end{equation}
satisfy 
\begin{equation}
  \label{eq:ellesti3}
 \ell(z) \leq q_N(v) \Big(\frac{\pi}{2} - \delta(z)\Big)^{-N} \quad \mbox{
   for } \quad z \in \Xi. 
\end{equation}
\end{lem}

\begin{prf} Clearly, both functions $\ell$ and $\delta$
  on $\Xi$ are $G$-invariant, 
  hence defines an $\Ad(K)$-invariant function on $i\Omega_\fp$
  and thus are determined by their values on elements of the form
  $[e,ith]$, $0 \leq t < \frac{\pi}{2}$.
By \eqref{eq:qN}, we have on multiples of $h$ the estimate
\begin{equation}
  \label{eq:ellest1}
  \ell([e,ith])
  \leq q_N(v) \Big(\frac{\pi}{2} - t\Big)^{-N}
=  q_N(v) \Big(\frac{\pi}{2} - \delta([e,it h]\Big)^{-N}
  \quad
  \mbox{ for } \quad 0 \leq t < \frac{\pi}{2},
\end{equation}
so that the lemma follows by $G$-invariance.
\end{prf}

For $\fg= \so_{1,n}(\R)$ we consider the Euler elements defined by
\[ h(x_0,\ldots, x_n) = (x_1,x_0, 0,\ldots, 0) \quad \mbox{ and } \quad
  h_n(x_0,\ldots, x_n) = (x_n,0,\ldots, 0, x_0),\]
so that the corresponding involution acts on $\C^{1+n}$ by 
\[ \tau_h(z_0, z_1, \ldots, z_n) = (-z_0, -z_1,z_2, \ldots, z_n).\]
Its action on $\g$ satisfies $\tau_h(h) = h$ and $\tau_h(h_n) = - h_n$,
so that $h \in \fq_\fp$ and $h_n \in \fh_\fp$. 

On $\Xi = G \times_K i\Omega_{\fp}$, the antiholomorphic
involution $\oline\tau_h([g,x]) = [\tau_h(g), - \tau_h(x)]$
has the fixed point set
\[ \Xi^{\oline\tau_h}
  = G^h_e.\Exp(i\Omega_\fp^{-\tau_h}) 
  = G^h_e.\Exp(i\Omega_{\fh_\fp})\] 
(see 
\cite[\S 8]{MNO23b} and the context of \cite[Thm.~6.1]{NO23}). 
We have
\[ 
  \Omega_{\fh_\fp}
  = e^{\ad \fh_\fk}  \Big\{ sh_n \colon |s| < \frac{\pi}{2}\Big\}.\] 
Assume that $|s| < \frac{\pi}{2}$ and consider the curve
\[ \gamma \colon \Big(-\frac{\pi}{2}, \frac{\pi}{2}\Big) \to \Xi, \quad
  \gamma(t) := \alpha_{it}\Exp(ish_n) \in \Xi.\]

To relate the functions $\ell$ and $\delta$ on the curve $\gamma$, 
we use a concrete model for $G/K \cong \bH^n$,
$G/H \cong \dS^n$ and $G_\C/K_\C \cong \bS^n_\C \subeq \C^{n+1}$.
For $G := \SO_{1,n}(\R)_e$, we have
\[ G/K \cong \bH^n := G.i\be_0 \subeq i \R^{1+n}\]
and
\[   G/H \cong \dS^n := G.\be_n
  = \{(x_0, \bx) \colon x_0^2 - \bx^2 = - 1\} \subeq \R^{1,n}. \]
We then have
\begin{equation}
  \label{eq:xi-rankone}
  \Xi = \bS^n_\C \cap (\R^{1,n} + i V_+) \quad \mbox{ for } \quad
  V_+ = \{ (x_0,\bx) \colon x_0 > 0, x_0^2 > \bx^2\}.
\end{equation}
Note that 
\[ \oline\tau_h(z_0, z_1, \ldots, z_n)
  = (-\oline z_0, -\oline z_1,\oline z_2, \ldots, \oline z_n)\]
is an antiholomorphic involution fixing $i\be_0 \in \bH^n$
and $\be_n\in \dS^n$, but it maps $\be_1$ to $- \be_1$.
We have
\[ (\C^{1+n})^{\oline\tau_h} = i \R \be_0 + i \R \be_1
  + \R \be_2 + \cdots + \R \be_n.\] 
Now 
\[ \alpha_{it}(z_0, \ldots, z_n)
  = (\cos t \cdot z_0 + i \sin t \cdot z_1,
  i \sin t \cdot z_0 + \cos t \cdot z_1, z_2, \ldots, z_n).\]
For $x \in \R^{1+n}$, we therefore get 
\[ \Im( (\alpha_{it} x)_0 ) = \sin t \cdot x_1 > 0
  \quad \mbox{ for } \quad 0 < t < \pi, \quad x \in W
  = \{ x \in \R^{1,n} \colon x_1 > |x_0|\},\]
so that $\alpha_{it} x \in \Xi$. We also see that
\[ \alpha_{\frac{\pi i}{2}}x = (i x_1,  i x_0, x_2, \ldots, x_n) 
  \in \Xi^{\oline\tau_h},\]
where 
\begin{align*}
 \Xi^{\oline\tau_h}
  &  = \{ (i x_0, i x_1, x_2, \ldots, x_n) \colon
    x_0 > |x_1|, - x_0^2 + x_1^2 - x_2^2 - \ldots - x_n^2= -1\}\\
  &  = \{ (i x_0, i x_1, x_2,\ldots, x_n) \colon
    x_0 > |x_1|,  \underbrace{x_0^2 -x_1^2}_{>0}  + x_2^2 + \cdots
    + x_n^2= 1\}.
\end{align*}
It follows in particular that
\[ x_0^2 - x_1^2 \in (0,1].\]

  \begin{lem} The $G$-invariant function $\delta$ on $\Xi =
\bS^n_\C \cap (\R^{1,n} + i V_+)$ is given by 
   \begin{equation}
   \label{eq:ell4}\delta(z)
   = \arccos(\sqrt{\beta(\Im z)}) 
   = \arccos(\sqrt{1  +  \beta(\Re z)}), 
   \end{equation}
where $\beta(x) := x_0^2 - x_1^2 - \cdots - x_n^2$. 
\end{lem}

  \begin{prf} Since $G$-orbits in $V_+$
    (cf.~\eqref{eq:xi-rankone}) intersect $\R_+ \be_0$,
it suffices to evaluate on elements $z \in \Xi$ with
$\Im z \in \R_+ \be_0$. With the stabilizer $\SO_n(\R)$ of $\be_0$  
we move $\Re z$ into $\R \be_0 + [0,\infty] \be_n$, so that 
\[ z = (z_0, 0, \cdots, 0, z_n)
  \quad \mbox{ with} \quad \Im z_0 \in i \R_+ \be_0,\quad 
 z_n \geq 0, \quad -z_0^2 + z_n^2 = 1.\]
It follows that $z_0^2 = z_n^2 - 1 \in \R$, so that $z_0 \in i \R$.
Now $z = (a i, 0, \ldots, 0,b)$ with $a,b \in \R$ and
$a^2 + b^2 = 1$.
For the Euler element $h_n$ we have
\begin{equation}
  \label{eq:ell3}
  \exp(s i h_n).i\be_0
  = \pmat{ \cos s & 0 & i \sin s \\
    0 & \1 &0 \\
     i \sin s & 0 & \cos s} i \be_0
  = i \cos s \cdot \be_0 - \sin s\cdot \be_n. 
\end{equation}
For $a = \cos s$ and $b = -\sin s$, we thus obtain for
$z = \exp(is h_n).i \be_0$, which corresponds to
  $[e, ish_n] \in\Xi$, the relation 
\begin{equation}
  \label{eq:ell2}
 \delta(z) = |s| = \arccos(\Im z_0)
 = \arccos\Big(\sqrt{1 - z_n^2}\Big).
 \end{equation}
This leads to the identity \eqref{eq:ell4}.
\end{prf}

 Starting with a $\oline\tau_h$-fixed element $z = (ix_0, ix_1, x_2, \ldots, x_n)$
 in $\Xi$, this leads to
 \[ \alpha_{it}(ix_0, ix_1, x_2,\ldots, x_n)
  = (\cos t \cdot ix_0 -  \sin t \cdot x_1,
  - \sin t \cdot x_0 + \cos t \cdot i x_1, x_2, \ldots, x_n)\]
 with imaginary part 
\[ (x_0 \cos t, x_1 \cos t, 0,\ldots, 0).\] 
We thus obtain for $0 \leq t < \frac{\pi}{2}$: 
\begin{equation}
  \label{eq:deltaonbrokenline}
 \delta(\alpha_{it}z)
  = \arccos\Big(\cos t \cdot \sqrt{x_0^2 - x_1^2}\Big) 
  \in \Big[ \arccos\Big(\sqrt{x_0^2 - x_1^2}\Big),\frac{\pi}{2}\Big).
\end{equation}

We need the asymptotics of this function
for $t \to \frac{\pi}{2}$.
So let
\[ \lambda := \sqrt{x_0^2 - x_1^2} \in (0,1]\]  
and consider the function
\[ g(t) :=\delta(\alpha_{it}z) = 
  \arccos(\lambda \cos t) \quad \mbox{ for } \quad
  0 \leq t < \frac{\pi}{2}.\]
Then $g$ extends smoothly to $(0,\pi)$ with
\[ g'\Big(\frac{\pi}{2}\Big)
  = \arccos'(0) \lambda \cos'\Big(\frac{\pi}{2}\Big)
  = \frac{1}{\cos'(\frac{\pi}{2})}\lambda \cos'\Big(\frac{\pi}{2}\Big)
  = \lambda.\]
We conclude that, for $t \to \frac{\pi}{2}$, we have
\begin{equation}
  \label{eq:knick-esti}
 \lambda \Big|\frac{\pi}{2}-t\Big| \sim \Big|\frac{\pi}{2}-g(t)\Big|
 = \frac{\pi}{2} - \delta(\alpha_{it}z). 
\end{equation}
For $z = \exp(sih_n).i\be_0$ as above, \eqref{eq:ell3} shows that 
$\lambda =  \cos(s)$. 

\begin{lem} \mlabel{lem:5.9}
  $q_N(e^{i \partial U(x)}v) < \infty$ for $x \in \Omega_{\fh_\fp}$.
\end{lem}

\begin{prf} Since $q_N$ is invariant under $U(H_K)$ and
  $\fq_\fp = e^{\ad \fh_\fk}(\R h_n)$, it suffices to show that 
  \[ q_N(e^{is \partial U(h_n)}v) < \infty\quad \mbox{ for } \quad
    |s| < \pi/2,\] 
  which means that 
\begin{equation}
  \label{eq:rank-one-est} t \mapsto
  \Big(\frac{\pi}{2} - t\Big)^N
  \|e^{it \partial U(h)} e^{is \partial U(h_n)} v \|
=   \Big(\frac{\pi}{2} - t\Big)^N
\ell(\alpha_{it} \Exp(ish_n))
\end{equation}
is bounded on the interval $|t| < \frac{\pi}{2}$.
In view of \eqref{eq:ellesti3}, it suffices to show for
$z_s := \Exp(ish_n)$ that
\[  \Big(\frac{\pi}{2} - t\Big)^N
  \Big(\frac{\pi}{2} - \delta(\alpha_{it} z_s)\Big)^{-N} \]
is bounded, which follows immediately from \eqref{eq:knick-esti}.
\end{prf}

In view of Theorem~\ref{thm:E.4},
Lemma~\ref{lem:5.9} implies for $\sE_K = \R v$ that
\[ \beta^+(e^{i \partial U(\Omega_{\fh_\fp})}\sE_K) 
  \subeq \cH^{-\infty}(U_h)_{\rm KMS}. \]
In particular,  Hypothesis~(H2) is satisfied.
Next we observe that $U^{-\omega}(G^h_e)$ commutes with $J$ and $U_h$, 
  hence leaves $\cH^{-\infty}(U_h)_{\rm KMS}$ invariant. 
We thus obtain with
$W^G = G^h_e \exp(\Omega_{\fq_\fk})H$
(see \eqref{eq:fiber-diffeo} in Appendix~\ref{app:b}) and
Lemma~\ref{lem:2.16}(c): 
\begin{align*}
 U^{-\infty}(W^G)\sE_H
&  = U^{-\infty}(G^h_e) e^{\partial U(\Omega_{\fq_\fk})}\sE_H
  = U^{-\infty}(G^h_e) \beta^+(e^{\partial U(i \Omega_{\fh_\fp})}\sE_K) \\
&\subeq   U^{-\infty}(G^h_e) \cH^{-\infty}(U_h)_{\rm KMS} 
=   \cH^{-\infty}(U_h)_{\rm KMS} \subeq \cH^{-\infty}_{\rm KMS}.
\end{align*}

\begin{thm} \mlabel{thm:5.9}
  Let $G$ be a connected Lie group with Lie algebra
  $\g = \fsl_2(\R)$ and 
  $(U,\cH)$ be an irreducible unitary 
  representation of~$G$.
  Then $U$  extends to an antiunitary
  representation of 
  $G_{\tau_h} = G \rtimes \{\1,\tau_h\}$.
  Let $v \in \cH$ be a $J$-fixed $K$-eigenvector, 
  $\cE := \C v$ and $\sE_K := \R v$. 
  Then {\rm Hypothesis~(H2)} is satisfied, so that
  $\sE_H = \beta^+(\sE_K) \subeq \cH^{-\infty}$. 
  Further, the net $\sH^M_{\sE_H}$ on 
  the non-compactly causal
  symmetric space $M = G/H$ satisfies {\rm (Iso), (Cov), (RS)} and
  {\rm (BW)}, where $W = W_M^+(h)_{eH}$ is the connected component
  of the positivity domain of $h$ on $M$, containing the base point.
\end{thm}

\begin{prf} First we derive the existence of the antiunitary extension
  of $U$ from \cite[Thm.~4.24]{MN21}.
    As all $K$-types in $\cH$ are $1$-dimensional, they consist 
    of eigenvectors. Let $\cH_\chi = \C v_\chi$ be the eigenspace
    corresponding to the character $\chi \in \hat K$. 
    As $\tau_h(k) = k^{-1}$ for $k \in K$, all $K$-eigenspaces
    are $J$-invariant, hence they all contain $J$-fixed vectors. 
    The preceding discussion shows that Hypothesis~(H2) 
    holds for $\cE = \C v$, so that the theorem follows from
    Theorem~\ref{thm:4.9}.   
  \end{prf}

\begin{rem} (The possibilities for $H$)
 For $m \in \N \cup \{\infty\}$, 
    let $G_m$ be a connected Lie group with Lie algebra $\g = \fsl_2(\R)$
    and $|Z(G_m)| = m$. For $m \in \N$ this means that
    $Z(G_m) \cong \Z/m\Z$ and
    $G_m$ is an $m$-fold covering of $\Ad(G_m) \cong \PSL_2(\R) \cong G_1$.
    Note that $G_2 \cong \SL_2(\R)$. 
    Further $G_\infty \cong\tilde\SL_2(\R)$ is simply connected
    with $Z(G_\infty) \cong \Z$. 

    We consider the Cartan involution $\theta(x) = - x^\top$, 
    the Euler element
    \[ h = \frac{1}{2}{\pmat{1 & 0\\ 0 & -1}} \quad \mbox{   and } \quad
      z = \pmat{ 0 & -\pi \\ \pi & 0}\in \fk = \so_2(\R), \]
    which satisfies
    $e^z = -\1$. Then
  \[ K = \exp(\R z), \quad Z(G_m) = \exp(\Z z), \quad \mbox{ and }  \quad 
    \tau_h(\exp tz) = \tau(\exp tz) = \exp(-tz)\] 
  because $\tau = \theta \tau_h$.   We conclude that
  \[ K^\tau = \{e\} \quad \mbox{  if } \quad m = \infty
    \quad \mbox{ and } \quad K^\tau = \{e, \exp(mx/2)\}\
    \ \mbox{ otherwise}.\]
For $m = \infty$, $H = G_m^\tau$ is connected. For $m \in \N$, the group
$G_m^\tau= K^\tau \exp(\fh)$ has two connected components,
but if $m$ is odd, then $K^\tau$ does not fix the Euler element~$h\in C^\circ$.
Therefore only $H := \exp(\fh)$ leads to a causal symmetric space $G_m/H$.
If $m$ is even, then $H$ can be either $(G_m)^\tau_e$ or $G_m^\tau$.

In $G_1 \cong \PSL_2(\R)$, the subgroup $H$ corresponds to $\SO_{1,1}(\R)_e$
and the noncompactly causal symmetric space 
$G_1/H \cong \dS^2$ is the $2$-dimensional de Sitter space.

The universal covering $\tilde\dS^2$ is obtained  for $m = \infty$,
$G_\infty = \tilde\SL_2(\R)$ and then $H = \exp(\fh)$ is connected.
All other coverings of $\dS^2$ are obtained as $G_m/H$ for $H = \exp(\fh)$. 
\end{rem}

\begin{rem}
  If $\g = \fsl_2(\R)$, then
  $G_{\rm ad} = \PSL_2(\R) \cong \SO_{1,2}(\R)_e$ (the M\"obius group), 
    and $H_{\rm ad} = \exp(\R h)$, then $G_{\rm ad}/H_{\rm ad} \cong \dS^2$ 
    and it follows easily that $W_{M_{\rm ad}}^+(h)$ is connected;
    see \cite[Thm.~7.1]{MNO23b} for a generalization of this observation
    to non-compactly causal spaces.
    Therefore $W = W_{M_{\rm ad}}^+(h)$ in the preceding theorem.
    If $Z(G)$ is non-trivial, 
    then the connected components of $W_M^+(h)$ can be labeled
    by the elements of $Z(G)$ because this subgroup
    acts non-trivially on $M = G/H$, 
    leaving the positivity region $W_M^+(h)$ invariant.
    In any irreducible representation $(U,\cH)$ we have $U(Z(G)) \subeq \T$,
    but this subgroup preserves the standard subspace
    $\sV$ if and only if it is contained
    in $\{\pm 1\}$.     
\end{rem}

\subsection{An application to intersections of standard subspaces}

In $\fsl_2(\R)$, we consider the two Euler elements 
\begin{equation}\label{eq:handk}
h := \frac{1}{2} \pmat{ 1 & 0 \\ 0 & -1} \quad \mbox{ and }\quad 
 h_1 := \frac{1}{2} \pmat{ 0 & 1 \\ 1 & 0} \in \fh = \so_{1,1}(\R).
 \end{equation} 
Let $(U,\cH)$ be a unitary representation 
of group $G := \tilde\SL_2(\R)$.
By Theorem~\ref{thm:5.9}, $U$ extends to an (anti-)unitary 
representation of $\tilde\GL_2(\R) := \tilde\SL_2(\R) \rtimes \{ \1,\tau_h\}$,
and we put 
\[ J := U(\tau_h), \quad \Delta  := e^{2\pi i \cdot \partial U(h)}
\quad \mbox{ and }\quad \sV := \Fix(J \Delta^{1/2}).\]

The following proposition solves the
$\SL_2$-problem, as formulated in \cite[\S 4]{MN21}
in the affirmative.
From \cite[Thm.~1.1, Cor.~1.3(c)]{GL95} one can deduce this result
for  principal series representations and 
lowest and highest weight representations, 
but it was not known to  hold for the complementary series. 

\begin{thm} Let $(U,\cH)$ be an irreducible unitary
  representation of $\tilde\SL_2(\R)$.
  For every $t \in \R$, the intersection
  $ U(\exp t h_1) \sV \cap \sV$ is a standard subspace.
\end{thm}

\begin{prf} We pick a $J$-fixed $K$-eigenvector $v \in \cH$
  and define $\sE_H := \beta^+(v)$ as above.
  By Theorem~\ref{thm:5.9} we then obtain a net
  $\sH^M_{\sE_H}(\cO)$ of real subspace indexed by open subsets
  $\cO \subeq M = G/H$. Moreover,
  \[ \sV = \sH^M_{\sE_H}(W) \]
  holds for the basic connected component $W = W_M^+(h)_{eH}$
  of the positivity domain of the modular vector field on~$M$.
  For each $t \in \R$, the intersection
  $W_t := W \cap \exp(th_1)W$ is non-empty because
  $\exp(th_1) \in H$ fixes the base point in $G/H$.
  Therefore
  \begin{align*}
 U(\exp th_1)\sV \cap \sV
& = U(\exp th_1)\sH^M_{\sE_H}(W) \cap \sH^M_{\sE_H}(W)\\
& = \sH^M_{\sE_H}(\exp(th_1) W) \cap \sH^M_{\sE_H}(W)
    \supeq \sH^M_{\sE_H}(W_t).
  \end{align*}
As $\sH^M_{\sE_H}(W_t)$ is cyclic by the Reeh--Schlieder property
(Theorem~\ref{thm:5.9}),
it follows that 
$U(\exp th_1)\sV \cap \sV$ is cyclic, hence standard because it is
contained  in the standard subspace~$\sV$.   
\end{prf}

\subsection{Positive energy representations of $\PSL_2(\R)$}
  \label{subsec:psl2}

  The non-trivial positive energy representations $(U_s,\cH_s),
  s = 2,4,6,\ldots$  of the M\"obius  group 
  $G = \PSL_2(\R)$ can be realized on the reproducing kernel
  Hilbert spaces $\cH_s \subeq \cO(\C_+)$ on the upper half plane
  with the reproducing kernel
\begin{equation}
  \label{eq:ks}
 Q(z,w) 
= \Big(\frac{z - \oline w}{2i}\Big)^{-s}, 
\end{equation}
for which positive definiteness follows from the fact that
the function $\big(\frac{z}{i}\big)^{-s}$ is the Laplace transform
of the Riesz measure
\[ d\mu_s(p) = \Gamma(s)^{-1} p^{s -1}\, dp 
\quad \mbox{ on } \quad (0,\infty), \]
see \cite[\S 6]{NOO21}.
Point evaluations on $\cH_s$ take the form
\[ f(w) = \la Q_w, f \ra \quad \mbox{ with } \quad Q_w(z) = Q(z,w)
  = \Big(\frac{z - \oline w}{2i}\Big)^{-s} \quad \mbox{ for }
  \quad z,w \in \C_+.\]
Note that the Fourier transform defines a unitary isomorphism
\[ \cF \colon L^2(\R_+, \mu_s) \to \cH_{s}\subeq \cO(\C_+), \quad
  \cF(f)(z) := 2^s \int_0^\infty e^{izp} f(p)\, d\mu_s(p)\]
(\cite[Lemma~3.10]{NOO21})
which extends to an isomorphism 
\[ \cF \colon L^2(\R_+, \mu_s)^{-\infty}  \to \cH_{s}^{-\infty} \subeq \cO(\C_+) \]
and this exhibits the image of $\cH_s^{-\infty}$ under the boundary values on
$\R$ as a space of tempered distributions on $\R$
(cf.\ \cite[Rem.~3.7]{NOO21}).
The representation $U_s$ (lifted to $\SL_2(\R)$) is given by 
\[ (U_s(g)f)(z)
  := (a - cz)^{-s} f(g^{-1}.z)
  = (a - cz)^{-s} f\Big(\frac{dz-b}{a-cz}\Big) \]
for
\[  g = \pmat{a & b \\ c & d}, \quad 
  g.z = \frac{a z + b}{cz+d}.\]
Identifying $g \in G$ with the corresponding M\"obius transformation,
the derivative of $g$ is $g'(z) = (cz + d)^{-2}$, so that we obtain
\[ (U_s(g^{-1})f)(z)  := g'(z)^{s/2} f(g.z) \quad \mbox{ for } \quad
  g \in G,z  \in \C_+\]
and in particular 
\begin{equation}
  \label{eq:Qform1}
 U_s(g)Q_z = \oline{g'(z)^{s/2}} Q_{g.z} \quad \mbox{ for } \quad
 g \in G,z  \in \C_+.
\end{equation}
Since the $(Q_z)_{z \in \C_+}$ are analytic vectors, we obtain an embedding
\[ \iota \colon \cH_s^{-\omega} \into \cO(\C_+), \quad
  \iota(\alpha)(z) := \alpha(Q_z), \]
extending the inclusion $\cH_s \into \cO(\C_+)$ in a
$G$-equivariant way. As we shall see below, the holomorphic functions
$(Q_x)_{x \in \R}$ defined by $Q_x(z) = Q(z,x) = \big(\frac{z - x}{2i}\big)^{-s}$
are distribution vectors, and the action of $G$ on these distribution vectors
satisfies
\begin{equation}
  \label{eq:Qform2}
U_s^{-\omega}(g)Q_x = g'(x)^{s/2} Q_{g.x} \quad \mbox{ for } \quad
g \in G, x, g.x \in \R.
\end{equation}
It follows in particular that $Q_x$ is an eigenvector for the
action of the parabolic subgroup
\[ G^x := \{ g \in G \colon g.x =x \}.\]

We consider the Euler elements $h$ and $h_1$ from \eqref{eq:handk}
with $H = \exp(\R h_1)$, and extend $U_s$ to an antiunitary
representation of $G_{\tau_h}$ by setting
$U_s(\tau_h) := J$, where $J$ is the conjugation 
\begin{equation}
  \label{eq:jdef}
  (JF)(z) := e^{\pi i s/2} \oline{F(-\oline z)}= (-1)^{s/2}  \oline{F(-\oline z)},
\end{equation}
which satisfies $J U_s(g) J = U_s(\tau_h(g))$ for $g \in G$. 

The subgroup $K = \exp(\so_2(\R)) = \PSO_2(\R)$ fixes the point $i \in \C_+$ and
by \eqref{eq:Qform1}
\[ Q_i(z) = \Big(\frac{z  + i}{2i}\Big)^{-s} \]
is a $K$-eigenfunction, so that we may put $\cE := \C Q_i$. 
As $J Q_i = e^{\pi i s/2} Q_i$, we obtain the fixed point space 
\[ \sE_K = \cE^J = \R e^{\pi i s/4} Q_i.\]

To determine $\sE_H \subeq \cH^{-\infty}$, we have to
determine the limit of $e^{it \partial U_s(h)} Q_i$ for $t \to \pm\pi/2$.
From $\exp(th).z = e^t h$, we get for $t \in \R$ with \eqref{eq:Qform1} that 
\[ U_s(\exp th) Q_i = e^{st/2} Q_{e^t i}.\]
Analytic continuation leads for $|t| < \pi/2$ to 
\[ e^{it \partial U(h)} Q_i = e^{i st/2} Q_{e^{it} i} ,\]
and for $t \to -\pi/2$ we get the limit $e^{-s \pi i/4} Q_1 = i^{-s/2} Q_1$.
In particular, Corollary~\ref{cor:5.1} implies that $Q_1
\in \cH_s^{-\infty}$ and we find
\[ \sE_H = \R Q_1.\] 
Since the orbit of $Q_1$ under the group
of translations is $(Q_x)_{x \in\R}$, all these functions
are distribution vectors.

We also note that, for $x > 0$, and $\Delta = e^{2\pi i \partial U(h)}$,
we obtain in $\cH_s^{-\infty}$ the relation 
\[ \Delta^{1/2} Q_x = (-1)^{s/2} Q_{-x} = J Q_x,\]
so that
\begin{equation}
  \label{eq:svpos}
Q_x \in \cH^{-\infty}_{\rm KMS} \quad \mbox{ for  } \quad x > 0.  
\end{equation}
As the family $(Q_x)_{x > 0}$ is the orbit of $Q_1$ under the modular group
$\exp(\R h)$, acting with test functions on $\R_+$ generates $\sV$
  (\cite[Prop.~3.10]{Lo08}).
Likewise $(Q_x)_{x < 0}$ generates $\sV' = J\sV$
(here we use that $s \in 2 \N$). It follows that
\begin{align*}
 \sV^\infty = \sV \cap \cH_s^\infty
&  = \{ f \in \cH_s^\infty \colon \la f, \sV' \cap \cH^\infty_s \ra \subeq \R \}\\
&  =  \{ f\in \cH_s^\infty \colon
                                                                                     (\forall x < 0)\ \la Q_x,f \ra = f(x) \in \R \}.
\end{align*}
We conclude that
\[ \sV = \{ f \in \cH_s \colon f(\R_-) \subeq \R \} \]
in the sense of boundary values on $\R$
(cf.\ (3.42) in\cite{NOO21}). 

As $Q_1$ is an eigenvector of the parabolic subgroup
$P_+ := G^1 \supeq H$ (the stabilizer of $1 \in \R$),
  we obtain
  $\sH^G_{\sE_H}(\cO) = \sH^G_{\sE_H}(\cO P_+)$
  for every open subset $\cO \subeq G$
  with \cite[Lemma~2.11]{NO21}, so that
  the net $\sH^G_{\sE_H}$ on $G$ actually descends to a
  net on $G/P_+ \cong \bS^1$.

Likewise negative energy representations 
lead to nets on $\bS^1 \cong G/P_-$ for $P_- := G^{-1}$
(the stabizer of $-1 \in \R$). As 
  $P_+ \cap P_- = H$, we obtain an equivariant embedding
  \[ \dS^2 \cong G/H \into G/P_+ \times G/P_-  \cong \T^2.\] 
  The above construction thus leads to nets on
  $\dS^2$ that can be reconstructed from their
  pushforwards under the projections
  $p_\pm \colon \dS^2 \to \bS^1$.

  \begin{rem} \mlabel{rem:herm}
    We find a similar situation whenever
    $\g$ is a simple hermitian Lie algebra, i.e.,
    unitary highest weight representations
    (positive energy representations in the physics terminology)
    exist. Then the existence of an Euler element in $\g$ implies
    that it is of tube type (\cite[Prop.~3.11(b)]{MN21})
    and $G/H$ is a causal symmetric space
    of Cayley type that embeds into a product
    \[ p = (p_+, p_-) \colon G/H \to G/P_+ \times G/P_-,\]
    where $G/P_\pm$ are causal homogeneous spaces
    (compactifications of euclidean Jordan algebras),  
    for which $G$ is the conformal group (or some covering).
    As $G/P_+ \cong G/P_-$ carries two $G$-invariant causal
    structures, the above embedding leads to
    causal embeddings of $G/H$, endowed with a
    compactly causal structure, and also with a non-compactly
    causal structure (cf.~\cite{HO97}). 
    For $\g = \fsl_2(\R)$, this corresponds to the embedding
    of $2$-dimensional de Sitter space $\dS^2$ (which is non-compactly
    causal) and also of the dual anti-de Sitter space
    $\AdS^2$ into the product $G/P_+ \times G/P_-$.
    In \cite[\S 5.2]{NO21} covariant nets on $G/P_\pm$
    corresponding to positive energy representations
    are constructed along the lines indicated above for
    $\PSL_2(\R)$.

    For the conformal group $G = \SO_{2,d}(\R)_e$ of compactified
    Minkowski space, this leads to nets of standard subspaces 
    which yield by second quantization conformally covariant
    nets of operator algebras, as they are discussed in \cite{GL03}.
    It is an interesting question when, and to which extent,
    a net $\sH^M_{\sE_H}$ on $M$ can be reconstructed
    from its pushforward nets $\sH^{G/P_\pm}_{\sE_H}$.
    On $G/P_\pm$, the results of \cite{GL03} suggest ``nice situations'' for
    positive/negative energy representations of $G$ and these properties
    can probably be characterized in terms of isotony requirements
    on the nets. In particular, we expect \cite{MNO23b} to provide
    the necessary geometric background to analyse nets in terms of
    KMS properties of geodesic observers in wedge regions. 
\end{rem}

\section{Outlook: Beyond linear simple Lie groups}

\mlabel{sec:6}

Our main result Theorem~\ref{thm:4.9} is formulated under
the assumption of Hypothesis~(H2), which
can be derived for linear simple Lie groups
from the Kr\"otz--Stanton Extension Theorem
(Theorem~\ref{thm:ks04}) and the
Automatic Continuity Theorem
(Theorem~\ref{thm:AutomaticContinuity}).
Both are limited to linear groups. In this
section we outline one possible strategy
to extend these results to the non-linear case,
hence in particular to connected simple Lie groups with
infinite center.
We shall see in Section~\ref{subsec:6.1} that
the Casselman-Subrepresentation Theorem extends
to Harish--Chandra modules for which $G$ does not have to be linear,
but the Casselman--Wallach Globalization Theorem~\cite{BK14},
respectively its presently available proofs, does
not extend in any obvious fashion; see Subsection~\ref{subsec:6.3}
for a discussion.

Let $(U, \cH)$ be a unitary representation
of a connected semisimple Lie group~$G$. Note that the
center $Z(G)$ may be infinite and $G$ may not be linear.
We fix a Cartan
decomposition $\g = \fk \oplus \fp$ and the corresponding
integral subgroup $K := \exp \fk$. Then $\Ad(K)$ is compact,
but $K$ need not be compact. However, we have a polar diffeomorphism
$K \times \fp \to G, (k,x) \mapsto k \exp x$. In particular
$\pi_1(K) \cong \pi_1(G)$.

\subsection{Casselman's Subrepresentation Theorem
for non-linear groups}
\mlabel{subsec:6.1}

The notion of an {\it (admissible) $(\g,K)$-module} generalizes
to our context in the obvious fashion.
As all irreducible unitary $K$-representations are finite-dimensional,
this causes no problem if admissibility is defined as the finiteness
of all $K$-multiplicities. 

In \cite[Thm.~5]{HC53} Harish-Chandra shows that,
if $U$ is {\it quasi-simple}, i.e., $U(Z(G)) \subeq \T$ and
it has an infinitesimal character in the sense that 
\[ \dd U(Z(\cU(\g))) \subeq \C \1,\] then all 
$K$-finite vectors are analytic and 
the multiplicity of each $K$-type is finite
(\cite[Lemma~33]{HC53}).
This assumption is in particular
satisfied for any irreducible unitary representation,
so that the space $\cH^{[K]}$ of $K$-finite vectors is an
admissible $(\g,K)$-module.


Let $\fp_{\rm min} = \fm + \fa + \fn \subeq \g$ be a minimal
parabolic, where $\fa \subeq \fp$ is maximal abelian, 
$\fm = \fz_\fk(\fa)$, and $\fn$ is the sum of all positive root spaces
for a positive system $\Sigma^+(\g,\fa)$ of restricted roots. 
Osborne's Lemma (\cite[Prop.~3.7.1]{Wa88}) 
asserts the existence of a finite-dimensional
subspace $F \subeq \cU(\g_\C)$ such that
\[ \cU(\g_\C) = \cU(\fn_\C) F Z(\cU(\g_\C)) \cU(\fk_\C).\]
This implies that every finitely generated admissible
$\g$-module $V$ with an infinitesimal character 
is finitely generated as an $\cU(\fn_\C)$-module,
hence in particular that
\[ \dim(V/\fn V) < \infty.\] 

In \cite{BB83} 
Beilinson and Bernstein show that 
$V\not = \fn V$ for non-zero admissible $(\g,K)$-modules,
where $K$ does not have to be compact.
Their argument makes heavy use of $\cD$-modules.
Another proof of this fact is given in \cite[Thm.~3.8.3]{Wa88}.
Wallach states it for real reductive groups, i.e., finite coverings
of linear groups, but all arguments in its proof
work for general connected reductive groups,
where Harish--Chandra modules are replaced by
admissible $(\g,K)$-modules, where $K$ is connected but
not necessarily compact. 


As $\fn \trile \fp_{\rm min}$  is an ideal,
the finite-dimensional space 
$V/\fn V$, associated to a finitely generated 
admissible $(\g,K)$-module, 
carries a natural structure of a $(\fp_{\rm min}, M)$-module. 
We thus obtain an
irreducible $M$-subrepresentation $(\sigma, W_\sigma)$ of $V/\fn V$ 
and a $\lambda \in \fa_\C^*$ such that $P_{\rm min} := MAN$
acts on $W_\sigma$ by
\[ \sigma_\lambda(man) := a^{\lambda + \rho} \sigma(m). \]
Write $W_{\sigma,\lambda}$ for the corresponding representation
of $P_{\rm min}$, 
so that we have a $(\fp_{\rm min},M)$-morphism
\[ \psi \colon V \to W_{\sigma,\lambda}\]
and thus, by Frobenius reciprocity, a $(\g,K)$-morphism into a
principal series 
\[ \Psi \colon V \to V_{\sigma,\lambda}^{[K]},\]
where
\begin{align*}
& V_{\sigma,\lambda}^\infty\\
 & := \{ f \in C^\infty(G,W_{\sigma,\lambda}) \colon
  (\forall g \in G, m \in M, a \in A, n \in N)\
   f(mang) = a^\lambda \sigma(m) f(g) \},
\end{align*}
endowed with the right translation action of $G$.
Note that evaluation in the identity is a $P_{\rm min}$-equivariant map
\[  \ev_e \colon  V^\infty_{\sigma,\lambda} \to W_{\sigma,\lambda}, \quad
  f \mapsto f(e).\]

Therefore the Casselman Subrepresentation Theorem
(cf.\ \cite[\S 3.8.3]{Wa88}) extends to connected
semisimple Lie groups~$G$:

\begin{thm} \mlabel{thm:6.1}
  If $G$ is a connected semisimple Lie group
  and $(U,\cH)$ an irreducible unitary representation,
  then the corresponding $(\g,K)$-module $\cH^{[K]}$ 
  embeds in a $(\g,K)$-equivariant way
  into a principal series representation $V_{\sigma,\lambda}^{[K]}$.
\end{thm}

In \cite{HC54a}, Harish--Chandra shows that every irreducible
admissible $\g$-module appears as a subquotient of a principal
series representation (cf.\ also \cite[Thm.~3.5.6]{Wa88}).
His arguments do not use linearity of the group.

\begin{prob}
Show that the $(\g,K)$-morphism
$ \cH^{[K]} \to V_{\sigma,\lambda}^\infty$ from Theorem~\ref{thm:6.1}
extends to a continuous 
linear map on~$\cH^\infty$.
This is contained in \cite[Cor.~5.2]{CO78} for the linear case
  and in \cite[Cor.~11.5.4]{Wa92} for connected real reductive groups,
  which include connected semisimple groups with finite center.  
  \end{prob}

\subsection{Growth control}
\mlabel{subsec:6.2b}

Let $G$ be a connected Lie group. A {\it scale on $G$}
is  submultiplicative function $s \colon  G \to \R_+$ 
for which $s$ and $s^{-1}$ are locally bounded.
The set of scales on $G$ carries an order defined by
\[ s \prec s' \quad \mbox{ if } \quad
  (\exists N \in \N)\ \ \sup_{g \in G} \frac{s(g)}{s'(g)^N} < \infty.\]
Two scales $s$ and $s'$ are said to be {\it equivalent} if
$s \prec s'$ and $s' \prec s$.
The equivalence class $[s]$ of $s$ is called a {\it scale structure}.
Any left invariant Riemannian metric $d$ on $G$
defines by $s_{\rm max}(g) := e^{d(g,e)}$ a maximal scale structure
on $G$ (cf.\ \cite[\S 2.1]{BK14}). Simple scales are functions of
the form $s(k \exp x) = e^{\|x\|}$ for $k \in K$, $x \in \fp$, where
$\|\cdot\|$ is an $\Ad(K)$-invariant norm on~$\fp$.

\begin{ex} On $G = \R$ the maximal scale structure is defined by
  $s_1(x) := e^{|x|}$, but there are others, such as the scale structure
  defined by
  \[ s_2(x) := \Big\| \pmat{1 & x \\ 0 & 1}\Big\|
    \leq 1 + |x| \Big\|\pmat{0 & 1 \\ 0 & 0}\Big\|,\]
  which grows only linearly.   
\end{ex}

Any continuous
representation $\pi$ of $G$ on a seminormed space $(E,p)$ specifies
a scale by $s_\pi(g) := \|\pi(g)\|$ and $(\pi,E)$ is called {\it $s$-bounded}
if $s_\pi \prec s$.

Let $(G,[s])$ be a Lie group with scale structure. 
An {\it $F$-representation} of $(G,[s])$ is a continuous representation
$(\pi,E)$ of $G$ on a Fr\'echet space $E$ whose topology is defined
by a countable family of $s$-bounded seminorms $(p_n)_{n \in \N}$ which are
{\it $G$-continuous}
in the sense that the representation of $G$ on
the seminormed space $(E,p_n)$
is continuous for every $n \in \N$.
Not every continuous representation on a Fr\'echet spaces
is an $F$-representation (cf.\ \cite[Rem.~2.8]{BK14}).

Below we assume that $G$ is {\bf connected semisimple}
and that $[s]$ is maximal. 

An $F$-representation $(\pi, E)$ is called {\it smooth} if $E = E^\infty$
as topological vector space. Smooth $F$-representations are
called {\it SF-representations}. 

An element $f \in L^1(G)$ is called {\it rapidly decreasing}
if all functions $g \mapsto s(g)^n f(g)$, are $L^1$.
The subspace $\cR(G) \subeq L^1(G)$ of rapidly decreasing functions
then carries an $F$-representation
\begin{equation}
  \label{eq:rap-dec}
  ((L \times R)(g_1, g_2)f)(g) := f(g_1^{-1} g g_2)
\end{equation}
of $G \times G$ and it is a Fr\'echet algebra under convolution.
Further, any $F$-representation $(\pi,E)$ of $(G,[s])$ integrates to a continuous
algebra representation
\[ \cR(G) \times E \to E, \quad
  (f,v) \mapsto \pi(g)v \quad \mbox{ for } \quad
  \pi(f) = \int_G f(g)\pi(g)\, dg.\]

For $u \in \cU(\g)$, we write
$L_u$ and $R_u$ for the corresponding operators on 
$C^\infty(G)$ defined by the derived representation of $G\times G$
and define the {\it Schwartz space of $G$} by 
\[ \cS(G) := \cR(G)^\infty \subeq C^\infty(G) \]
for the $G \times G$-representation \eqref{eq:rap-dec}
on $\cR(G)$. Note that
the smoothness of the elements of $\cR(G)^\infty \subeq L^1(G)^\infty$
follows from the local Sobolev Lemma.

On any $SF$-representation of $(G,[s])$
we then have representations of the three convolution algebras
\[ C_c^\infty(G) \subeq \cS(G) \subeq \cR(G)\]
and the Dixmier--Malliavin Theorem (\cite{DM78}) implies that
$\pi(C^\infty_c(G)) E = E^\infty$, which in turn implies 
\[ E^\infty = \pi(\cS(G))E. \]

\subsection{Localization in central unitary characters} 
\mlabel{subsec:6.3}

If $G$ is a connected simple Lie group
and $\eta_G \colon G \to G_\C$ its universal complexification,
then $\eta_G(G) \subeq G_\C$ is a linear simple group
and $\eta_G \colon  G \to \eta_G(G) \cong G/\ker(\eta_G)$
is a covering with discrete central kernel.
Therefore, extending results from linear simple groups to
general connected simple Lie groups requires a localization
in central characters. We outline the main ideas in this section.

Let $G$ be a locally compact group, 
$Z \subeq Z(G)$ a closed subgroup and
${\chi \colon  Z \to  \T}$ a unitary character. We write
$q_{G/Z} \colon G \to G/Z$ for the canonical projection and
\begin{align*}
& C_c(G)_\chi := \{ f \in C(G,\C) \colon \\
&\qquad   (\forall g \in G, z \in Z)\ f(gz) = \chi(z)^{-1} f(g),
  q_{G/Z}(\supp(f)) \ \mbox{ compact}\}.
\end{align*}
If $\bL_\chi := G \times_Z \C$ is the line bundle  over $G/Z$ obtained
by factorization of the $Z$-action $z.(g,w) = (gz, \chi(z)^{-1}w)$,
then
\[ C_c(G)_\chi \cong C_c(G/Z,\bL_\chi) \]
corresponds to the space of compactly supported sections of $\bL_\chi$.

On $C_c(G)_\chi$ we define a convolution product by
\[  (f_1 * f_2)(g) := \int_{G/Z} f_1(x) f_2(x^{-1}g)\, d x, \]
where we use that, as a function of $x$,
the integrand factors through a function
on $G/Z$. Therefore the integral exists by compactness of $q_{G/Z}(\supp(f_1))$.
The subspace $C_c(G)_\chi$ is also invariant under the involution defined by 
\[  f^*(g) := \oline{f(g^{-1})} \Delta_G(g)^{-1}.\]
We thus obtain a $*$-algebra $(C_c(G)_\chi,*)$.
Completing with respect to the norm
\[ \|f\|_1 := \int_{G/Z} |f(gZ)|\, d(gZ) \]
leads to the Banach $*$-algebra $L^1(G)_\chi$.

A unitary representation $(\pi,\cH)$ is called a {\it $\chi$-representation}
if
\[ \pi(z) = \chi(z) \1 \quad \mbox{ for all } \quad z \in Z.\]
Then
\[ \pi(f) := \int_{G/Z} f(g)\pi(g)\, d(gZ) \]
is well-defined because the integrand factors through a function on $G/Z$.
The arguments used for $Z = \{e\}$ carry over to this context and show that
we thus obtain a contractive $*$-representation
of the Banach $*$-algebra $L^1(G)_\chi$ on $\cH$.

\begin{rem} All  concepts from
  Section~\ref{subsec:6.2b} 
  carry over to the $\chi$-local situation
  if we work with scales on the quotient group $G/Z$.
  An important special case is $Z = \ker(\eta_G)$, where
  $\eta \colon G \to G_\C$ is the universal complexification of a
  connected semisimple Lie group.
  This leads in particular to the Fr\'echet convolution algebras 
\[ \cS(G)_\chi \subeq \cR(G)_\chi \subeq L^1(G)_\chi.\]   
\end{rem}

\begin{rem} Suppose that $(\pi, E)$ is a $\chi$-representation,
  $f \in C_c(G)$ and
  \[ f_\chi(g) := \int_Z f(gz) \chi(z)\, dz.\]
For $v \in \cH$, we then have
  \begin{align} \label{eq:fchi}
 \pi(f) v
&    = \int_G f(g) \pi(g)v\, dg 
    = \int_{G/Z} \int_Z  f(gz) \pi(gz)v\, dz\, d(gZ) \notag \\
&    = \int_{G/Z} \Big(\int_Z f(gz) \chi(z)\, dz \Big) \pi(g)v\, d(gZ) \\
&    = \int_{G/Z} f_\chi(g) \pi(g)v\, d(gZ)
       = \pi(f_\chi)v\, .
  \end{align}
  As the map $C^\infty_c(G) \to C_c^\infty(G)_\chi, f \mapsto f_\chi$, 
  is surjective, for any SF-$\chi$-representation, we derive
  from the Dixmier--Malliavin Theorem 
  \[  E^\infty    = \pi(C^\infty_c(G))E\ {\buildrel \eqref{eq:fchi}\over =}\  \pi(C^\infty_c(G)_\chi)E
    = \pi(\cS(G)_\chi)E.\] 
\end{rem}

\begin{prob} \mlabel{prob:w1}
  Show the existence of a unique minimal SF-globalization of
  $\chi$-HC-modules. This would generalize \cite{BK14} that deals with
  the case of linear groups.

Inspection of the arguments in \cite{BK14} shows that 
their Theorem~4.5 and Corollary 5.6 
remain  true in the $\chi$-context.
The key point is a generalization of their Proposition~11.2. 
\end{prob}

From  \cite[Cor.~2.16]{BK14} we obtain:

\begin{prop}
If $(\pi,\cH)$ is an irreducible unitary
  $\chi$-representations, 
  then $(\pi^\infty, \cH^\infty)$ 
  is an SF globalization of $\cH^{[K]}$.  
\end{prop}

\begin{prob} \mlabel{prob:w2} Suppose that
  $(\pi_{\sigma,\lambda},V_{\sigma,\lambda})$ is a $\chi$-principal series,
  i.e., the representation
  ${\sigma \colon M = Z_K(A) \to \U(W_\sigma)}$ satisfies 
\[ \sigma\res_{Z(G)} = \chi \1.\] 
Are the closures $E^\infty$ of irreducible $(\g,K)$-submodules $E
\subeq V_{\sigma,\lambda}^{[K]}$ in
$V^\infty_{\sigma,\lambda}$ obtained with Theorem~\ref{thm:6.1}
initial SF globalizations? This would lead to a $G$-morphism
\[ E^\infty \to \cH^\infty \subeq \cH \]
which is precisely what we need to transfer the existence of the
holomorphic extension map $\Phi\colon \bE \to E^\infty \subeq V^\infty_{\sigma,\lambda}$ that
we should get from the complex Iwasawa decomposition
(following \cite{KSt04}) to a holomorphic map $\bE \to \cH$.

Given a unitary representation $(U,\cH)$, one may alternatively
work with the minimal
globalization technique developed in \cite{GKKS22} to show that
we have a continuous intertwiner~$E^\omega \into \cH^\omega$. 
\end{prob}

For connected real reductive groups (including connected semisimple
  groups with finite center), the solution of Problems~\ref{prob:w1} and
  \ref{prob:w2} is well-known and stated as the two parts of
  Theorem~11.6.7 in \cite{Wa92}.

Let $V := \cH^{[K]}$ be the irreducible $(\g,K)$-module
underlying our representation
and $v\in V$ a cyclic vector. Then
\[ \cS(G)_{\chi,v} := \{ f \in \cS(G)_\chi \colon \pi(f)v = 0\} \]
is a closed left ideal in $\cS(G)_\chi$ and
\[ V_{\rm min}^\infty := \cS(G)_\chi/\cS(G)_{\chi,v} \]
is an SF-representation of $G$ with $V_{\rm min}^{\infty,[K]} = V$. It clearly satisfies
\[ \pi(\cS(G)_\chi) V_{\rm min} = V_{\rm min}\]
(cf.\ \cite[Rem.~2.19]{BK14}).

The proof of \cite[Prop.~11.2]{BK14} carries over to our context
and implies that, for every smooth $\cS(G)_\chi$-representation
$(\pi_E,E)$, the restriction map
\begin{equation}
  \label{eq:resmap}
  \Hom_G(V_{\rm min}^\infty, E) \to  \Hom_{(\g,K)}(V, E^{[K]})
\end{equation}
is a linear isomorphism. Therefore the remaining problem is to identify
$V_{\rm min}^\infty$, for an irreducible unitary representation
with $\cH^\infty$ and for a subrepresentation of a minimal
principal series with the corresponding space of smooth vectors.

\appendix

\section{An Automatic Continuity Theorem}
\mlabel{app:C}


\begin{defn} A reductive Lie group $G$ with finitely many connected
  components is said to be {\it of Harish--Chandra class} if 
  $\Ad(G) \subeq \Aut(\g_\C)_e$ and the semisimple commutator
  group $(G,G)$ has finite center.

 In \cite[\S 2.1]{Wa88}, a real Lie group $G$ is called
  {\it real reductive} if it is a finite covering
  of an open subgroup of a reductive real algebraic subgroup
  $G_\R \subeq \GL_n(\R)$. 
\end{defn}

If $G$ is semisimple and connected, it is real reductive if
and only if its center $Z(G)$ is finite, but it always
is of Harish--Chandra class; (as required in \cite{Ba87}).
So the Casselman--Wallach Globalization Theorem 
\cite[Thm.~11.6.7]{Wa92} holds for $G$ if $Z(G)$ is finite.
Let $\tau$ be an involution on $G$ and $H$ an open subgroup of its group $G^\tau$ of fixed points. In this appendix we show the following automatic continuity result:

\begin{theorem}\label{thm:AutomaticContinuity}
 Let $G$ be a real reductive group of Harish--Chandra class,
    $\tau$ an involutive automorphisms of $\g$ and $\fh := \g^\tau$. 
Let $(U,\cH)$ be an irreducible unitary representation of $G$, $\cH^{-\omega}$ the space of hyperfunction vectors and $(\cH^{-\omega})^{[\fh]}$ the subspace of $\fh$-finite hyperfunction vectors. Then
$$ (\cH^{-\omega})^{[\fh]} \subseteq \cH^{-\infty}. $$
\end{theorem}

\begin{remark} A special case of the preceding theorem is the automatic 
continuity of $\mathfrak{h}$-fixed hyperfunction vectors (in contrast to 
$\mathfrak{h}$-finite vectors):
\begin{equation}
	(\mathcal{H}^{-\omega})^{\mathfrak{h}} \subseteq 
\mathcal{H}^{-\infty}.\tag{AC}\label{eq:AutContHfixed}
\end{equation}
This statement also follows from the Automatic Continuity Theorem of van 
den Ban--Brylinski{--Delorme}. In fact,
their result is more general than 
\eqref{eq:AutContHfixed}, as they are dealing with linear
  functionals on $\cH^{[K]}$.
In \cite[Thm.]{BD92} it is shown that every 
$H$-fixed linear functional on an irreducible 
$(\mathfrak{g},K)$-module $V$ extends continuously to its 
Casselman--Wallach globalization $V^\infty$. The proof uses in a crucial 
way the main result of \cite{BaD88} which constructs the 
$\mathfrak{g}$-equivariant embedding $\Phi$ of $V$ into $C^\infty(G/H)$ 
from the $\mathfrak{h}$-fixed functional on $V$. This construction can 
also be generalized to the case of $\mathfrak{h}$-finite vectors. 
However, since we already start with $\mathfrak{h}$-finite hyperfunction 
vectors, we can use them to embed $V$ into the space of smooth sections 
of $G\times_HW$ in terms of matrix coefficients and thus avoid a lengthy 
discussion of the extension of \cite{BaD88} to $\mathfrak{h}$-finite vectors. 
We therefore discuss automatic continuity only in the generality we need 
in this paper, i.e.,
  starting with the $\fh$-finite hyperfunction vecrors instead of
  $\fh$-finite linear functionals on the Harish--Chandra module $\cH^{[K]}$.
\end{remark}

For the subspace $(\cH^{-\omega})^\fh\subseteq(\cH^{-\omega})^{[\fh]}$ of $\fh$-fixed hyperfunction vectors, the inclusion $(\cH^{-\omega})^{\fh}\subseteq\cH^{-\infty}$ follows from the Automatic Continuity Theorem of
Brylinski--Delorme~\cite[Thm.~1]{BD92}. The proof we present follows closely their argument, in particular we use van den Ban's results~\cite{Ba88} about the asymptotic expansion of matrix coefficients.\\

We first observe that the Automatic Continuity Theorem~\ref{thm:AutomaticContinuity} is equivalent to the following statement: every $\fh$-equivariant map $\varphi:\cH^\omega\to W$ to a finite-dimensional representation $(\pi_W,W)$
of $\fh$ extends continuously to $\cH^\infty\to W$.
For its proof, we consider the space
$$ C^\infty(G\times_H W) = \{f\in C^\infty(G,W):f(gh)=\pi_W(h)^{-1}f(g)
\mbox{ for all }g\in G,h\in H\}, $$
whose elements represent smooth sections of the vector bundle
$G\times_HW\to G/H$ associated with~$W$.
We now have to  extend the corresponding matrix coefficient map
$$ \Phi: \cH^\omega \to C^\infty(G\times_H W), \quad \Phi(v)(g) = \varphi(U(g^{-1})v) $$
continuously to $\Phi:\cH^\infty\to C^\infty(G\times_H W)$. Then
$$ \ev_e\circ\Phi:\cH^\infty\to W, \quad v\mapsto\Phi(v)(e) $$
is the desired extension of $\varphi$. \\

Since $\cH^\infty$ is a smooth admissible Fr\'echet representation of moderate growth (SAF representation in the sense of \cite{BK14}), the extension of
\[ \Phi:\cH^\omega\to C^\infty(G\times_HW) \]  to $\cH^\infty$ will follow from the functoriality of the Casselman--Wallach Globalization Functor once we have shown that $\Phi$ maps $\cH^\omega$ into a certain SAF representation. For this matter, we introduce a scale of smooth Fr\'echet subrepresentations of $C^\infty(G\times_HW)$ of moderate growth.

Let $\theta$ be a Cartan involution on $G$ that commutes with $\tau$ and fix an embedding $\iota:G\into \GL_n(\R)$ under which $\iota(\theta(g))=(g^\top)^{-1}$. Then
$$ \|g\| = \|\iota(g)\oplus\iota(\theta(g))\| \qquad (g\in G), $$
with the operator norm $\|\cdot\|$ on $M_{2n}(\R)$, defines a norm on $G$ with the following properties:
\begin{align*}
	& \|g^{-1}\|=\|g\|=\|\tau(g)\| && \mbox{for all }g\in G,\\
	& \|xy\|\leq\|x\|\|y\| && \mbox{for all }x,y\in G,\\
	& \|k_1\exp(tX)k_2\| = \|\exp X\|^t && \mbox{for all }k_1,k_2\in K,X\in\fp,t\geq0.
\end{align*}
Fixing a norm $\|\cdot\|_W$ on $W$, we define for $N\in\N$, $D\in \cU(\fg)$ and $f\in C^\infty(G\times_HW)$:
$$ p_{N,D}(f) = \sup_{g\in G}\|g\|^{-N}\|(L_Df)(g)\|_W \in [0,\infty], $$
where $L_D$ is the (right invariant) differential
  operator on $C^\infty(G,W)$ specified by the
  left multiplication action of $G$ on itself. 
For $N\in\N$ we put
$$ \cA_N(G\times_HW) = \{f\in C^\infty(G\times_HW):p_{N,D}(f)<\infty\mbox{ for all }D\in \cU(\fg)\}. $$
Endowed with the seminorms $(p_{N,D})_{D\in \cU(\fg)}$, the space $\cA_N(G\times_HW)$ becomes a Fr\'{e}chet space. The following result is proven along the same lines as \cite[Lemme 1]{BD92}:

\begin{lemma}
	The left regular representation of $G$ on $\cA_N(G\times_HW)$ is a smooth Fr\'echet representation of moderate growth.
\end{lemma}

In order to extend $\Phi:\cH^\omega\to C^\infty(G\times_HW)$ to
$\cH^\infty$, we show that $\Phi$ maps the subspace $\cH^{[K]}$ of $K$-finite vectors into one of the moderate growth representations $\cA_N(G\times_HW)$. This is done using van den Ban's results about the asymptotic behavior of matrix coefficients in \cite{Ba87}.

\begin{proposition}
	There exists $N\in\N$ such that $\Phi(\cH^{[K]})\subseteq\cA_N(G\times_HW)$.
\end{proposition}

\begin{proof}  
Let $\fa\subseteq\fp\cap\fq$ be a maximal abelian subspace and $A=\exp(\fa)$. For any choice $\Delta^+\subseteq\Delta(\g^{\tau\theta},\fa)$ of positive roots where $\g^{\tau\theta}=\fk\cap\fh+\fp\cap\fq$ we consider the negative Weyl chamber
$$ A^- = A^-(\Delta^+) = \{a\in A:a^\alpha<1\mbox{ for all }\alpha\in\Delta^+\}. $$
Then the Cartan decomposition $G=KAH$ holds (see e.g. \cite[Cor.~1.4]{Ba87}). Now let $v\in\cH^{[K]}$, then $v$ is contained in a finite-dimensional $K$-invariant subspace $E\subseteq\cH$. Consider the function
$$ F:G\to E^*\otimes W\simeq\Hom_\C(E,W), \quad F(g)=\varphi\circ U(g)^{-1}|_E. $$
We note that $F(g)(v)=\Phi(v)(g)$, so it suffices to
estimate $F(g)$. The function $F$ is $(\mu_1,\mu_2)$-spherical in the sense of \cite[page 230]{Ba87} with $\mu_1(k)=(U(k^{-1})\res_E)^*$ and $\mu_2=\pi_W$:
$$ F(kgh) = \pi_W(h)^{-1} F(g) U(k^{-1})\res_E \quad \mbox{ for } \quad k\in K,g\in G,h\in H. $$
Moreover, since $F$ is a matrix coefficient of the irreducible representation $(U,\cH)$, it is annihilated by a cofinite ideal $I$ of the center of
$\cU(\fg)$ (see \cite[pp.~230--231]{Ba87}). We may therefore apply the results of \cite{Ba87}.

Applying \cite[Thm.~6.1]{Ba87} yields a character $\omega$ of $A$ and $C>0$ with the property that
$$ \|F(a)\| \leq C\omega(a) \qquad \mbox{for all }a\in\overline{A^-}. $$
Here $\|\cdot\|$ denotes the tensor product norm on
$E^*\otimes W$ induced by a $K$-invariant norm on $E$ and the norm $\|\cdot\|_W$ on $W$. We note that the character $\omega$ only depends on the cofinite ideal $I$ annihilating the representation
$U$ and not on the vector $v\in\cH^{[K]}$. Bounding the character $\omega$ by a power of the norm $\|\cdot\|$ on $A$ (see \cite[Lemma 2.A.2.3]{Wa88}) and applying this argument to all of the finitely many choices of positive systems $\Delta^+\subseteq\Delta(\g^{\tau\theta},\fa)$ and hence to all Weyl chambers in $A$,
shows that there
exist $N\in\N$ and $C>0$ such that
$$ \|F(a)\| \leq C\|a\|^N \qquad \mbox{for all }a\in A. $$

To extend this estimate to an estimate on all of $G$ we note that the finite-dimensional representation $(\mu_2,W)$ of $H$ is of moderate growth and hence there exist $C' > 0$ and $M\in \N$ such that
\[  \|\mu_2(h)w\|_W \leq C'\|h\|^M\|w\|_W
\quad \mbox{ for } \quad h\in H,w\in W.\] 
Together with the unitarity of $\mu_1$ this implies
$$ \|F(kah)\| = \|\pi_W(h)^{-1} F(a) U(k^{-1})\| 
\leq C'\|h\|^M\|F(a)\| \leq CC'\|h\|^M\|a\|^N. $$
But the properties of the norm imply that for $a\in A$ and $h\in H$:
$$ \|ah\| = \|\tau(ah)\| = \|a^{-1}h\| $$
and hence
$$ \|a\|^2 = \|a^2\| = \|ahh^{-1}a\| \leq \|ah\|\|h^{-1}a\| = \|ah\|\|a^{-1}h\| = \|ah\|^2. $$
This also implies
$$ \|h\| = \|a^{-1}ah\| \leq \|a^{-1}\|\|ah\| = \|a\|\|ah\| \leq \|ah\|^2, $$
so we obtain
$$ \|F(kah)\| \leq CC'\|ah\|^{N+2M} = CC'\|kah\|^{N+2M} \qquad (k\in K,a\in A,h\in H). $$
Since $F(g)(v)=\Phi(v)(g)$,
this shows that $\Phi(\cH^{[K]})\subseteq\cA_{N+2M}(G\times_HW)$.
\end{proof}

\begin{proof}[Proof of Theorem~\ref{thm:AutomaticContinuity}]
  Since $\Phi:\cH^{[K]}\to\cA_N(G\times_HW)$ and
  $\cA_N(G\times_HW)$ is a smooth Fr\'echet representation of moderate growth, the closure $\overline{\Phi(\cH^{[K]})}\subseteq\cA_N(G\times_HW)$ is a smooth admissible Fr\'echet representation of moderate growth. By the Casselman--Wallach
  Globalization Theorem (\cite[Thm.~11.6.7]{Wa92}),
   the map $\Phi:\cH^{[K]}\to\overline{\Phi(\cH^{[K]})}$ extends uniquely to a continuous intertwining operator $\Phi:\cH^\infty\to\overline{\Phi(\cH^{[K]})}\subseteq\cA_N(G\times_HW)$. Now the map
	$$ \ev_e\circ\Phi:\cH^\infty\to W, \quad v\mapsto\Phi(v)(e) $$
	is the desired extension of $\cH^\omega\to W$.
\end{proof}

\begin{remark}
  We note that in Sections 1--5 of \cite{Ba87} the assumption on $G$ to be of Harish-Chandra's class is not used, so the results about the asymptotic expansion of matrix coefficients also hold for the universal covering $G$ of a Hermitian Lie group. Since the center $Z(G)$ of $G$ acts by a unitary character $\chi$ in an irreducible unitary representation $(U,\cH)$, the arguments in the proof of Theorem~\ref{thm:AutomaticContinuity} can be adapted to show that $\cH^\omega$
  embeds for some $N\in\N$ into
  \begin{align*}
    & \cA_{\chi,N}(G\times_HW)  \\
    = &\left\{f\in C^\infty(G\times_HW):\begin{array}{ll}f(cg)=\chi(c)f(g)\mbox{ for all }g\in G,c\in Z(G),\\p_{N,D}(f)<\infty\mbox{ for all }D\in \cU(\fg)\end{array}\right\}.
  \end{align*}
 Here, $p_{N,D}$ is defined using a norm function on a linear quotient of $G$. If a version of the Casselman--Wallach Globalization Theorem holds for the group $G$, the embedding $\cH^\omega\hookrightarrow\cA_{\chi,N}(G\times_HW)$ extends to $\cH^\infty\hookrightarrow\cA_{\chi,N}(G\times_HW)$.
\end{remark}

\section{Wedge regions in non-compactly causal spaces}
\label{app:b}

In this appendix we put some of the results from \cite{MNO23b} into the
context in which they are used in the present paper.

As above, $G$ denotes a connected simple Lie group, $h \in \g$ is an Euler
element, $\tau =  \theta \tau_h$ for a Cartan involution $\theta$
satisfying $\theta(h) = - h$ and $M = G/H$ is a corresponding non-compactly
causal symmetric space, where the causal structure is specified by a
maximal $\Ad(H)$-invariant closed convex cone $C \subeq \fq$ satisfying
$h \in C^\circ$ (cf.~\cite[Thm.~4.21]{MNO23a}).

First we consider the ``minimal'' space associated to the triple
$(\g,\tau,C)$. It is obtained as
\[ M_{\rm min} := G_{\ad}/H_{\ad},\]
where
\[ G_{\rm ad} := \Ad(G) = \Inn(\g) \quad \mbox{ and }\quad
  H_{\rm ad} := K_{\rm ad}^h \exp(\fh_\fp) \subeq G_{\ad}^\tau\]
(see \cite[Rem.~4.20(b)]{MNO23a} for more details).
In this space the positivity domain $W_{M_{\rm min}}^+(h)$ is connected
by  \cite[Thm.~7.1]{MNO23b}. 
Further, \cite[Thm.~8.2, Prop.~8.3]{MNO23b} imply that
\[ W_{M_{\rm min}}^+(h) = G^h_e.\Exp(\Omega_{\fq_\fp}).\]
By \cite[Rem.~4.20(a)]{MNO23a}, we have
$H = H_K \exp(\fh_\fp)$ with $H_K \subeq K^h$, so that
$\Ad(H) \subeq H_{\rm ad}$. Therefore
\[  q_M \colon M \to M_{\rm min}, \quad gH \mapsto \Ad(g) H_{\rm ad} \in M_{\rm min} \]
defines a covering of causal symmetric spaces. The stabilizer in $G$
of the base point in $M_{\rm min}$ is the subgroup
\[ H^\sharp := \Ad^{-1}(H_{\rm ad})
  = K^h \exp(\fh_\fp) \]
because $Z(G) = \ker(\Ad) \subeq K^h$.
Note that $H^\sharp$ need not be contained in $G^\tau$ because
$\tau$ may act non-trivially on $K^h$. A typical example
if $G= \tilde\SL_2(\R)$ with $K \cong \R$ and $\tau(k) = k^{-1}$ for $k \in K$.
So we may consider $M_{\rm min}$ as the homogeneous $G$-space 
\[ M_{\rm min} \cong G/H^\sharp.\]

As $q_M$ is a $G$-equivariant covering of causal manifolds,
\[ W_M^+(h) = q_M^{-1}(W_{M_{\rm min}}^+(h))
  = q_M^{-1}(G^h_e.\Exp(\Omega_{\fq_\fp})) \]
and the inverse image under the map $q \colon G \to G/H = M$ is
\begin{align*}
 q^{-1}(W_M^+(h))
&  = G^h_e \exp(\Omega_{\fq_\fp}) H^\sharp
  = G^h_e \exp(\Omega_{\fq_\fp}) K^h \exp(\fh_\fp)\\
  &  = G^h_e  K^h\exp(\Omega_{\fq_\fp}) \exp(\fh_\fp)
= G^h\exp(\Omega_{\fq_\fp}) \exp(\fh_\fp).
\end{align*}

Next we recall from \cite[Prop.~8.3]{MNO23b} that the map
\begin{equation}
  \label{eq:ad-fiber-diffeo}
 G^h_e \times_{K^h_e} \Omega_{\fq_\fk} \to W^+_{M_{\rm ad}}(h), \quad
 [g,x] \mapsto g\exp(x)H_{\ad}
\end{equation}
is a diffeomorphism. Therefore $W^+_{M_{\rm ad}}(h)$ 
is an affine bundle over the Riemannian symmetric space $G^h_e/K^h_e$,
hence contractible and therefore simply connected.
So its inverse image $W_M^+(h)$ in $M$ is a union of open connected
components, all of which are mapped diffeomorphically onto
$W^+_{M_{\rm ad}}(h)$ by~$q_M$.
It follows in particular that the diffeomorphism \eqref{eq:ad-fiber-diffeo}
lifts to a diffeomorphism 
\begin{equation}
  \label{eq:fiber-diffeo}
 G^h_e \times_{K^h_e} \Omega_{\fq_\fk} \to W = W^+_M(h)_{eH}, \quad
 [g,x] \mapsto g\exp(x)H.
\end{equation}

\end{document}